\documentclass[11pt]{amsart}
\usepackage{tikz}
\usetikzlibrary{arrows.meta}
\usepackage{esint}
\usepackage{verbatim}
\usepackage{graphicx}
\usepackage{subfigure}
\usepackage{amsfonts}
\usepackage{amssymb}
\usepackage{amsmath}
\usepackage{amsthm}
\usepackage{latexsym}
\usepackage{amscd}
\usepackage{mathrsfs}
\usepackage{verbatim}
\usepackage{mathrsfs}
\usepackage{enumitem}
\usepackage{braket}
\usepackage{float}
\usepackage[margin=1.3in]{geometry}
\usepackage{comment}
\usepackage{bm}
\usepackage{xcolor}
\usepackage{imakeidx}
\makeindex[title=Index of Notation]

\usepackage{hyperref}
\hypersetup{
	colorlinks=true,
	linkcolor=black,
	urlcolor=black,
	citecolor=black
}
\usepackage{caption}
\usepackage{array}

\usepackage{amssymb}
\usepackage{pdfpages}
\usepackage{verbatim}

\usepackage{appendix}
\usepackage{xspace}

\newtheorem{theorem}{Theorem}[section]

\newtheorem{lemma}[theorem]{Lemma}
\newtheorem{proposition}[theorem]{Proposition}
\newtheorem{corollary}[theorem]{Corollary}
\theoremstyle{definition}
\newtheorem{definition}[theorem]{Definition}

\theoremstyle{remark}

 \makeatletter
 \def\ExtendSymbol#1#2#3#4#5{\ext@arrow 0099{\arrowfill@#1#2#3}{#4}{#5}}
 
 \makeatother

 \makeatletter
 \def\ExtendSymbol#1#2#3#4#5{\ext@arrow 0099{\arrowfill@#1#2#3}{#4}{#5}}
 \newcommand\longright[2][]{\ExtendSymbol{-}{-}{\rightarrow}{#1}{#2}}
 \makeatother

\makeatletter

\newcommand{\Rmnum}[1]{\expandafter\@slowromancap\romannumeral #1@}
\makeatother

\newcommand{\ID}{\mathbb{D}}

\newcommand{\IN}{\mathbb{N}}
\newcommand{\IR}{\mathbb{R}}

\newcommand{\CK}{\mathcal{K}}

\numberwithin{equation}{section}

\makeatletter
\@namedef{subjclassname@2020}{\textup{2020} Mathematics Subject Classification}
\makeatother


\begin{document}
	\title{The Calder\'on-Zygmund inequalities on evolving Riemannian manifolds}
	
	\author{Yongheng Han}
	\address{School of Mathematical Science, University of Science and Technology of China,  Hefei City,  Anhui Province 230026}
	
	\email{hyh2804@mail.ustc.edu.cn}

	\author{Bing Wang}
	\address{ Institute of Geometry and Physics, and School of Mathematical Sciences, University of Science and Technology of China, Hefei 230026, China; Hefei National Laboratory, Hefei 230088, China}
	
	\email{topspin@ustc.edu.cn}
    
	\date{\today}

	
	\begin{abstract}
         The Calder\'on–Zygmund inequality is a cornerstone of harmonic analysis and partial differential equations.  In this article, we establish various Calder\'on-Zygmund inequalities on evolving Riemannian manifolds with bounded curvature.  We also provide concrete applications of such inequalities.  
	\end{abstract}	
	
	\maketitle

\tableofcontents

\section{Introduction}
\label{sec1}

This article is concerned with Calder\'on-Zygmund (CZ) inequalities on Riemannian manifolds and Ricci flows. The history of Calder\'on-Zygmund inequalities originates from the 1950s with the seminal work of Alberto Calder\'on and Antoni Zygmund, who developed the theory of singular integrals and their associated inequalities (cf. \cite{CZ52}).

Let $T$ be a linear operator on $\mathbb{R}^n$ and it can be written as
\begin{equation*}
    T f(x) := \lim\limits_{\varepsilon\to 0} \int_{\mathbb{R}^n}\chi_{|x-y|>\varepsilon} K(x,y) f(y) dy
\end{equation*}
for some measurable function $K(x,y)$.
We call $T$ a CZ operator if $T$ is a bounded operator on $L^2(\mathbb{R}^n)$ and  $K(x,y)$ satisfies the following conditions for some uniform constants $C=C(n)$ and $\delta=\delta(n)$. 
\begin{itemize}
\item $ |K(x,y)| \leq C |x-y|^{-n}$.
\item $\displaystyle |K(x,y) - K(x',y)| + |K(y,x) - K(y,x')| 
\leq C|x-y|^{-n} \cdot \left( \frac{|x-x'|}{|x-y|}\right)^{\delta}$. 
\end{itemize}
Calder\'on and Zygmund showed that for all $p \in (1, \infty)$ and $f \in L^p(\mathbb{R}^n)$, one has
\begin{equation}
    \|Tf\|_{L^p} \leq C(n,p) \cdot \|f\|_{L^p}.
\label{eqn:HB13_2}    
\end{equation}
The CZ inequality (\ref{eqn:HB13_2}) plays a central role in the theory of partial differential equations, particularly in the regularity analysis of elliptic equations.
Let $G$ be the Newton potential. Then a solution of the Poisson equation $-\Delta u=f$ can be represented as: 
\begin{equation*}
    u = (-\Delta)^{-1} f = G \ast f. 
\end{equation*}
Thus, we have
\begin{equation*}
    D^2 u = D^2 (-\Delta)^{-1} f = (D^2 G) \ast f.
\end{equation*}
Since the convolution with $D^2 G$ is a CZ operator, we can apply inequality (\ref{eqn:HB13_2}) to obtain:
\begin{equation}
\label{eq15}
    \|D^2 u\|_{L^p} \leq C(\|u\|_{L^p} + \|f\|_{L^p}).
\end{equation}

It is natural to ask whether can we generalize such an important inequality to  Riemannian manifolds $(M,g)$. 
In other words, whether the following inequality holds:
\begin{equation}
\label{eq0.1}
    \|\mathrm{Hess}(u)\|_{L^p} \leq C(\|u\|_{L^p} + \|\Delta u\|_{L^p}).
\end{equation}
Before we do this generalization, let us recall how the CZ inequality is proved in Euclidean space. 
The classical proof of the CZ inequality proceeds as follows: 
\begin{itemize}
\item Prove (\ref{eqn:HB13_2}) for $p=2$, by integration by parts.
\item Prove a weak $L^1$-estimate, by CZ deomoposition.
\item Prove (\ref{eqn:HB13_2}) for $1<p<2$, by Marcinkiewicz interpolation. 
\item Prove (\ref{eqn:HB13_2}) for $2<p<\infty$, by duality. 
\end{itemize}
Note that both the CZ decomposition and the duality depends heavily on the Euclidean structure. 
Thus, the CZ inequality (\ref{eq0.1}) may fail on some wierd manifolds.
In fact, Marini and Veronelli \cite{MV21} constructed a complete non-compact $n$-dimensional Riemannian manifold of positive sectional curvature which does not support any CZ inequality for $p > n$ (see also \cite{Li19}). It is worth mentioning that their example even shows that the duality fails, since the CZ inequality is still valid there for $1 < p < 2$, only fails for $p > n$.

 What are the natural geometric conditions for (\ref{eq0.1}) to hold?
 In~\cite{GP15}, G\"uneysu and Pigola obtained Calder\'on-Zygmund inequalities on manifolds with bounded Ricci curvature and injectivity radius. In a survey article, Pigola claimed (cf. Theorem 10.1 of~\cite{Pig20}) that \eqref{eq0.1} holds for $p > \max(2, \frac{n}{2})$ when $\|Rm\|$ is bounded.
 Cao, Cheng, and Thalmaier~\cite{CCT21} showed that the inequality~\eqref{eq0.1} holds if $1 < p \leq 2$ when $M$ has a lower Ricci bound, or if $2 < p < \infty$ and $\|\mathrm{Rm}\|+\|\nabla \mathrm{Rc}\|$ are in Kato class. Recently, Cheng, Thalmaier, and Wang \cite{CTW25} provided a new proof via the Riesz transform. 
 
 In the collapsing case, the standard approach of using local coordinates to invoke classical $W_p^{2}$
estimates is no longer viable, which constitutes a fundamental limitation of the classical approach. 

Caffarelli and Peral \cite{CP98} used a method of approximation to obtain $W_p^{1}$ estimates for elliptic equations in divergence form. Wang \cite{Wan03} gave a new proof of the classical CZ estimates (the inequality \eqref{eq15}). He used the method of approximation, the Vitali covering lemma, and the Hardy–Littlewood maximal function. 
In \cite{LW06}, Li and Wang generalized Wang's methods and gave a new proof for the estimates of CZ type singular integrals. In this paper, we adapt Li and Wang's methods and generalize them to Riemannian manifolds. \\

We put the study of CZ inequality in a more general setting. Previously, the background for the CZ inequality are complete Riemannian manifolds. Now we replace them by evolving Riemannian manifolds $\{(M, g(t)), -1 \leq t \leq 0\}$.    If $g(t) \equiv g$, this background is called static and we return to the classical study.  However, we can also consider more general $g(t)$. Among them, the Ricci flow is a very important case.    In this paper, we shall only study the evolving Riemannian manifolds $\{(M, g(t)), -1 \leq t \leq 0\}$ such that one of the following conditions are satisfied. 
\begin{itemize}
    \item  $g(t) \equiv g$;
    \item  $g(t)$ evolves by the Ricci flow.
\end{itemize}
In both cases, we assume 
\begin{align}
    \sup_{x \in M, t \in [-1, 0]} |Rm|(x,t) \leq \Lambda_0. 
\label{eqn:HA02_1}    
\end{align}
For simplicity of notation, we denote
\begin{align}
&\mathcal{M} :=M \times [-1, 0], \quad \mathcal{M}' :=M \times [-\frac14, 0],\label{eqn:HA11_0}\\ 
&(\mathcal{M} \times \mathcal{M})^+ :=\left\{ (x,t;y,s)|x, y\in M, \; -1 \leq s <t \leq 0 \right\}.    
\label{eqn:HA11_0A}
\end{align}
Fix a point $q \in M$ and denote
 \begin{align}
     \mathcal{Q}:=B_{\frac14}(q,0) \times [-\frac{1}{16}, 0], 
     \quad 
     \mathcal{Q}':=B_{\frac18}(q,0) \times [-\frac{1}{64}, 0].
 \label{eqn:HA24_7}    
 \end{align}
We use $\dot{}$ to denote the time derivative.

We study the CZ inequality in a more general background. Let $\mathcal{E},\mathcal{F}$ be smooth vector bundles on $M$. 
Let
$\mathcal{K}: (\mathcal{M} \times \mathcal{M})^+ \to \mathcal{E} \boxtimes \mathcal{F}^*$ be a kernel section and the convolution operator $\mathcal{T} :=\mathcal{K}*$.  
The operator $\mathcal{T}$ is called a $(C_1,C_2)$-CZ operator if the kernel satisfies an appropriate $C_1$-bound and $\|\mathcal{T}\|_{2,2}$ is uniformly bounded by $C_2$ (see Definition \ref{def2.1} for details).  Our main result is the following theorem.

\begin{theorem}[\textbf{Main result}]
\label{thm:HB09_1}
  Suppose $\{(M,g(t)), -1 \leq t \leq 0\}$ is an evolving manifold satisfying (\ref{eqn:HA02_1}).
  Suppose $\mathcal{T}$ is a $(C_1,C_2)$-Calder\'on-Zygmund integral operator. 

  For each $p \in [2, \infty)$, there is a positive constant $C=C(n,p,\Lambda_0,C_1,C_2)$ such that
      \begin{align}
      \|\mathcal{T}f\|_{L^p(\mathcal{M}')}\leq C\|f\|_{L^p(\mathcal{M})} 
      \label{eqn:HB11_1}    
       \end{align}
   for every $f \in C^{\infty}(\mathcal{M}, \mathcal{F})$.     In particular, if $f$ is supported in $\mathcal{Q}'$, then 
    \begin{align}
      \|\mathcal{T}f\|_{L^p(\mathcal{M})}\leq C\|f\|_{L^p(\mathcal{Q}')}. 
      \label{eqn:HB13_3}    
    \end{align}
\end{theorem}

Let $\mathcal{E}=\mathcal{F}=M\times \mathbb{R}$ and $\mathcal{T}$ be the convolution with $\nabla_x^2 H(x,t;y,s)$, the Hessian of the heat kernel. Then we obtain the following Corollary.

\begin{corollary}
     \label{cly:HA02_2}
      Let $\{(M,g(t)), -1 \leq t \leq 0\}$ be an evolving manifold that satisfies (\ref{eqn:HA02_1}).
      Let $u$ be a smooth function such that $(\partial_t -\Delta)u=f$. 
      Then for any $p\in [2,\infty)$, there exists a positive constant $C=C(n,p,\Lambda_0)$ 
      such that 
    \begin{align}
    \label{eqn:HB13_1}
            \begin{split}
                \|\dot{u}\|_{L^p(\mathcal{M}')}+\|\mathrm{Hess}(u)\|_{L^p(\mathcal{M}')}
                \leq C \left\{ \|u\|_{L^p(\mathcal{M})}+\|f\|_{L^p(\mathcal{M})} \right\}
            \end{split}
	\end{align}
	for any  $u\in C^\infty_c(\mathcal{M})$.
 	In particular, if $g(t) \equiv g$ satisfies (\ref{eqn:HA02_1}), and $p \in [2, \infty)$, then 
		\begin{align}
                \label{eqn:HA02_3}
			\|\mathrm{Hess}(u)\|_{L^p(M)}\leq C \left\{\|u\|_{L^p(M)}+\|\Delta u\|_{L^p(M)} \right\}
		\end{align}
		for any  $u\in C^\infty_c(M)$.
\end{corollary}

Note that inequality (\ref{eqn:HA02_3}) in Corollary~\ref{cly:HA02_2} establishes the Calder\'on-Zygmund for smooth functions on complete Riemannian manifolds with only bounded sectional curvature (\ref{eqn:HA02_1}).
The key new ingredient is that we have neither the assumption of non-collapsing nor the requirement of $\|\nabla Rc\|$. 

For the purpose of investigating the Ricci flow,  we apply Theorem~\ref{thm:HB09_1} to the situation that $\mathcal{E}=\mathcal{F}=Sym^2(T^*M)$.  Then (\ref{eqn:HB13_1}) holds if $g(t)$ solves the Ricci flow solution and we regard $u$ as a smooth $(0,2)$-tensor field. More precise statement and information can be found in Theorem~\ref{thm:HA02_3}.  We also have the following estimates, which look different from (\ref{eqn:HB13_1}).

Let $\Psi$ be the heat kernel of the Lichnerowicz Laplacian for the $(0,2)$-tensor field. 
Let $\mathcal{T}=\mathcal{K}*$ and $\mathcal{K}=\nabla_x \nabla_y \Psi(x,t;y,s)$.
In this case, the application of (\ref{thm:HB09_1}) implies the following Corollary. 

\begin{corollary}
    \label{cly:HB09_2}
    Suppose $\{(M,g(t)), -1 \leq t \leq 0\}$ is an evolving manifold that satisfies (\ref{eqn:HA02_1}). 
    Suppose $u \in C^{\infty}(\mathrm{Sym}^2(T^*M))$ is a solution of $(\partial_t -\Delta_L) u=\nabla^*f$ 
    for some $f\in C^{\infty} (\mathrm{Sym}^2(T^*M) \otimes T^*M)$. Here, $\Delta_L$ is the Lichnerowicz Laplacian and $\nabla^*$ is the formal adjoint operator of $\nabla$. Then for each $p \in [2, \infty)$,  there exists a positive constant 
        $C=C(n,p,\Lambda_0)$ such that  
		\begin{align}
            \label{eqn:HB09_3}
			 \|\nabla u\|_{L^p(\mathcal{M}')}
            \leq C \left\{ \|u\|_{L^p(\mathcal{M})}+\|f\|_{L^p(\mathcal{M})} \right\}. 
		\end{align}
    In particular, if $f$ is supported in $\mathcal{Q}'$, then 
    \begin{align}
      \|\nabla u\|_{L^p(\mathcal{M})}
            \leq C  \|f\|_{L^p(\mathcal{Q}')}.  
      \label{eqn:HB13_4}    
    \end{align}
 \end{corollary}

 Corollary~\ref{cly:HB09_2} plays an important role in the study of Ricci flow on a complete manifold with bounded $\|Rm\|$.  Using (\ref{eqn:HB13_4}), one can uniformize the proof of the short-time existence, uniqueness, and continuous dependence of the Ricci flow.  
 Interested readers are referred to  Cai-Wang~\cite{CaiWang} for more details. \\

We now briefly discuss the key points of the proof of our results. 

\begin{proof}[Outline of proof of Theorem~\ref{thm:HB09_1}:]
 Our starting points are the Gaussian kernel estimate of $\mathcal{K}$ 
 and the $L_2$-estimate of $\mathcal{T}$.  We apply the stability of maximal function and transfer the bound of $\|\mathcal{T}f\|_{L^p}$ to $\|\mathbf{M}(\mathcal{T}f)\|_{L^p}$. 
 Since $\mathbf{M}(\mathcal{T}f)$ behaves much more stable than $\mathcal{T} f$, we are able to show that $|\{\mathbf{M}(\mathcal{T}f)>\lambda\}|$ decreases very fast when $\lambda$ goes to $\infty$.   Combined with the $L^2$-estimate, this decreasing speed is enough for the desired $L^p$-estimate if $p>2$.   The stability of maximal function $\mathbf{M}(\mathcal{T}f)$ can be deduced from the proper kernel estimate of $\mathcal{K}$.  In our proof, the locally doubling property, which is guaranteed by the Ricci lower bound, plays an important role in the maximal function argument.  The decreasing speed estimate is firstly proved locally and then upgraded to a global one via covering argument. 
\end{proof} 

 \begin{proof}[Outline of proof of Corollary~\ref{cly:HA02_2}:]
 The key point is to derive better heat kernel estimate under the curvature assumption (\ref{eqn:HA02_1}). 
 The classical Li-Yau estimate provides an initial $C^0$-bound of the heat kernel $H$.  In order to apply Theorem~\ref{thm:HB09_1}, we need local $C^{2+\alpha}$-estimate of $H$. 
 The direct method of maximum principle requires the  curvature derivatives condition, which is not available here.  Alternatively, 
 we pullback the metric locally to a ball on the tangent space of a base point. 
 This pullback metric is automatically non-collapsing and has a $W_p^2$-harmonic radius uniformly bounded from below.  By Sobolev embedding, the metric is uniformly $C^{1+\alpha}$ and the heat solution in this coordinate has a uniform $C^{3+\alpha}$-estimate, which implies the desired $C^{2+\alpha}$-estimate of the heat kernel. 
 \end{proof}

 \begin{proof}[Outline of proof of Corollary~\ref{cly:HB09_2}:]
 For the sake of Theorem~\ref{thm:HB09_1}, we need heat kernel estimate and an $L^2$-estimate.  The method of obtaining the proper heat kernel estimate is similar to the one in the proof of Corollary~\ref{cly:HA02_2}: we apply the gradient estimate for tensor-valued heat solutions in Euclidean space with pullback metric. 
 The $L^2$-estimate here is more delicate. However, it is still the application of integration by parts and Bochner formula. 
 \end{proof}

\textbf{Structure of this paper:}

 In section~\ref{sec2}, we collect some facts and fix the notation for the rest of the paper.
 In section~\ref{sec3}, we study the regularity estimate for heat solutions in a model domain.
 In section~\ref{sec4}, we use an exponential map to pull back the metric on Riemannian manifold to Euclidean balls. Thus, we can localize the study of the heat kernel in a Euclidean ball and derive the desired gradient estimates for the heat kernels.
 In section~\ref{sec5}, we use the heat kernel estimates to study the stability of Hardy-Littlewood maximal functions.  In section~\ref{sec6}, we apply the stability of maximal functions to estimate the decreasing speed of the measure of level set, which is sufficient to obtain the CZ-inequality (\ref{eqn:HB13_2}). 
 In section~\ref{sec7}, we discuss various situations where the CZ-inequality can be applied. \\

\textbf{List of important notation:}

\begin{itemize}
\item We use $\dot{}$ to denote the time derivative. For example, $\dot{u}=\partial_t u$.
\item $\mathcal{K}$: Calder\'on-Zygmund kernel (cf. Definition~\ref{def2.1}). 
\item $\mathcal{M} := M \times [-1, 0], \; \mathcal{M}' :=M \times [-\frac14, 0]$. 
\item $\mathbf{M}(f)$: local version maximal function (cf. (\ref{eqn:HC08_1})). 
\item $P_r(x,t)$: parabolic ball (cf. (\ref{eqn:HA24_11})).
\item $\mathcal{Q}:=B_{\frac14}(q,0) \times [-\frac{1}{16}, 0]$, $\mathcal{Q}':=B_{\frac18}(q,0) \times [-\frac{1}{64}, 0]$. 
\item $\mathcal{T}$: Calder\'on-Zygmund operator (cf. Definition~\ref{def2.1}). 
\item $dyds :=d\mu(y,s) ds$. 
\item $\eta$: cutoff function with time variable.
\item $\xi=\xi(n)<\frac{1}{100n\pi}$ is a small positive constant (cf. Definition~\ref{dfn:HC07_2}). \\
\end{itemize}

\textbf{Acknowledgments}

    The authors are supported by the Stable Support Project for Youth Teams in
the Basic Research Field,  Chinese Academy of Sciences (YSBR-001), the National Natural Science Foundation of China (NSFC-12431003) and a research fund from Hefei National Laboratory.

	\section{Preliminary}
    \label{sec2}
    In this section, we collect some well-known results in Riemannian geometry and harmonic analysis  which will be used later.

    Denote by $B_r(x)$ the geodesic ball of the center $x\in M$ with radius $r>0$ and by $|B_r(x)|$ or $V_x(r)$  the volume $\mu(B_r(x))$.

    \begin{theorem}[Bishop-Gromov]\label{thm1.1}
        Suppose $(M, g)$ is a complete n-dimensional Riemannian manifold with $Rc\geq -(n-1) \Lambda_0$,  and $q\in M$ is an arbitrary point. Then for any $0<r_1<r_2 <\infty$, we have
        \begin{align}
            \frac{|B_{r_2}(q)|}{|B_{r_1}(q)|}\leq \frac{|B_{r_2}^{\Lambda_0}|}{|B_{r_1}^{\Lambda_0}|}
        \label{eqn:HA24_8}    
        \end{align}
        where $B_r^{\Lambda_0}$ is a geodesic ball of radius $r$ is the space form of the sectional curvature $-\Lambda_0$. 
        In particular, we have 
        \begin{align}
            \frac{|B_{r_2}(q)|}{|B_{r_1}(q)|}\leq  C(n, \Lambda_0) \cdot \left( \frac{r_2}{r_1}\right)^n
        \label{eqn:HA24_9}    
        \end{align}
        if $0<r_1<r_2 \leq 1$. 
    \end{theorem}

Using volume comparison, we have the following covering lemma.  
\begin{lemma}
\label{lm7.1}
Suppose $(M, g)$ is a complete n-dimensional Riemannian manifold with $Rc\geq -(n-1) \Lambda_0$. Then there is a sequence of points 
    $\{x_i\}\subset M$ and a constant $C=C(n,\Lambda_0)$ such that 
    \begin{itemize}
        \item $B_{\frac15}(x_i)\cap B_{\frac15}(x_j)=\emptyset$ for all $i,j\in \IN^{+}$ with  $i\neq j$;
        \item $\cup_{i=1}^\infty B_{1}(x_i)=M$;
        \item $\#\{j|m-1\leq d(x_i,x_j)\leq m\}\leq e^{Cm}$, for any $i\in \IN^{+}$.
    \end{itemize}
    \end{lemma}
    
\begin{proof}
We only prove the last property. Define
    \begin{align*}
        A^i_m := \{j|m-1\leq  d(x_j,x_i)\leq m\}.
    \end{align*}
 By volume comparison theorem, we have 
 \begin{equation}
 \label{eq7.1}
     \begin{split}
        \sum_{j\in A^i_m}|B_{\frac12}(x_j)|\leq |B_{m+1}(x_i)|\leq e^{Cm}|B_1(x_i)|.
     \end{split}
 \end{equation}
 On the other hand,
 \begin{align*}
     |B_{\frac12}(x_j)|\geq e^{-Cm}|B_{m+1}(x_j)|\geq e^{-Cm}|B_{1}(x_i)|.
 \end{align*}
 Plugging it into~\eqref{eq7.1} yields
 \begin{equation*}
     \begin{split}
        \sum_{j\in A^i_m}e^{-Cm}|B_{1}(x_i)|\leq |B_{m+1}(x_i)|\leq e^{Cm}|B_1(x_i)|.
     \end{split}
 \end{equation*}
 which implies $\#A^i_m\leq e^{Cm}$. 
\end{proof}

  In this subsection, we consider the metric measure space $\mathcal{M} :=M\times [-1, 0]$. 
  Given $0<r\leq 1$ and $(x,t)\in M\times [-1,0]$, denote 
    \begin{align}
        P_r(x,t) :=B_r(x)\times \{[t-r^2, t+r^2] \cap [-1, 0]\}. 
    \label{eqn:HA24_11}    
    \end{align}
    \begin{lemma}[Vitali covering]
        Let $I$ be a family of parabolic balls. 
        Then there exists a disjoint subcollection $\{P_{r_k}(X_k)\}$ such that
        \begin{align}
            \cup_{\lambda \in I}P_{\lambda} \subset \cup_k P_{5r_k}(X_k).
        \label{eqn:HA24_10}    
        \end{align}
    \end{lemma}

    We define  local version centered Hardy–Littlewood maximal functions by
	\begin{align}
		\mathbf{M}(f)(x,t):=\sup_{0<r<\frac12}\frac{1}{|P_r(x,t)|}\int_{P_r(x,t)}|f| dzds=\sup_{0<r<\frac12}\fint_{P_r(x,t)}|f|dzds.
    \label{eqn:HC08_1}    
	\end{align}
	Since $M$ is locally doubling, we have the Hardy–Littlewood maximal inequality and the Lebesgue differentiation theorem.
   
	\begin{lemma}\cite{Ste70,Med25}\label{thm1.4}
		Suppose $(M, g)$ is a complete n-dimensional Riemannian manifold with $Rc\geq -(n-1) \Lambda_0$. 
        For each $p \in (1, \infty]$, there exists a positive constant $C=C(n,\Lambda_0,p)$ such that 
		\begin{equation}
			\begin{split}
				\|\mathbf{M}(f)\|_{L^p}\leq C\|f\|_{L^p},
			\end{split}
		\end{equation}
		and 
		\begin{equation}
			\begin{split}
				|\{(x,t)\in M\times [-1,0]:\mathbf{M}f(x,t)\geq \lambda\}|\leq \frac{C}{\lambda}\|f\|_{L^1}.
			\end{split}
		\end{equation}
	\end{lemma} 

\begin{lemma}\label{lebesgue}
   Suppose $(M, g)$ is a complete n-dimensional Riemannian manifold with $Rc\geq -(n-1) \Lambda_0$. 
   Then for every $f\in L^1_{\mathrm{loc}}(M\times [-1, 0])$ we have 
   for almost every $(x,t)\in M \times [-1,0]$:
    \begin{equation}
        \lim\limits_{r\to 0^{+}}\frac{1}{|P_r(x,t)|}\int_{P_r(x,t)}|f|dzds =f(x,t).
    \end{equation}
\end{lemma}

The harmonic coordinate plays an important role in the study of Riemannian geometry. 
We first recall the definition of related concepts.
	   
\begin{definition}($W_p^{k}$-harmonic radius) 
Suppose $(M^n, g)$ is a Riemannian manifold, $n<p<\infty$, and $x \in M$. 
The $W_p^{k}$-harmonic radius at $x$ is the supremum of all $R>0$ such that  there exists a coordinate chart $\phi:B_R(x)\to \IR^n$ satisfying 
\begin{itemize}
     \item $\Delta \phi^j=0$ on $B_{R}(x)$  and $\phi^j(x)=0$ for each $j$;
    \item $\frac12 \delta_{ij}\leq g_{ij}\leq 2 \delta_{ij}$,  in $B_{R}(x)$ as symmetric bilinear forms;
    \item $\sum_{1\leq |J|\leq k}R^{|J|-\frac{n}{p}}\|\partial^J g_{ij}\|_{L^p(B_R(x))}\leq 1$.
\end{itemize}
We denote the $W_{p}^{k}$-harmonic radius at $x$ by $r_{k,p}(x)$.   
\label{dfn:HB05_1}
\end{definition}

For simplicity of notation, we fix
    \begin{align}
    \label{eqn:HB07_4}
        p=n+4.  
    \end{align}

\begin{lemma}
\label{lma:HB13_5}
    There is a constant $\xi_a=\xi_a(n) \in (0, \frac{1}{100n\pi})$ with the following property.

    Suppose $(M, g)$ is a Riemannian manifold satisfying $|Rm| \leq \xi^2$.  Suppose $x \in M$ and
    $B_{100n\pi}(0) \subset T_x M$ is equipped with the pull-back metric $\tilde{g}=Exp_x^{*}g$.
    Then the $W_p^2$-harmonic radius of $0$ is at least $100$.
\end{lemma}

In the Ricci flow case, due to the Shi-type estimate (cf.~\cite{Shi1}, ~\cite{Shi2}),  
the covariant derivatives of the curvature tensor can be bounded by the curvature bound. Therefore, we have the following Lemma.

\begin{lemma}
\label{lma:HC07_1}
    There is a constant $\xi_b=\xi_b(n) \in (0, \frac{1}{100n\pi})$ with the following property.

    Suppose $\{(M, g(t)), -1 \leq t \leq 0\}$ is a Ricci flow solution satisfying $|Rm| \leq \xi^2$.  Suppose $x \in M$ and
    $B_{100n\pi}(0) \subset T_x M$ is equipped with the pull-back metric $\tilde{g}(t)=Exp_x^{*}g(t)$, where $Exp$ is the exponential map with respect to $g(0)$.
    Then for each $t \in [-\frac12, 0]$, 
    the $W_p^4$-harmonic radius of $0$ with respect to $g(t)$ is at least $100$.
\end{lemma}

The proofs of Lemma~\ref{lma:HB13_5} and Lemma~\ref{lma:HC07_1} are standard. Thus, we omit it here.  

\begin{definition}
    We define
    \begin{align}
        \xi := \min\{\xi_a, \xi_b\} \in \left(0, \frac{1}{100n\pi} \right), 
    \label{eqn:HC07_3}    
    \end{align}
    where $\xi_a$ and $\xi_b$ are the constants in Lemma~\ref{lma:HB13_5} and Lemma~\ref{lma:HC07_1} respectively. 
\label{dfn:HC07_2}    
\end{definition}

\begin{definition}
\label{dfn:HB14_1}
Suppose $0<\alpha<1$, $r>0$, and $l$ is a nonnegative integer. 
Suppose $s$ is a smooth section of a vector bundle $\mathcal{E} \to M$. 
Define
\begin{align}
\label{eqn:HB12_1}
    [s]_{l+\alpha}^{(r)}(x)
    :=\sup_{d(x,x')<r}
     \sup_{\gamma\in \Xi(x,x')} 
     \frac{|\nabla^l s(x)-\tau_\gamma(x,x')\nabla^l s(x')|}{d^\alpha(x,x')}
     <\infty,  
\end{align}
where 
\begin{align}
\Xi(x,x') :=\{\textrm{the shortest geodesic segment connecting}  \;x \; \textrm{and} \;x'\}, 
\end{align}
and $\tau_{\gamma}$ is the parallel transport along $\gamma$. 
\end{definition}

\section{Heat equations in harmonic atalas}
\label{sec3}

  In this section, we study fundamental estimates of heat solutions in harmonic coordinates.

By comparing the geodesic curve connecting a boundary point to the origin, it is not hard to see from the second condition of Definition~\ref{dfn:HB05_1} that 
\begin{align}
   B_{0.5R}(0) \subset  B_{\frac{\sqrt{2}}{2}R}(0) \subset \phi_R (B_R(x)) \subset B_{\sqrt{2}R}(0)
\label{eqn:SL09_1}   
\end{align}
for every $R<r_{k,p}(x)$.  
Therefore, the Euclidean ball $B_{0.5R}(0) \subset \mathbb{R}^n$ is equipped with two metrics, the Euclidean metric $g_E$ and the push-forward metric $\phi_R^* g$.   For simplicity of notation, we still denote $\phi_R^* g$ by $g$.  In other words, we can regard the identity map from $(B_{0.5R}(0), g)$ to $(B_{0.5R}(0), g_E)$ as a harmonic map.   
Thus, we have
\begin{align}
    0=\Delta_g x^k=g^{ij} \left( \frac{\partial^2 x^k}{\partial x^i \partial x^j} -\Gamma_{ij}^l \frac{\partial x^k}{\partial x^l} \right)=-g^{ij}\Gamma_{ij}^k. 
\label{eqn:SL09_2}    
\end{align}
The tensor $g_{ij}$ is now a matrix-valued function. In light of the relationship (\ref{eqn:SL09_1}), it follows from the definition of the harmonic radius that the following inequalities hold. 
\begin{align}
  &\frac12 \sum_{i=1}^n (V^i)^2 \leq   g_{ij}V^i V^j \leq 2\sum_{i=1}^n (V^i)^2, \quad \forall \; V \in \IR^n; \label{eqn:SL09_3} \\
  &\sum_{1\leq |J|\leq k}R^{|J|-\frac{1}{4}}\|\partial^J g_{ij}\|_{L^p(\ID)}\leq 1. \label{eqn:SL09_4} 
\end{align}

Suppose $R \geq 100$. Then $B_{0.5R}(0) \supset B_2(0)$. 
Define
\begin{align}
\begin{cases}
    &B:= B_2(0), \quad B' :=B_1(0); \\
    &\Omega := B \times [-4, 0], \quad \Omega' :=B' \times [-1, 0]. 
\end{cases}
\label{eqn:SL09_10}    
\end{align}

\begin{definition}
A local model space-time is a smooth family of metrics $g(t)$ defined on $\Omega$ such that
\begin{align}
    &\Delta_{g(0)} x^i=0 \; \textrm{on} \; B, \quad \forall \; i \in \{1, 2, \cdots, n\};\\
    &\sup_{\Omega} |Rm|(x,t) \leq \xi^2. 
\end{align}
Here, $\xi$ is the small positive constant defined in Definition~\ref{dfn:HC07_2}.
It is chosen in this way for later purposes (cf. (\ref{cdn:HB07_3}) and the discussion above it). 

Furthermore, the following estimates are satisfied. 
\begin{itemize}
    \item  If $g(t) \equiv g$, we require that $g|_{B}$ satisfies (\ref{eqn:SL09_3}) and (\ref{eqn:SL09_4}) with $k=2$. 
    \item  If $g(t)$ solves the Ricci flow equation, we require that $g(0)|_{B}$ satisfies
 (\ref{eqn:SL09_3}) and (\ref{eqn:SL09_4}) with $k=4$. 
\end{itemize}
\label{dfn:HA10_1}
\end{definition}

We shall study the behavior of the heat solutions in local model space-time. 

The following definitions of the parabolic Sobolev and H\"older norms are well known.

    \begin{align}
       &\|u\|_{W_p^{2m,m}(\Omega)} 
       :=\sum_{|J|+2k\leq  2m} \|D^J \partial_t^ku\|_{L^p}, \label{eq3.2}\\
       &\|u\|_{C^{m+\alpha, \frac{m+\alpha}{2}}(\Omega)}
       :=\sum_{|J|+2k\leq m} \left(\|D^J\partial_t^ku\|_{C^0}+r^{\alpha}[D^J \partial_t^ku]_{\alpha,\frac{\alpha}{2}} \right),   \label{eq3.3}   
    \end{align}
    where $J$ is a multi-index.  We have the following Sobolev embedding properties. 

\begin{lemma}
\cite[Theorem 10.4.10]{Kry24}
\label{lma:SL02_1}
    For any $m \in \IN$, $n+2<p<\infty$ and $\alpha=1-(n+2)/p$, we have 
    \begin{align}
         W_p^{2m,m}(\Omega) \hookrightarrow C^{2m-1+\alpha,\frac{2m-1+\alpha}{2}}(\Omega).
    \label{eqn:SL22_1}     
    \end{align}
    In particular, there is a constant $C=C(n,m,p)>0$ such that 
    \begin{align}
         \|u\|_{C^{2m-1+\alpha,\frac{2m-1+\alpha}{2}}}
         \leq C  \|u\|_{W_p^{2m,m}(\Omega)}.
    \label{eqn:SL22_2}     
    \end{align}
\end{lemma}

Let $L$ be a linear operator of the form
\begin{align}
 Lu=a^{ij}\partial^2_{ij}u+b^i\partial_i u+cu, \quad \; i,j \in \{1,2,\cdots,n\}, 
 \label{eqn:SL22_7}           
\end{align}
which satisfies the following conditions for some positive constant $\Theta$: 
\begin{align}
\begin{cases}
    &a^{ij}\in C^0(\Omega); \\
    &\Theta^{-1}|V|^2\leq a^{ij}V_iV_j\leq \Theta |V|^2,\quad \forall V \in \IR^n; \\
    &|a^{ij}|+|b^i|+|c| \leq \Theta.    
\end{cases}    
\label{eqn:SL22_3}
\end{align}
We shall consider the following parabolic equation
\begin{align}
    (\partial_t  -L) u =f.
\label{eqn:SL22_6}
\end{align}

\begin{lemma}\cite[Theorem 6.4.2]{Kry24}
\label{lma:SL08_1}
 Suppose that the coefficients of $L$ satisfy (\ref{eqn:SL22_3}). 
 Suppose $u$ is a smooth solution of~\eqref{eqn:SL22_6}.
 Then 
\begin{align}
    \|u\|_{W_p^{2,1}(\Omega')}\leq C \left\{\|u\|_{L^p(\Omega)}+\|f\|_{L^p(\Omega)} \right\}
\label{eqn:HB04_3}    
\end{align}
for some positive constant $C=C(n,p,\Theta)$. 
\end{lemma}

\begin{lemma}
\cite[Theorem 8.12.1]{Kry96}
\label{lma:SL31_12}
 Suppose that the coefficients of $L$ satisfy (\ref{eqn:SL22_3}) and the following bounds 
		\begin{equation}
			\begin{split}
			 \|a_{ij}\|_{C^{2m-2+\alpha, \frac{2m-2+\alpha}{2}}(\Omega)}+
				\|b_{i}\|_{C^{2m-2+\alpha, \frac{2m-2+\alpha}{2}}(\Omega)} 
                +\|c\|_{C^{2m-2+\alpha, \frac{2m-2+\alpha}{2}}(\Omega)}<\Theta
			\end{split}
		\end{equation}
        for some $m\geq 1$ and $\alpha\in (0,1)$. 
        Suppose $u\in C^{2+\alpha, \frac{2+\alpha}{2}}(\Omega)$ satisfies the equation \eqref{eqn:SL22_6}.
		Then 
		\begin{equation}
			\|u\|_{C^{2m+\alpha, \frac{2m+\alpha}{2}}(\Omega')}
            <C \left\{ \|u\|_{L^\infty(\Omega)}+\|f\|_{C^{2m-2+\alpha, \frac{2m-2+\alpha}{2}}(\Omega)} \right\}
		\end{equation}
        for some positive constant $C=C(n, m, \alpha, \Theta)$.  
\end{lemma}

\begin{lemma}\cite[Theorem 5.1]{Sch}
\label{lma:HA10_1}
    Let $u\in C^{2,1}(\Omega)$ be a classical solution of the following parabolic system
\begin{align*}
\begin{cases}
&u_t^{\mu}-A_{ij}^{\mu\nu}D_{ij}u^\nu +B_i^{\mu\nu}D_iu^{\nu}+C^{\mu\nu}u^\nu = f^\mu,  
\; \text{in} \; \Omega;\\
&u = 0,  \; \text{on} \; \partial \Omega.
\end{cases}
\end{align*}
Suppose that the following conditions hold
\begin{align}
\begin{cases}
    &A_{ij}^{\mu\nu}\in C^0(\Omega); \\
    &\Theta^{-1}|V|^2\leq a^{\mu\nu}_{ij}V^i_\mu V^j_\nu\leq \Theta |V|^2,\quad \forall \; V \in \mathbb{R}^n;\\
    &\|A_{ij}^{\mu\nu}\|_{C^{\alpha,\frac{\alpha}{2}}}+\|B_i^{\mu\nu}\|_{C^{\alpha,\frac{\alpha}{2}}}+\|C^{\mu\nu}\|_{C^{\alpha,\frac{\alpha}{2}}} 
    \leq \Theta. 
\end{cases}    
\end{align}
Then $D^2u,u_t\in C^{\alpha,\frac{\alpha}{2}}$ and 
\begin{equation}
    \begin{split}
        \|u\|_{C^{2+\alpha,\frac{1+\alpha}{2}}(\Omega)}\leq C \left\{\|u\|_{C^0(\Omega)}+\|f\|_{C^{\alpha,\frac{\alpha}{2}}(\Omega)} \right\}, 
    \end{split}
\end{equation}
where  $C=C(n,\alpha,\Theta)$. 
\end{lemma}

\begin{lemma}\cite[Theorem 7.2]{Sch}
\label{lma:HA10_2}
    Let $u\in C^{2,1}(\Omega)$ be a classical solution of the following parabolic system
    \begin{align*}
    \begin{cases}
     &u_t^{\mu}-A_{ij}^{\mu\nu}D_{ij}u^\nu +B_i^{\mu\nu}D_iu^{\nu}+C^{\mu\nu}u^\nu=f^\mu, \quad \text{in} \; \Omega;  \\
     &u=0, \quad \text{on} \; \partial \Omega.    
    \end{cases}
   \end{align*}
 Suppose that the following conditions hold
\begin{align*}
\begin{cases}
    &A_{ij}^{\mu\nu}\in C^0(\Omega); \\
    &\Theta^{-1}|V|^2\leq a^{\mu\nu}_{ij}V^i_\mu V^j_\nu\leq \Theta |V|^2,\quad \forall \; V \in \mathbb{R}^n;  \\
    &\|A_{ij}^{\mu\nu}\|_{L^\infty(\Omega)}+\|B_i^{\mu\nu}\|_{L^\infty(\Omega
    )}+\|C^{\mu\nu}\|_{L^\infty(\Omega)} \leq \Theta.
\end{cases}    
\end{align*}
Then
\begin{equation}
    \begin{split}
        \|u\|_{W_p^{2,1}(\Omega)}
        \leq C \left\{\|u\|_{L^p(\Omega)}+\|f\|_{L^p(\Omega)} \right\}, 
    \end{split}
\end{equation}
where $C=C(n,p,\Theta)$. 
\end{lemma}

Then we start to analyze the regularity of the metrics and PDE solutions on $\Omega$.    

\begin{lemma}
\label{lma:SL30_1}
 Based on the choice of $\Omega$ and $g$, the following estimates hold.
\begin{itemize}
    \item In the static metric case, we have 
    \begin{align}
       \|g\|_{C^{1+\alpha}(B_2)} \leq C\|g\|_{W_p^{2}(B_2)} \leq C. 
    \label{eqn:SL30_1}   
    \end{align}
    \item  In the Ricci flow case, we have
    \begin{align}
        \|g\|_{C^{2+\alpha, \frac{2+\alpha}{2}}(\Omega)} \leq C. 
    \label{eqn:SL30_2}    
    \end{align}
\end{itemize}    
In short, we always have
\begin{align}
    \|g\|_{C^{1+\alpha, \frac{1+\alpha}{2}}(\Omega)} \leq C. 
\label{eqn:SL30_6}    
\end{align}
\end{lemma}

\begin{proof}
  It is clear that (\ref{eqn:SL30_6}) follows from the combination of (\ref{eqn:SL30_1}) and (\ref{eqn:SL30_2}).  

  In light of the $W_p^2$-harmonic radius assumption and the elliptic Sobolev embedding, we obtain  (\ref{eqn:SL30_1})  directly. Thus, it suffices to show (\ref{eqn:SL30_2}).
  
    As the metric $g(0)$ satisfies the uniform $W_p^4(B_2)$-norm bound, similar to (\ref{eqn:SL30_1}), the elliptic Sobolev embedding implies
  \begin{align}
      \|g(0)\|_{C^{3+\alpha}(B_2)} \leq C.  
  \label{eqn:SL30_4}    
  \end{align}
  Then we can take both time and space derivatives of $g$. Each time derivative of $g_{ij}$ is a tensor consisting of a curvature tensor and their covariant derivatives, which are bounded by the Shi-type estimate. 
  Thus, we have
  \begin{align}
      \sum_{j=0}^3 \|\partial_t^j g\|_{L^{\infty}(\Omega)} \leq C. 
  \label{eqn:SL30_3}    
  \end{align}
  The space partial derivatives can also be bounded, after we take care of the Levi-Civita connection terms. 
  To be more precise, we have
\begin{align*}
    &\partial_t g_{ij}=-2R_{ij}, \Rightarrow \; C^{-1}g_{ij}(0) \leq g_{ij}(t) \leq C g_{ij}(0), \\
    &\partial_t \Gamma_{ij}^k=-g^{kl}(R_{li,j}+R_{lj,i}-R_{ij,l}), \Rightarrow |\Gamma| \leq C,\\
    &\partial_t \partial_k g_{ij}=-2 \partial_k R_{ij}
    =-2(R_{ij,k}+\Gamma_{ki}^qR_{qj} + \Gamma_{kj}^qR_{qi}), \Rightarrow |\partial g| \leq C, \\
    &\partial_t \partial_l \Gamma_{ij}^k=\partial_l \partial_t \Gamma_{ij}^k=g^{-1}*g^{-1}*\partial g* (\nabla Rc)+g*\nabla \nabla Rc + g* \nabla Rc * \Gamma \Rightarrow |\partial \Gamma| \leq C.
\end{align*}
Then we continue to estimate $\partial^2 g$, $\partial^2 \Gamma$, $\partial^3 g$ by taking their time derivatives and comparing them with their values at time $t=0$.  Since all the covariant derivatives of the curvature are uniformly bounded, it follows from (\ref{eqn:SL30_4}) that
\begin{align}
    \sum_{|J| \leq 3} \|\partial_J g\|_{L^{\infty}(\Omega)} \leq C. 
\label{eqn:SL30_5}    
\end{align}
Therefore, (\ref{eqn:SL30_2}) follows from the combination of (\ref{eqn:SL30_3}) and (\ref{eqn:SL30_5}). 
\end{proof}

\begin{proposition}
\label{prn:HA04_1}
    Suppose $u$ solves the heat equation $(\partial_t - \Delta_g) u=0$ on $\Omega$. Then
    \begin{align}
        \|u\|_{C^{2+\alpha,\frac{2+\alpha}{2}}(\Omega')} \leq C(n) \|u\|_{L^{\infty}(\Omega)}. 
    \label{eqn:HA03_4}    
    \end{align}
\end{proposition}

\begin{proof}
We should prove (\ref{eqn:HA03_4}) case by case. \\

\textit{Case 1. $g$ is static.}\\

By the Sobolev embedding, it suffices to show 
    \begin{align}
        \|u\|_{W_p^{4,2}(\Omega')} \leq C(n) \|u\|_{L^{\infty}(\Omega)}. 
    \label{eqn:SL09_05}    
    \end{align}
In light of the harmonic condition (\ref{eqn:SL09_2}), we can write down the heat equation as
\begin{align}
    (\partial_t - g^{ij} \partial_i \partial_j ) u=0. 
\label{eqn:SL09_5}    
\end{align}
As $g^{ij}$ is independent of time, the Sobolev embedding implies $g^{ij}\in C^{1+\alpha}(B)$. Thus, (\ref{eqn:SL09_05}) follows from
the standard parabolic equation theory.  We provide more details for the convenience of the readers. In each step, it is possible that we need to shrink the domain slightly. For simplicity of notation, we ignore this and do the domain change just in the last step.

By Lemma~\ref{lma:SL08_1} we have 
\begin{align}
  \|u\|_{W_p^{2,1}(\Omega)} \leq C \|u\|_{L^p(\Omega)} \leq  C\|u\|_{L^\infty(\Omega)}.
\label{eqn:SL14_1}  
\end{align}
Let $\dot{u}=\partial_t u$. Then (\ref{eqn:SL14_1}) actually means that
\begin{align}
    \|\dot{u}\|_{L^p(\Omega)} + \|\partial u\|_{L^p(\Omega)} +\|\partial \partial u\|_{L^p(\Omega)}
    \leq C\|u\|_{L^\infty(\Omega)}.
\label{eqn:SL14_3}    
\end{align}
Here by $\partial$ without subscripts, we mean the derivative in the space direction. 
Taking the space derivative of (\ref{eqn:SL09_5}) yields the 
\begin{align}
    \left(\partial_t - g^{ij}\partial_i \partial_j \right) \partial_ku=\partial_k g^{ij}\partial_{i} \partial_j u.
\label{eqn:SL09_6}    
\end{align}
Applying Lemma~\ref{lma:SL08_1} on (\ref{eqn:SL09_6}), we obtain 
\begin{align}
    \|\partial_k u\|_{W_p^{2,1}(\Omega)}
    \leq C (\|\partial_k u\|_{L^p(\Omega)}+\|\partial_k g^{ij}\partial^2_{ij}u\|_{_{L^p(\Omega)}}).
\label{eqn:SL14_2}    
\end{align}
Since $ \partial_k g^{ij}=-g^{iq}g^{jl} \partial_k g_{lq}$ is uniformly bounded on $\Omega$,  it follows from (\ref{eqn:SL14_1}) and (\ref{eqn:SL14_2}) that 
\begin{align}
 \|\partial^3 u\|_{L^{p}(\Omega)} \leq C \| u\|_{W_p^{2,1}(\Omega)}\leq C \|u\|_{L^\infty(\Omega)}.
\label{eqn:SL09_8} 
\end{align}

Taking the time derivative of (\ref{eqn:SL09_5}), we obtain $\left(\partial_t - g^{ij}\partial_i \partial_j \right) \dot{u}=0$. By De Giorgi-Nash-Moser iteration and (\ref{eqn:SL14_3}), we have 
\begin{align*}
    \|\dot{u}\|_{L^{\infty}(\Omega)} \leq C \|\dot{u}\|_{L^p(\Omega)} \leq C \|u\|_{L^{\infty}(\Omega)}. 
\end{align*}
Thus, similar argument as in (\ref{eqn:SL14_3}) implies that
\begin{align}
    \|\partial_t \dot{u}\|_{L^p(\Omega)} + \|\partial \dot{u}\|_{L^p(\Omega)} 
    +\|\partial \partial \dot{u}\|_{L^p(\Omega)}
    \leq C\|\dot{u}\|_{L^\infty(\Omega)} \leq C \|u\|_{L^{\infty}(\Omega)}, 
\label{eqn:SL14_4}    
\end{align}
which implies
\begin{align}
    \|\partial_t \partial \partial u\|_{L^p(\Omega)}
    \leq C\|u\|_{L^\infty(\Omega)}. 
    \label{eqn:SL14_5}
\end{align}
By (\ref{eqn:SL09_8}), (\ref{eqn:SL14_5}), and the Sobolev embedding, we know 
\begin{align}
    \|\partial \partial u\|_{L^{\infty}(\Omega)} \leq  \|\partial \partial u\|_{C^{\alpha}(\Omega)}
    \leq C \|u\|_{L^\infty(\Omega)}. 
\label{eqn:SL09_7}    
\end{align}

Taking the second space derivative of (\ref{eqn:SL09_5}) yields
\begin{equation}
\begin{split}
    \left(\partial_t-g^{ij}\partial_i \partial_j \right)\partial_k \partial_l u
    =\partial_k g^{ij}\partial_i \partial_j \partial_lu+\partial_l g^{ij}\partial_i \partial_j \partial_ku+\partial_k \partial_l g^{ij}\partial_i \partial_j u.
    \end{split}
\end{equation}
Note that $\partial g^{ij}$ is uniformly bounded and $\partial_i \partial_j u$ 
is dominated by (\ref{eqn:SL09_7}). 
Using Lemma~\ref{lma:SL08_1} and (\ref{eqn:SL09_8}), we have 
\begin{align}
    \|\partial^4 u\|_{L^p(\Omega)}
    \leq C \left\{ \|\partial^3 u\|_{L^p(\Omega)} + \|\partial \partial g\|_{L^{p}(\Omega)} \|u\|_{L^{\infty}(\Omega)}\right\}
    \leq C \|u\|_{L^\infty(\Omega)}. 
\label{eqn:SL27_1}    
\end{align}
Recalling that from (\ref{eqn:SL14_4}) we have
\begin{align}
    \|\partial_t \dot{u}\|_{L^p(\Omega)} \leq C \|u\|_{L^\infty(\Omega)}. 
\label{eqn:SL27_2}    
\end{align}
Recall the definition of the $W_p^{4,2}$-norm, it is straightforward that (\ref{eqn:SL09_05}) follows from (\ref{eqn:SL14_3}), (\ref{eqn:SL09_8}), (\ref{eqn:SL14_4}), (\ref{eqn:SL27_1}) and (\ref{eqn:SL27_2}). \\

\textit{Case 2. $g$ solves the Ricci flow equation.}\\

Now, the chart is only harmonic with respect to $g(0)$. 
Compared with (\ref{eqn:SL09_5}), the heat equation is slightly more complicated
\begin{align}
    \dot{u} -g^{ij} \partial_i \partial_j  u 
    +g^{ij}\Gamma_{ij}^k \partial_k u=0. 
\label{eqn:HA03_5}    
\end{align}
In light of the metric regularity (\ref{eqn:SL30_2}), we can set $m=1$ and apply Lemma~\ref{lma:SL31_12} to obtain 
\begin{align*}
    \|u\|_{C^{2+\alpha, \frac{2+\alpha}{2}}(\Omega')}
    \leq C \|u\|_{L^{\infty}(\Omega)},
\end{align*}
which is exactly (\ref{eqn:HA03_4}). 
\end{proof}

\begin{lemma}
    Suppose $u$ is a smooth function on $B$.  Then 
\begin{align}
        \sup_{x \in B'} \left\{ |\nabla u|_g(x) + |\nabla \nabla u|_g(x) +[\nabla \nabla u]_{g,\alpha}^{(\frac12)}(x) \right\} \leq C(n) \|u\|_{C^{2+\alpha}(B)}. 
    \label{eqn:SL11_1}  
\end{align}
Notice that (cf. Definition~\ref{dfn:HB14_1}) 
 \begin{align}
   [\nabla \nabla u]_{g,\alpha}^{(\frac12)}(x) := \sup_{y \in B_{\frac12}(x)}  \frac{|u_{ij}(x)-\tau_{\gamma}(x,y)u_{ij}(y)|}{|\gamma|^{\alpha}},   
 \label{eqn:HB07_1}  
 \end{align}
 where $\gamma$ is any shortest geodesic connecting $x$ and $y$ under the metric $g$, $\tau_{\gamma}(x,y)$ is the parallel transportation from $y$ to $x$ along the geodesic $\gamma$. 
\label{lma:SL13_1}
\end{lemma}

\begin{proof}
 In $B'$, it is clear that
 \begin{align}
   \frac12 |\partial u|^2  
   \leq  |\nabla u|_g^2= g^{ij}u_i u_j \leq 2 |\partial u|^2 
 \label{eqn:SL11_2}  
 \end{align}
 holds pointwise. 
 Note that
 \begin{align*}
     |\nabla \nabla u|_g^2&=g^{ij}g^{kl} u_{ik}u_{jl}
     =g^{ij}g^{kl} (\partial_i \partial_k u -\Gamma_{ik}^q \partial_q u) 
     (\partial_j \partial_l u - \Gamma_{jl}^m \partial_m u)\\
     &\leq C \left\{ |\partial \partial u|^2 + |\Gamma|^2 |\partial u|^2 \right\}
      \leq C \left\{ |\partial \partial u|^2 + |\partial g|^2 |\partial u|^2 \right\}. 
 \end{align*}
 Since $\partial g$ is uniformly bounded by the Sobolev embedding, we have 
 \begin{align}
     \sup_{x \in B'} |\nabla \nabla u|_g(x) \leq C \|u\|_{C^2(B)} 
     \leq C \|u\|_{C^{2+\alpha}(B)}. 
 \label{eqn:SL11_3}    
 \end{align}
 
 Now fix $x\in B'$ and $y \in B_{\frac12}(x)$.  
 Let $\gamma$ be a shortest unit-speed geodesic connecting $x$ and $y$ such that $\gamma(0)=x$ and $\gamma(L)=y$.  Under the assumption of small curvature, we know $\gamma \subset B$. 
 We want to show
 \begin{align}
    |u_{ij}(x)-\tau_{\gamma}(x,y)u_{ij}(y)| \leq C \|u\|_{C^{2+\alpha}(B)} \cdot L^{\alpha}. 
 \label{eqn:SL11_5}   
 \end{align}
 For simplicity of notation, we define
   \begin{equation}
    \Upsilon(\gamma(\tau)):=\tau_\gamma(\gamma(\tau),y) \nabla^2 u(y)=\Upsilon_{ij}(\gamma(\tau))dx^i\otimes dx^j.
   \end{equation}
 From the definition, it is clear that $\Upsilon_{ij}(y)=u_{ij}(y)$. 
 Then we have 
   \begin{equation}
   \begin{split}
    &\quad |u_{ij}(x)-\tau_{\gamma}(x,y)u_{ij}(y)|\\
    &=|u_{ij}(x)-\Upsilon_{ij}(x)|
       \leq |u_{ij}(x)-\Upsilon_{ij}(y)| + |\Upsilon_{ij}(y)-\Upsilon_{ij}(x)|\\
    &=\underbrace{|u_{ij}(x)-u_{ij}(y)|}_{I} + \underbrace{|\Upsilon_{ij}(y)-\Upsilon_{ij}(x)|}_{II}. 
   \end{split}
   \label{eqn:SL11_7}
   \end{equation}
  For part $I$, we know 
  \begin{equation*}
  \begin{split}
      I&=\left|\partial_i \partial_j u(x)-\partial_i \partial_j u(y) - (\Gamma_{ij}^k(x) \partial_k u(x) -\Gamma_{ij}^k(y) \partial_k u(y)) \right|\\
      &=\left| \left\{\partial_i \partial_j u(x)-\partial_i \partial_j u(y) \right\}
      -\Gamma_{ij}^k(x)  \left\{\partial_k u(x)-\partial_k u(y) \right\} 
      +\left\{ \Gamma_{ij}^k(y) -\Gamma_{ij}^k(x)  \right\}\partial_k u(y) \right|\\
      &\leq C L^{\alpha} \left\{ \|\partial^2 u\|_{C^{\alpha}(B)} + \|\partial g\|_{C^0(B)}
      \|\partial u\|_{C^{\alpha}(B)} + \|\partial u\|_{C^0(B)}
      \|\partial g\|_{C^{\alpha}(B)} \right\}. 
  \end{split}    
  \end{equation*}
  In short, we have
  \begin{align}
      I \leq CL^{\alpha} \|u\|_{C^{2+\alpha}(B)}. 
      \label{eqn:SL11_8}
  \end{align}
  For part II, we note that $\Upsilon$ satisfies the following equation of parallel transportation
    \begin{align*}
    \frac{d\Upsilon_{j_1 j_2}}{d\tau}-(\Gamma^{j}_{j_1k}\Upsilon_{jj_2}+\Gamma^{j}_{j_2k}\Upsilon_{j_1j})\frac{dx^k(\gamma(\tau))}{d\tau}=0.
    \end{align*}
  This can be written as 
  \begin{align}
    \frac{d\Upsilon_{j_1 j_2}}{d\tau}-\left\{ \left(\Gamma^{j}_{j_1k}\frac{dx^k(\gamma(\tau))}{d\tau} \right)\Upsilon_{jj_2}+ \left(\Gamma^{j}_{j_2k} \frac{dx^k(\gamma(\tau))}{d\tau} \right) \Upsilon_{j_1j} \right\} =0.
  \label{eqn:HB06_3}  
  \end{align}
  Note that 
  \begin{align*}
      \left\|\Gamma^{j}_{l k}\frac{dx^k(\gamma(\tau))}{d\tau} \right\|_{C^0(\gamma)} 
      \leq C \|\partial g\|_{C^0(B)}
      \leq C. 
  \end{align*}
  Applying Gronwall's inequality (cf. (\ref{eqn:HB05_4}) in Lemma~\ref{lma:HB05_2}) on the ODE system (\ref{eqn:HB06_3}), we obtain
  \begin{align}
      \|\Upsilon\|_{C^0(\gamma)} \leq C |\Upsilon(y)|  
      \leq C \|u\|_{C^0(B)} \leq C \|u\|_{C^{2+\alpha}(B)}. 
  \end{align}
  Thus, the standard ODE estimates (cf. (\ref{eqn:HB06_1}) in Lemma~\ref{lma:HB05_2}) imply that 
  \begin{align}
      II=|\Upsilon_{ij}(x)-\Upsilon_{ij}(y)| \leq C \|\Upsilon\|_{C^0(\gamma)} \cdot L
      \leq C \|u\|_{C^{2+\alpha}(B)} \cdot L^{\alpha}, 
      \label{eqn:SL11_6}
  \end{align}
  where we used the facts $L \in (0, 2)$ and $\alpha \in (0,1)$.
  
     Plugging (\ref{eqn:SL11_8}) and (\ref{eqn:SL11_6}) into (\ref{eqn:SL11_7}), we obtain (\ref{eqn:SL11_5}). 
  Combining (\ref{eqn:SL11_5}) with (\ref{eqn:SL11_2}) and (\ref{eqn:SL11_3}), we arrive at (\ref{eqn:SL11_1}). 
\end{proof}

  \begin{lemma}
  \label{lma:HB05_2}
      Let $A\in C([0,L],\IR^{n^2\times n^2})$ and $x(\tau) \in C^1([0,L],\IR^{n^2})$ satisfy
      \begin{align}
          \dot{x}(\tau)=A(\tau)x(\tau),\quad x(0)=x_0.
      \label{eqn:HB05_3}    
      \end{align}
      Then the following estimate holds:
      \begin{align}
      &|x(\tau)| \leq \exp\bigg(\int_{0}^{\tau}|A(s)|ds\bigg) \cdot |x_0|, \label{eqn:HB05_4} \\
      &|x(\tau)-x(0)| \leq  \|A\|_{C^0[0,L]} \cdot \|x\|_{C^0[0,L]} \cdot \tau. \label{eqn:HB06_1}
      \end{align}
  \end{lemma}
  
\begin{proof}
    The differential equation (\ref{eqn:HB05_3}) can be rewritten as the integral equation
    \begin{align}
        x(\tau) =x_0+\int_{0}^{\tau}A(s)x(s)ds, 
    \label{eqn:HB06_2}    
    \end{align}
    which implies
    \begin{align*}
        |x(\tau)|\leq |x_0|+\int_{0}^{\tau} |A(s)|\cdot |x(s)| ds.
    \end{align*}
    Applying Gronwall's inequality to the above inequality, we obtain (\ref{eqn:HB05_4}).
    It is clear that (\ref{eqn:HB06_1}) follows from (\ref{eqn:HB06_2}). 
\end{proof}

Combining Proposition~\ref{prn:HA04_1} and Lemma~\ref{lma:SL13_1}, we obtain the following Corollary. 

\begin{corollary}
    Suppose $u$ satisfies the heat equation. 
    Then we have 
     \begin{align}
        \sup_{(x,t) \in \Omega'} \left\{ |\nabla u|_g + |\nabla \nabla u|_g
        +[\nabla \nabla u]_{g,\alpha}^{(\frac12)}
        +|\dot{u}|\right\} \leq
         C(n)  \|u\|_{L^{\infty}(\Omega)}. 
    \label{eqn:HA03_7}    
    \end{align}
\label{cly:HA03_6}    
\end{corollary}

Then we deal with the tensor case. 

\begin{proposition}
    Suppose $u=u_{ij}dx^i\otimes dx^j$ satisfies
    \begin{align}
        (\partial_t -\Delta_g) u=h, \quad \textrm{on} \quad \Omega. 
    \label{eqn:SL12_1}    
    \end{align}
    Then 
    \begin{align}
        \|u\|_{W_p^{2,1}(\Omega')} \leq C(n) \left\{ \|u\|_{L^{\infty}(\Omega)} + \|h\|_{L^p(\Omega)} \right\}. 
        \label{eqn:SL14_11}
    \end{align}    
\end{proposition}

\begin{proof}
  In the given $W_p^{k}$-harmonic coordinate system of $g(0)$,  equation (\ref{eqn:SL12_1}) can be written as 
 \begin{align}
      (\partial_t-g^{kl}\partial_l \partial_k) u_{ij} 
      +\left\{ \partial_l u_{ja} \Gamma_{ki}^a +\partial_l u_{ia} \Gamma_{kj}^a+\partial_k u_{ja} \Gamma_{li}^a +\partial_k u_{ia} \Gamma_{lj}^a\right\}g^{kl} 
     ={f_{ij}}+h_{ij},
 \label{eqn:SL31_1}    
 \end{align}
 where
 \begin{align*}
     f_{ij}&=-u_{ja} \left(\partial_l \Gamma_{ki}^a - \Gamma_{kb}^a\Gamma_{li}^b \right)g^{kl}
    -u_{ia}\left(\partial_l \Gamma_{kj}^a - \Gamma_{kb}^a\Gamma_{lj}^b \right)g^{kl}
    + u_{ab} \left( \Gamma_{li}^a \Gamma_{kj}^b +\Gamma_{ki}^a \Gamma_{lj}^b\right)g^{kl}\\
    &\quad -\left(\partial_a u_{ij} -u_{jb}\Gamma_{ai}^b -u_{ib}\Gamma_{aj}^b\right) \Gamma_{lk}^a. 
 \end{align*}
 Note that $\frac12 \delta_{kl} \leq g_{kl} \leq 2 \delta_{kl}$ in $B$, which guaranties the uniform parabolic condition. 
 Also note that $\|\Gamma\|_{L^{\infty}(\Omega)} \leq C(n)$ by (\ref{eqn:SL30_6}). 
 Furthermore, the right hand side of (\ref{eqn:SL31_1}) can be bounded as
 \begin{align*}
     \|f\|_{L^p(\Omega)} 
     \leq C \|u\|_{L^{\infty}(\Omega)} \cdot \left\{1+\|\partial \partial g\|_{L^p(\Omega)}\right\}
     \leq C \|u\|_{L^{\infty}(\Omega)}. 
 \end{align*}
 Therefore, by the $W_p^{2,1}$ estimate for parabolic systems (cf. Lemma~\ref{lma:SL08_1}), we have
\begin{align*}
     \|u\|_{W_p^{2,1}(\Omega')} 
    &\leq C \left\{ \|u\|_{L^p(\Omega)} + \|f+h\|_{L^p(\Omega)}  \right\}
    \leq C \left\{ \|u\|_{L^{\infty}(\Omega)} + \|f\|_{L^p(\Omega)} +\|h\|_{L^p(\Omega)}  \right\}\\
   &\leq C \left\{ \|u\|_{L^{\infty}(\Omega)} + \|h\|_{L^p(\Omega)} \right\},  
 \end{align*}
 which is exactly (\ref{eqn:SL14_11}). 
 \end{proof}

\begin{theorem}
    Suppose $u=u_{ij}dx^i\otimes dx^j$ satisfies one of the following equations on $\Omega$ with a static metric $g$:
    \begin{itemize}
        \item[(1).] The heat equation:
        \begin{align}
        (\partial_t-\Delta_g)u=0.
        \label{eqn:SL13_4}
        \end{align}
        \item[(2).] The Lichnerowicz heat equation:
        \begin{align}
        (\partial_t- \Delta_L)u=0.
        \label{eqn:SL13_5}
        \end{align}
    \end{itemize}
    Then we have 
     \begin{align}
        \sup_{(x,t) \in \Omega'} \left\{ |\nabla u|_g +[\nabla u]_{g,\alpha}^{(\frac12)}
        +|\dot{u}|\right\} \leq
         C(n)  \|u\|_{L^{\infty}(\Omega)}. 
    \label{eqn:SL13_6}    
    \end{align}
\label{thm:HA03_3}    
\end{theorem}

\begin{proof}

In both cases, we can write $u$ as a solution of
\begin{align}
    (\partial_t - \Delta_g)u=A(u) 
    \label{eqn:SL19_4}
\end{align}
for some $A \in C^{\infty}(End(Sym^2T^*M))$ satisfying
\begin{align}
    |A| \leq \xi.
    \label{eqn:HB04_2}
\end{align}
In the case of the heat equation, $A \equiv 0$. In the case of the Lichnerowicz heat equation, we have 
\begin{align*}
   A(u)_{ij}=R_{ipqj}u_{kl}g^{pk}g^{ql}-R_{ip}u_{kj}g^{pk}-R_{jp}u_{ki}g^{pk}.  
\end{align*}
Therefore, it suffices to show (\ref{eqn:SL13_6}) under the condition (\ref{eqn:SL19_4}). \\

First, as $u$ solves (\ref{eqn:SL19_4}), then the standard $W_p^{2,1}$-estimate of the parabolic equation system implies that 
 \begin{align}
 \label{eqn:HB04_4}
 \begin{split}
   \|\dot{u}\|_{L^p(\Omega)} &\leq \|\dot{u}\|_{L^p(\Omega)} + \|\partial \partial u\|_{L^{p}(\Omega)} \leq \|u\|_{W_p^{2,1}(\Omega)} 
   \leq C \left\{ \|u\|_{L^p(\Omega)} + \|A(u)\|_{L^p(\Omega)} \right\}\\
   &\leq C \|u\|_{L^{\infty}(\Omega)}. 
 \end{split}  
 \end{align}

Secondly, $\dot{u}$ also solves the heat-type equation $(\partial_t-\Delta_g)\dot{u}=A(\dot{u})$. It follows from the bound of $|A|$ (cf. (\ref{eqn:HB04_2})) and Lemma~\ref{lma:SL08_1} that
\begin{align*}
    \|\dot{u}\|_{W_p^{2,1}(\Omega)}\leq  C\|\dot{u}\|_{L^p(\Omega)}.
\end{align*}
Then the parabolic Sobolev embedding (cf. Lemma~\ref{lma:SL02_1}) yields that 
 \begin{align}
 \|\dot{u}\|_{L^\infty(\Omega)}\leq  C\|\dot{u}\|_{L^p(\Omega)}.
\label{eqn:HB04_5} 
\end{align}
Combining (\ref{eqn:HB04_4}) and (\ref{eqn:HB04_5}), we have
\begin{align}
    \|\dot{u}\|_{L^\infty(\Omega)}\leq  C\|u\|_{L^{\infty}(\Omega)}. 
    \label{eqn:SL14_12}
\end{align}

 We now apply (\ref{eqn:SL14_11}) for both $u$ and $\dot{u}$.  Recall that  
\begin{align}
    \|u\|_{W_p^{2,1}(\Omega)} \leq C\|u\|_{L^{\infty}(\Omega)}.
    \label{eqn:SL19_1}
\end{align}
 Then the parabolic Sobolev embedding again implies that
 \begin{align*}
        \sup_{s \in [-1,0]} \|\partial u(\cdot, s)\|_{C^{\alpha}(B')} 
        \leq C(n)  \|u\|_{L^{\infty}(\Omega)}.   
 \end{align*}
 Then by a similar estimate as in the proof of Lemma~\ref{lma:SL13_1}, we obtain
 \begin{align}
        \sup_{s \in [-1,0]} \|\nabla u(\cdot, s)\|_{C_g^{\alpha}(B')} \leq C(n)  \|u\|_{L^{\infty}(\Omega)}.    
  \label{eqn:SL14_13}        
 \end{align}
 Therefore, (\ref{eqn:SL13_6}) follows from the combination of (\ref{eqn:SL14_12}) and (\ref{eqn:SL14_13}).    
\end{proof}

\begin{theorem}
    Suppose $\Omega$ is equipped with the Ricci flow metric $g(t)$.
    Suppose $u=u_{ij}dx^i\otimes dx^j$ satisfies the heat equation (\ref{eqn:SL13_4}) or the Lichnerowicz heat equation (\ref{eqn:SL13_5}). 
    Then we have 
     \begin{align}
        \sup_{(x,t) \in \Omega'} \left\{ |\nabla u|_g + |\nabla \nabla u|_g
        +|\dot{u}|\right\} \leq
         C(n)  \|u\|_{L^{\infty}(\Omega)}. 
    \label{eqn:SL31_3}    
    \end{align}
\label{thm:HA03_2}    
\end{theorem}

\begin{proof}
We define $h_{ij}=0$ in the heat equation case and define
\begin{align}
 h_{ij}=R_{ipqj}u_{kl}g^{pk}g^{ql}-R_{ip}u_{kj}g^{pk}-R_{jp}u_{ki}g^{pk}
\label{eqn:HB04_6} 
\end{align}
in the Lichnerowicz heat equation case. 

 As in (\ref{eqn:SL31_1}), we can write down the equation satisfied by $u$ as 
 \begin{align*}
      (\partial_t-g^{kl}\partial_l \partial_k) u_{ij} 
      +\left\{ \partial_l u_{ja} \Gamma_{ki}^a +\partial_l u_{ia} \Gamma_{kj}^a+\partial_k u_{ja} \Gamma_{li}^a +\partial_k u_{ia} \Gamma_{lj}^a\right\}g^{kl} 
     -{f_{ij}}-h_{ij}=0,    
 \end{align*}
 where
 \begin{align}
 \label{eqn:HB04_7}
 \begin{split}
     f_{ij}
     &=-u_{ja} \left(\partial_l \Gamma_{ki}^a - \Gamma_{kb}^a\Gamma_{li}^b \right)g^{kl}
    -u_{ia}\left(\partial_l \Gamma_{kj}^a - \Gamma_{kb}^a\Gamma_{lj}^b \right)g^{kl}\\
    &\quad + u_{ab} \left( \Gamma_{li}^a \Gamma_{kj}^b +\Gamma_{ki}^a \Gamma_{lj}^b\right)g^{kl}
      -\left(\partial_a u_{ij} -u_{jb}\Gamma_{ai}^b -u_{ib}\Gamma_{aj}^b\right) \Gamma_{lk}^a. 
 \end{split}   
 \end{align}
 
 By Lemma~\ref{lma:SL30_1}, in $\Omega$ we have 
    \begin{align}
        \|g\|_{C^{2+\alpha, \frac{2+\alpha}{2}}(\Omega)} + \|\Gamma \|_{C^{1+\alpha, \frac{1+\alpha}{2}}(\Omega)} 
        +\|\partial \Gamma \|_{C^{\alpha, \frac{\alpha}{2}}(\Omega)} < C. 
    \label{eqn:SH01_2}    
    \end{align}
 In light of the choice of $f_{ij}$ and $h_{ij}$ in (\ref{eqn:HB04_6}) and (\ref{eqn:HB04_7}), we can write 
 \begin{align*}
     -f_{ij}=(M*u)_{ij}, \quad -h_{ij}=(N*u)_{ij}, 
 \end{align*}
 for some matrix valued functions $M$ and $N$.
 Thus,  (\ref{eqn:SH01_2}) and the curvature derivative bounds imply that
 \begin{align*}
     \| M \|_{C^{\alpha, \frac{\alpha}{2}}(\Omega)} 
     +\| N \|_{C^{\alpha, \frac{\alpha}{2}}(\Omega)}<C. 
 \end{align*}
 The heat equation satisfied by $u$ can be written as
 \begin{align*}
     (\partial_t-g^{kl}\partial_l \partial_k) u + (\Gamma *\partial u) + M*u + N*u=0.  
 \end{align*}

    By setting $m=1$,  we apply Lemma~\ref{lma:SL31_12} to the above to obtain
    \begin{align}
			\|u\|_{C^{2+\alpha, \frac{2+\alpha}{2}}(\Omega')}
            <C  \|u\|_{L^\infty(\Omega)}, 
	\end{align}
    which implies (\ref{eqn:SL31_3}) directly. 

\end{proof}

Combining Theorem~\ref{thm:HA03_3} and Theorem~\ref{thm:HA03_2}, we immediately obtain the following corollary. 

\begin{corollary}
\label{cly:HA03_1}
   Suppose $\Omega$ is equipped with a static metric $g$ or Ricci flow metrics $g(t)$. 
   Suppose $u=u_{ij}dx^i\otimes dx^j$ satisfies the heat equation or the Lichnerowicz heat equation. 
   Then we have 
   \begin{align}
        \sup_{(x,t) \in \Omega'} \left\{ |\nabla u|_g + [\nabla u]_{g,\alpha}^{(\frac12)}
        +|\dot{u}|_g\right\} \leq
         C(n)  \|u\|_{L^{\infty}(\Omega)}. 
    \label{eqn:SL13_8}    
    \end{align}
\end{corollary}

\section{Heat kernel estimate}
\label{sec4}

In this section, we estimate various Laplacian tensor-valued heat kernels on evolving Riemannian manifolds $\{(M^n, g(t)), -1 \leq t \leq 0\}$ with bounded curvature $|Rm| \leq \Lambda_0$.  The underlying space is either static, or evolving by a Ricci flow.  
    

 Note that the curvature condition itself cannot exclude the happening of collapsing. 
 In other words, the cut radius could be very small.   This phenomenon causes the major difficulty of this section. However, it can be overcome by the standard technique: we shall lift the metrics $g(t)$ locally to some tangent space $T_x M$.

    \begin{lemma}
    \label{lma:HA10_3}
     Suppose $\{(M^n, g(t)), -1 \leq t \leq 0\}$ is an evolving manifold satisfying $|Rm| \leq \Lambda_0$. 
     Let $H(x,t;y,s)$ be the heat kernel function. Then 
      \begin{align}
      \label{eqn:HA10_2}
            H(x,t;y,s)\leq CV_{x,t}(\sqrt{t-s})^{-1}e^{-\frac{d_t^2(x,y)}{C(t-s)}}.
        \end{align}
    \end{lemma}

    \begin{proof}
      This follows directly from the volume comparison and the Li-Yau estimate. 
    \end{proof}

    Note that (\ref{eqn:HA10_2}) can be rewritten as 
    \begin{align*}
        (t-s)^{\frac{n}{2}} H(x,t;y,s) \leq C  \left\{\frac{V_{x,t}(\sqrt{t-s})}{(\sqrt{t-s})^n}\right\}^{-1} e^{-\frac{d_t^2(x,y)}{C(t-s)}}, 
    \end{align*}
    which has the advantage that both sides are scaling invariant.  Without loss of generality, we can rescale the space by $\lambda=4\Lambda_0 \xi^{-2} (t-s)^{-1}$ and shift the time-slice $t$ to $0$.  Thus, we obtain an evolving manifold $\{(M, \tilde{g}(\theta)), -\Lambda_0 \xi^{-2} \leq \theta \leq 0\}$ such that
    \begin{align*}
      \sup_{M \times [-4, 0]} |Rm|_{\tilde{g}} \leq   \sup_{M \times [-\Lambda_0 \xi^{-2}, 0]} |Rm|_{\tilde{g}} \leq \xi^2,  
    \end{align*}
    where we used the fact that $\Lambda_0 \xi^{-2} >>4$. 
    Using the rescaling property and volume comparison, the heat kernel $\tilde{H}$ satisfies the estimate
    \begin{align*}
        \tilde{H}(x,0;y,-4) \leq C V_{x,0}(1)^{-1} e^{-\frac{d_0^2(x,y)}{C}}. 
    \end{align*}
    We claim that similar estimates hold in a neighborhood of $(x, 0)$: 
     \begin{align}
        \sup_{w \in B_{10}^{\tilde{g}(0)}(x), -1 \leq t \leq 0} 
        \tilde{H}(w,t;y,-4) \leq C V_{x,0}(1)^{-1} e^{-\frac{d_0^2(x,y)}{C}}  =: K.
        \label{eqn:HA10_3}
    \end{align}
    In fact, the triangle inequality implies that
   \begin{align*}
       d_0^2(w,y)\geq \frac{1}{2}d_0^2(x,y)-d_0^2(w,x)\geq \frac12 d_0^2(x,y)-100.
   \end{align*}
   Adjusting $C$, we have
   \begin{align*}
       e^{-\frac{d_0^2(w,y)}{Cs}}\leq e^{-\frac{\frac12 d_0^2(x,y)-100}{Cs}}\leq Ce^{-\frac{d_0^2(x,y)}{Ct}}, 
       \quad \forall \;   w \in B_{10}^{\tilde{g}(0)}(x), -1 \leq t \leq 0. 
   \end{align*}
   Thus, we finished the proof of (\ref{eqn:HA10_3}). 
    
    Fix $x$ and define $\varphi$ as the exponential map from $T_x M$ to $M$, with respect to the metric $\tilde{g}(0)$:
    \begin{align*}
        \varphi(v) := Exp_{x}^{\tilde{g}(0)}(v), \quad \; \forall\; v \in T_x M \simeq \mathbb{R}^n. 
    \end{align*}
    Using this smooth map $\varphi$, we can pull back the evolving metrics $\tilde{g}$ to $T_x M$ as
    \begin{align*}
        \hat{g}(t) :=\varphi^* \tilde{g}(t), 
    \end{align*}
    Since $\xi$ is chosen sufficiently small (cf. (\ref{eqn:HC07_3})),  $\log |d \varphi|$ is uniformly bounded. 
    Consider $(B_{r}(0), \hat{g}(t))$ with $r=\frac{1}{100n\pi \xi}>>100$. 
    We can also assume that the $W_p^{k}$-harmonic radius at $0$ is at least $100$.  
    Thus,  we have the harmonic diffeomorphism $\psi: B_{10}(0) \to \mathbb{R}^n$.  
Then 
    \begin{align*}
         \psi: B_{10}^{\varphi^* \tilde{g}(0)}(x) \mapsto \mathbb{R}^n. 
    \end{align*}
    According to the definition of harmonic radius, we have 
    \begin{align}
       B_{5\sqrt{2}}(0) \subset  \psi (B_{10}^{\varphi^* \tilde{g}(0)}(x)) \subset B_{10\sqrt{2}}(0). 
    \label{eqn:HA10_4}   
    \end{align}
   Then it is not hard to see that 
   \begin{align*}
       \bar{g}(t) := \psi_* \varphi^* \tilde{g}(t)
   \end{align*}
   is well-defined on $\Omega  =B_2(0) \times [-4, 0]$.  Furthermore, we have
   \begin{align}
      \label{cdn:HB07_3}
    (\Omega, \bar{g}) \; \textit{is a model space-time in the sense of Definition~\ref{dfn:HA10_1}}.   
   \end{align}

   Fix $y$ and define
   \begin{align*}
       \bar{H}(z,t) :=\psi_* \varphi^* \tilde{H}(z,t;y,-4)=\tilde{H}(\varphi \psi^{-1}(z),t;y,-4).
   \end{align*}
   Then $\bar{H}$ is a heat solution on $(\Omega, \bar{g})$.  It follows from (\ref{eqn:HA10_3}) and (\ref{eqn:HA10_4}) that 
   \begin{align}
       \sup_{\Omega} \bar{H} \leq K. 
   \label{eqn:HA10_5}    
   \end{align}
   Thus, we can apply Proposition~\ref{prn:HA04_1} to obtain
   \begin{align}
        \|\bar{H}\|_{C^{2+\alpha,\frac{2+\alpha}{2}}(\Omega')} \leq C(n) K, 
    \label{eqn:HA10_6}    
    \end{align}
    which then implies that (cf. Corollary~\ref{cly:HA03_6})
    \begin{align}
        \sup_{(x,t) \in \Omega'} \left\{ |\nabla \bar{H}|_{\bar{g}} + |\nabla \nabla \bar{H}|_{\bar{g}}
        +[\nabla \nabla \bar{H}]_{\bar{g},\alpha}^{(\frac12)}
        +|\dot{\bar{H}}|\right\} \leq
         C(n) K. 
    \label{eqn:HA10_7}    
    \end{align}
    In particular, we have 
    \begin{align}
         \left. \left\{ |\nabla \bar{H}|_{\bar{g}} + |\nabla \nabla \bar{H}|_{\bar{g}}
        +[\nabla \nabla \bar{H}]_{\bar{g},\alpha}^{(\frac12)}
        +|\dot{\bar{H}}|\right\} \right|_{(0,0)}\leq
         C(n) K. 
    \label{eqn:HA10_8}    
    \end{align}
    Recalling the definition of $K$ in (\ref{eqn:HA10_3}) and rescaling back to the original $g$, we obtain the following Theorem. 

\begin{theorem}\label{thm3.2}
        Let $\{(M,g(t)), -1 \leq t \leq 0\}$ be an evolving manifold that satisfies (\ref{eqn:HA02_1}).
        Let  $H=H(x,t;y,s)$ be the heat kernel of the Laplace–Beltrami operator.
        Define
        \begin{align}
        \label{eqn:HB14_2}
          \theta :=t-s, \quad   \rho_{\theta} :=\frac{\xi \sqrt{\theta}}{4\sqrt{\Lambda_0}}. 
        \end{align}
        Then there exists a positive constant $C=C(n, \alpha)$ such that 
        \begin{equation}
            H+\theta^{\frac{1}{2}}|\nabla_x H|+ \theta|\nabla^2_x H|+\theta^{1+\frac{\alpha}{2}}[\nabla^2_x H]_{\alpha}^{(\rho_{\theta})}+\theta^{2}|\partial_t\nabla^2_x H|\leq CV_{x,t}(\sqrt{\theta})^{-1}e^{-\frac{d_t^2(x,y)}{C\theta}}
        \end{equation}
        for any $x,y\in M$ and $-1\leq s<t\leq 0$.
\end{theorem}

Suppose $\Psi(x,t;y,s)$ is the heat kernel of the Lichnerowicz Laplacian operator on symmetric $(0,2)$-tensor.  By definition, $\Psi(x,t;y,s)$ is an endomorphism from $Sym^2(T_y^*M)$ to $Sym^2(T_x^*M)$. 
Fixing $(y,s)$ and $E_y \in Sym^2(T_y^*M)$, we know that $\Psi(x,t;y,s) E_y$ is a symmetric function of $(0,2)$ that solves the Lichnerowicz heat equation.  Then
\begin{align*}
    \nabla_x^{g(t)} \left\{ \Psi(x,t;y,s) E_y \right\} 
    =\left\{ \nabla_x^{g(t)}  \Psi(x,t;y,s) \right\} E_y. 
\end{align*}
As usual, we use the canonical metric to identify $T^*M$ with $TM$.  For example, we use $g(t)$ at $(x,t)$ and use $g(s)$ at $(y,s)$.  Then we omit the metric we use if it is obvious.  
Thus, $\nabla_x \Psi(x,t;y,s)$ is an endomorphism from $Sym^2(T_y^*M)$ to $Sym^2(T_x^*M) \otimes T_x^* M$.  If we fix an element $F_x \in Sym^2(T_x^*M) \otimes T_x^* M$, then 
$F_x \left( \nabla_x \Psi(x,t;y,s) \right)$ is a symmetric $(0,2)$-tensor valued function with variable $(y,s)$. Then 
\begin{align*}
    \nabla_y \left\{ F_x \left( \nabla_x \Psi(x,t;y,s) \right) \right\}
    &=\nabla_y^{g(s)} \left\{ F_x \left( \nabla_x^{g(t)} \Psi(x,t;y,s) \right) \right\}\\
    &=F_x \left(\nabla_y^{g(s)} \nabla_x^{g(t)} \Psi(x,t;y,s) \right) 
     \in Sym^2(T_y^*M) \otimes T_y^*M. 
\end{align*}
Therefore,  we have an endomorphism: 
\begin{align*}
 Sym^2(T_y^*M) \otimes T_y^*M   \longright{\nabla_y \nabla_x \Psi(x,t;y,s)}
 Sym^2(T_x^*M) \otimes T_x^*M.    
\end{align*}
Fix $(y, s)$ and a unit element $F_y \in Sym^2(T_y^*M) \otimes T_y^*M$.
Note that 
$\nabla_y \nabla_x \Psi(x,t;y,s) F_y$ is a smooth section of $Sym^2(T^*M) \otimes T^*M$ with variable $(x,t)$.

\begin{theorem}
\label{thm3.3}
        Let $(M,g)$ be a complete $n$-dimensional Riemannian manifold satisfying (\ref{eqn:HA02_1}). 
        Let  $\Psi=\Psi(x,y,t-s)=\Psi(x,t;y,s)$ be the heat kernel of (Lichnerowicz or normal) Laplacian operator on symmetric $(0,2)$ tensor.
        Let $\theta=t-s$ and $\rho_{\theta}=\frac{\xi \sqrt{\theta}}{4\sqrt{\Lambda_0}}$.
        Then there exists a positive constant $C=C(n,\alpha)$ such that 
        \begin{align}
        \begin{split}
            &\quad |\Psi|+
            \theta^{\frac{1}{2}}|\nabla_x \Psi|+\theta|\nabla_x\nabla_y \Psi|+\theta^{1+\frac{\alpha}{2}}[\nabla_x\nabla_y \Psi]_\alpha^{(\rho_{\theta})}
            +\theta^{2}|\nabla_x\nabla_y \dot{\Psi}|\\
            &\leq C V_x(\sqrt{\theta})^{-1}e^{-\frac{d^2(x,y)}{C\theta}}
        \end{split}    
        \label{eqn:SL29_8}
        \end{align}
        for any $x,y\in M$ and $-1<s<t\leq 0$.
\end{theorem}

\begin{proof}
By time shifting, we may assume $s=0$. Then $0<\theta=t-s=t\leq 1$.
By Kato's inequality, we have $\partial_t|\Psi|\leq \Delta |\Psi|+C|\Psi|$.  Thus, the Li-Yau estimate implies that
\begin{align}
     |\Psi| \leq C V_x(\sqrt{t})^{-1}e^{-\frac{d^2(x,y)}{Ct}}.
\label{eqn:SL06_1}     
\end{align}
Furthermore, we have the following stronger estimate
 \begin{align}
     \|\Psi\|_{L^{\infty}(\Omega(x,t))} +\|\Psi\|_{L^{\infty}(\Omega(y,t))} 
     \leq CV_x(\sqrt{t})^{-1}e^{-\frac{d^2(x,y)}{Ct}}. 
 \label{eqn:SL29_6}    
 \end{align}
 Fix $y$ and regard $|\Psi|(x,y,t)$ as a function of $x$ and $t$.
 By Corollary~\ref{cly:HA03_1} and scaling, we have 
 \begin{align}
   t \|\dot{\Psi}\|_{L^{\infty}(\Omega(x,t))}  
   + t^{\frac12}\|\nabla_x \Psi\|_{L^{\infty}(\Omega(x,t))}  
   \leq   C \|\Psi\|_{L^{\infty}(\Omega(x,t))}. 
 \label{eqn:SL29_2}  
 \end{align}
 Swapping the role of $x$ and $y$, the same argument yields that
  \begin{align}
   t \|\dot{\Psi}\|_{L^{\infty}(\Omega(y,t))}  
   + t^{\frac12}\|\nabla_y \Psi\|_{L^{\infty}(\Omega(y,t))}  
   \leq   C \|\Psi\|_{L^{\infty}(\Omega(y,t))}. 
 \label{eqn:SL29_5}  
 \end{align}

 Then we fix $x$ and $V \in T_x M$ to be a unit vector.  Then $\Phi(y,t) :=\nabla_{x,V} \Psi(x,y,t)$ is a tensor valued function of $(y, t)$, satisfying the Lichnerowicz heat equation. So does $\dot{\Phi}$. By scaling, we know
 \begin{align*}
  t\|\dot{\Phi}\|_{L^{\infty}(\Omega(y,t))} 
     + t^{\frac12}\|\nabla_y \Phi\|_{L^{\infty}(\Omega(y,t))} 
   \leq   C \|\Phi\|_{L^{\infty}(\Omega(y,t))}
   \leq   C t^{-\frac12} \|\Psi\|_{L^{\infty}(\Omega(y,t))}, 
 \end{align*}
 where we apply (\ref{eqn:SL29_5}) in the last step.  By the arbitrary choice of $V$, we have
 \begin{align}
     t \|\nabla_x \nabla_y \Psi \|_{L^{\infty}(\Omega(y,t))} \leq C \|\Psi\|_{L^{\infty}(\Omega(y,t))}. 
 \label{eqn:SL29_3}    
 \end{align}
 Applying the same argument on $\dot{\Psi}=\partial_t \Psi$ and using (\ref{eqn:SL29_5}), we obtain 
 \begin{align*}
     t \|\nabla_x \nabla_y \dot{\Psi} \|_{L^{\infty}(\Omega(y,t))} 
     \leq C \|\dot{\Psi}\|_{L^{\infty}(\Omega(y,t))}
     \leq C t^{-1} \|\Psi\|_{L^{\infty}(\Omega(y,t))},   
 \end{align*}
 which is exactly
  \begin{align}
     t^2 \|\nabla_x \nabla_y \dot{\Psi} \|_{L^{\infty}(\Omega(y,t))} 
     \leq C \|\Psi\|_{L^{\infty}(\Omega(y,t))}.  
 \label{eqn:SL29_4}    
 \end{align}

 It follows from Definition~\ref{dfn:HB14_1} and Corollary~\ref{cly:HA03_1} that 
 \begin{align*}
     t^{\frac{1+\alpha}{2}}[\nabla_x\nabla_y \Psi]_\alpha^{(\frac12)} \leq C \|\nabla_y \Psi \|_{L^{\infty}(\Omega(x,t))}
     \leq C t^{-\frac12}\|\Psi\|_{L^{\infty}(\Omega(x,t))}, 
 \end{align*}
 which yields
 \begin{align}
     t^{1+\frac{\alpha}{2}}[\nabla_x\nabla_y \Psi]_\alpha^{(\frac12)} 
     \leq C \|\Psi\|_{L^{\infty}(\Omega(x,t))}. 
 \label{eqn:SL29_7}    
 \end{align}

 Combining (\ref{eqn:SL29_2}), (\ref{eqn:SL29_5}), (\ref{eqn:SL29_4}) and (\ref{eqn:SL29_7}), we arrive at
 \begin{align*}
            &\quad t^{\frac{1}{2}}|\nabla_x \Psi|+t|\nabla_x\nabla_y \Psi|+t^{1+\frac{\alpha}{2}}[\nabla_x\nabla_y \Psi]_\alpha^{(\frac12)}+t^{2}|\nabla_x\nabla_y \dot{\Psi}|\\
            &\leq C \left( \|\Psi\|_{L^{\infty}(\Omega(x,t))} +\|\Psi\|_{L^{\infty}(\Omega(y,t))} \right).
\end{align*}
Plugging (\ref{eqn:SL29_6}) into the above inequality, we obtain (\ref{eqn:SL29_8}). 
\end{proof}

If the underlying space-time is a Ricci flow solution, we have better regularity. 
\begin{theorem}
\label{thm:HA02_5}
        Let $\{(M,g(t)), -1 \leq t \leq 0\}$ be a Ricci flow that satisfies (\ref{eqn:HA02_1}).
        Let  $\Psi=\Psi(x,t;y,s)$ be the heat kernel of the (Lichnerowicz or normal) Laplacian operator on symmetric $(0,2)$ tensor and $\theta=t-s>0$.
        Then there exists a positive constant $C=C(n,k)$ such that 
        \begin{align}
           \sum_{0 \leq j+l \leq k} \theta^{\frac{j+l}{2}}|\nabla_x^j\nabla_y^l \Psi|
            \leq C V_x(\sqrt{\theta})^{-1}e^{-\frac{d^2(x,y)}{C\theta}}
        \end{align}
        for any $x,y\in M$ and $-1 \leq s<t\leq 0$.
\end{theorem}

\section{Singular integrals and maximal functions}
\label{sec5}

In this section, we shall apply the regularity estimate in Section~\ref{sec3} and the heat kernel estimate in Section~\ref{sec4} to study the Calder\'on-Zygmund integral operator (cf. Definition~\ref{def2.1}) $\mathcal{T}$. 
The purpose of this section is to estimate the variation of $\mathcal{T}f$ and $\mathbf{M}(|\mathcal{T}f|)$ by $\mathbf{M}(f)$, where $\mathbf{M}$ means the maximal function. 
The explicit formulation can be found in Proposition~\ref{prn:SK27_1} and Proposition~\ref{prn:SK27_2}. 

 \begin{definition}
 \label{def2.1}		
 Let $\mathcal{M}=M\times [-1, 0]$ be an evolving manifold satisfying $|Rm| \leq \Lambda_0$.
 Let $\mathcal{E} \to M,  \mathcal{F} \to M$ be smooth vector bundles on $M$. 
 \begin{itemize}
 
        \item  A smooth section $\mathcal{K}: (\mathcal{M} \times \mathcal{M})^+ \to \mathcal{E} \boxtimes \mathcal{F}^*$ is called a $C_1$-Calder\'on-Zygmund
		kernel if it satisfies
        \begin{align}
          \theta |\CK(x,t;y,s)| + \theta^{1+\frac{\alpha}{2}} [\CK(x,t;y,s)]_\alpha^{(\rho_{\theta})}
          +\theta^2|\dot{\CK}(x,t;y,s)| 
          \leq C_1V^{-1}_{x,t}(\sqrt{\theta}) e^{-\frac{d^2(x,y)}{C_1 \theta}} 
        \label{eqn:SJ06_1}  
        \end{align}
        for some positive constants $C_1$ and $\alpha \in (0,1)$.  Here $\theta=t-s$ and $\rho_{\theta}=\frac{\xi \sqrt{\theta}}{4\sqrt{\Lambda_0}}$. 
        
        \item  $\mathcal{T}$ is called a singular integral operator associated with $\mathcal{K}$ if
		\begin{align}
		    \mathcal{T}f(x,t) :=\int_{-1}^t \int_M \CK(x,t;y,s)f(y,s) dyds
		\end{align}
		for any  $f\in C^{\infty}(\mathcal{M}, \mathcal{F})$.
		
    \item $\mathcal{T}$ is called a $(C_1, C_2)$-Calder\'on-Zygmund operator if its operator norm on $L^2$ is uniformly bounded by $C_2$. 
        In other words, 
        \begin{align}
            \|\mathcal{T}f\|_{L^2(\mathcal{M})}\leq C_2\|f\|_{L^2(\mathcal{M})}
            \label{eqn:HA19_1}
        \end{align}
		for all smooth compactly supported $f$.

    \item  $\mathcal{T}$ is called a Calder\'on-Zygmund operator if there exist $C_1$ and $C_2$ such that $\mathcal{T}$ is a $(C_1,C_2)$-Calder\'on-Zygmund operator. 

    \end{itemize}
 \end{definition}
    
 We fix a point $q\in M$ and define
 \begin{align}
     P_r(q,t) :=B_r^{g(t)}(q) \times [t-r^2,\min\{t+r^2, 0\}]. 
 \label{eqn:HA24_5}    
 \end{align}
 Denote 
 \begin{align}
     \mathcal{Q}':=B_{\frac18}(q,0) \times [-\frac{1}{64}, 0], 
     \quad
     \mathcal{Q}:=B_{\frac14}(q,0) \times [-\frac{1}{16}, 0]. 
 \label{eqn:HA24_6}    
 \end{align}
 Fix 
 \begin{align*}
     X_0=(x_0, t_0) \in B_{\frac14}(q,0) \times [-\frac{1}{64}, 0]. 
 \end{align*}
 Recall the definition of maximal function
 \begin{align}
		\mathbf{M}(f)(x,t)
        :=\sup_{0<r<\frac12}\frac{1}{|P_r(x,t)|}\int_{P_r(x,t)}|f|
        dzds=\sup_{0<r<\frac12}\fint_{P_r(x,t)}|f|dzds.
 \label{eqn:HA18_3}       
 \end{align}
 Note that $P_r(x_0,t_0) \subset M \times [-1, 0]$ and
 \begin{align*}
   \frac{1}{C} V_{y,0}(r)=\frac{1}{C}|B_r^{g(0)}(y)|_{d\mu_{g(0)}} \leq  V_{y,t}(r)=|B_r^{g(t)}(y)|_{d\mu_{g(t)}} \leq C |B_r^{g(0)}(y)|_{d\mu_{g(0)}}=C V_{y,0}(r). 
 \end{align*}
 Due to this equivalence, we shall just denote $V_{y,0}(r)$ by $V_{y}(r)$ and we have
 \begin{align}
      \frac{1}{C} V_{y}(r) \leq  V_{y,t}(r) \leq C V_{y}(r). 
      \label{eqn:HA18_1}
 \end{align}
 We use $d=d_{g(0)}$. Then for any $x,y \in M$ we have
 \begin{align}
     \frac{1}{C} d(x,y) \leq d_{g(t)}(x,y) \leq C d(x,y).
     \label{eqn:HA18_2}
 \end{align}

\begin{lemma}
Suppose $Y_0=(y_0,t_0)$, $Y=(y,t_0) \in P_{4r}(Y_0)=P_{4r}(y_0, t_0)$. 
Suppose $\mathrm{supp} f \in \mathcal{Q}' \setminus P_{5r}(Y_0)$ (See Figure~\ref{fig1}).
Then we have
    \begin{align}
        |\mathcal{T}f(Y)-\mathcal{T}f(Y_0)|< C  \mathbf{M}(|f|)(Y_0) \cdot \left( \frac{d(y,y_0)}{r}\right)^{\alpha}. 
    \label{eqn:HA19_4}    
    \end{align}
\end{lemma}

\begin{figure}[htbp]
    \centering
   \resizebox{0.3\textwidth}{!}{

\tikzset{every picture/.style={line width=0.75pt}} 

\begin{tikzpicture}[x=0.75pt,y=0.75pt,yscale=-1,xscale=1]

\draw  [fill={rgb, 255:red, 128; green, 128; blue, 128 }  ,fill opacity=1 ] (116.5,50) -- (577.5,50) -- (577.5,511) -- (116.5,511) -- cycle ;
\draw  [color={rgb, 255:red, 128; green, 128; blue, 128 }  ,draw opacity=1 ][fill={rgb, 255:red, 155; green, 155; blue, 155 }  ,fill opacity=1 ] (185,117.5) -- (506,117.5) -- (506,438.5) -- (185,438.5) -- cycle ;
\draw  [color={rgb, 255:red, 255; green, 255; blue, 255 }  ,draw opacity=1 ][fill={rgb, 255:red, 255; green, 255; blue, 255 }  ,fill opacity=1 ][line width=1.5]  (249.75,182.25) -- (441.25,182.25) -- (441.25,373.75) -- (249.75,373.75) -- cycle ;
\draw  [fill={rgb, 255:red, 74; green, 74; blue, 74 }  ,fill opacity=1 ] (412,278.5) .. controls (412,276.01) and (414.01,274) .. (416.5,274) .. controls (418.99,274) and (421,276.01) .. (421,278.5) .. controls (421,280.99) and (418.99,283) .. (416.5,283) .. controls (414.01,283) and (412,280.99) .. (412,278.5) -- cycle ;
\draw  [fill={rgb, 255:red, 74; green, 74; blue, 74 }  ,fill opacity=1 ] (341,278) .. controls (341,275.51) and (343.01,273.5) .. (345.5,273.5) .. controls (347.99,273.5) and (350,275.51) .. (350,278) .. controls (350,280.49) and (347.99,282.5) .. (345.5,282.5) .. controls (343.01,282.5) and (341,280.49) .. (341,278) -- cycle ;
\draw  [dash pattern={on 0.84pt off 2.51pt}]  (345.5,540.5) -- (345.5,17.5) ;
\draw [shift={(345.5,15.5)}, rotate = 90] [color={rgb, 255:red, 0; green, 0; blue, 0 }  ][line width=0.75]    (21.86,-6.58) .. controls (13.9,-2.79) and (6.61,-0.6) .. (0,0) .. controls (6.61,0.6) and (13.9,2.79) .. (21.86,6.58)   ;
\draw  [dash pattern={on 0.84pt off 2.51pt}]  (40,279.5) -- (649,276.51) ;
\draw [shift={(651,276.5)}, rotate = 179.72] [color={rgb, 255:red, 0; green, 0; blue, 0 }  ][line width=0.75]    (21.86,-6.58) .. controls (13.9,-2.79) and (6.61,-0.6) .. (0,0) .. controls (6.61,0.6) and (13.9,2.79) .. (21.86,6.58)   ;

\draw (405,292.4) node [anchor=north west][inner sep=0.75pt]  [font=\fontsize{2.65em}{3.18em}\selectfont]  {$Y$};
\draw (296,297.4) node [anchor=north west][inner sep=0.75pt]  [font=\fontsize{2.65em}{3.18em}\selectfont]  {$Y_{0}$};
\draw (370,2.4) node [anchor=north west][inner sep=0.75pt]  [font=\fontsize{2.65em}{3.18em}\selectfont]  {$t$};
\draw (610,299.4) node [anchor=north west][inner sep=0.75pt]  [font=\fontsize{2.65em}{3.18em}\selectfont]  {$M$};
\draw (342,192.4) node [anchor=north west][inner sep=0.75pt]  [font=\fontsize{2.65em}{3.18em}\selectfont]  {$P_{4r}( Y_{0})$};
\draw (349,122.4) node [anchor=north west][inner sep=0.75pt]  [font=\fontsize{2.65em}{3.18em}\selectfont]  {$P_{5}{}_{r}( Y_{0})$};
\draw (515,57.4) node [anchor=north west][inner sep=0.75pt]  [font=\fontsize{2.65em}{3.18em}\selectfont]  {$Q'$};
\draw (131,58.4) node [anchor=north west][inner sep=0.75pt]  [font=\fontsize{2.65em}{3.18em}\selectfont]  {$\mathrm{supp} \ f$};

\end{tikzpicture}

}
    \caption{}
    \label{fig1}
\end{figure}

\begin{proof}
  We can connect $y$ and $y_0$ by a shortest geodesic segment $\gamma$.  We denote the parallel transportation from $y$ to $y_0$ along $\gamma$ by $\tau_{g(t_0)}(y_0, y)$. Since there is no ambiguity of
  the underlying metric, we use simply $\tau (y_0, y)$. 

  For simplicity, we denote $Y_0=(y_0, t_0)$ and $Y=(y, t_0)$. Then
  \begin{align}
  \begin{split}
      &\quad \left||\mathcal{T}f|(Y)-|\mathcal{T}f|(Y_0)|=||\tau(y_0,y) \mathcal{T}f(Y)|-|\mathcal{T}f(Y_0)| \right|\\
      &\leq \left|\tau(y_0,y) \mathcal{T}f(Y)-\mathcal{T}f(Y_0) \right|\\
      &=\left| \int_{-1}^t\int_M \big\{\tau(y_0,y)\mathcal{K}(y,t;z,s)-\mathcal{K}(y_0,t;z,s)\big\} f(z,s)dzds  \right|\\
      &\leq \int_{-1}^t\int_M \big|\tau(y_0,y)\mathcal{K}(y,t;z,s)-\mathcal{K}(y_0,t;z,s)\big| |f|(z,s)dzds. 
  \end{split} 
  \label{eqn:HB23_0}
  \end{align}
  The term $\big|\tau(y_0,y)\mathcal{K}(y,t;z,s)-\mathcal{K}(y_0,t;z,s)\big|$ shall be estimated depending on the size of $d(y_0, y)$.
  If $d(y_0,y) \leq \sqrt{t-s}$, we have 
  \begin{align}
      \big|\tau(y_0,y)\mathcal{K}(y,t;z,s)-\mathcal{K}(y_0,t;z,s)\big|
      \leq C \cdot  \left(\frac{d(y,y_0)}{\sqrt{t-s}} \right)^{\alpha} \cdot \frac{e^{-\frac{d^2(y_0,z)}{C(t-s)}}}{V_{y_0}(\sqrt{t-s}) \cdot (t-s)}. 
  \label{eqn:HB23_1}    
  \end{align}
  If $d(y_0, y)>\sqrt{t-s}$, we can connect $y_0$ and $y$ by a shortest geodesic $\gamma$ and set $y_1, y_2, \cdots, y_L=y$ along $\gamma$ such that the distance between $y_k$ and $y_{k+1}$ is comparable to $\sqrt{t-s}$. 
   Using volume comparison,  $V_{y_k}^{-1}(\sqrt{t-s}) e^{-\frac{d^2(y_k,z)}{C(t-s)}}$ are all comparable to $V_{y_0}^{-1}(\sqrt{t-s}) e^{-\frac{d^2(y_0,z)}{C(t-s)}}$. 
  By adjusting $C$ slightly if necessary, we have
  \begin{align*}
      &\quad \big|\tau(y_0,y)\mathcal{K}(y,t;z,s)-\mathcal{K}(y_0,t;z,s)\big|\\
      &\leq \sum_{k=1}^{L} \big|\tau(y_{k-1}, y_k)\mathcal{K}(y_k,t;z,s)-\mathcal{K}(y_{k-1},t;z,s) \big|\\
      &\leq C \cdot L \cdot  V_{y_0}^{-1}(\sqrt{t-s}) \cdot (t-s)^{-1} 
       \cdot e^{-\frac{d^2(y_0,z)}{C(t-s)}}. 
  \end{align*}
  Since $L \sim \frac{d(y_0,y)}{\sqrt{t-s}}$, it follows that
  \begin{align}
   \big|\tau(y_0,y)\mathcal{K}(y,t;z,s)-\mathcal{K}(y_0,t;z,s)\big|
   \leq C \cdot  \frac{d(y,y_0)}{\sqrt{t-s}}  \cdot \frac{e^{-\frac{d^2(y_0,z)}{C(t-s)}}}{V_{y_0}(\sqrt{t-s}) \cdot (t-s)}. 
  \label{eqn:HB23_2} 
  \end{align}
  Plugging (\ref{eqn:HB23_2}) and (\ref{eqn:HB23_1}) into (\ref{eqn:HB23_0}) yields that
  \begin{align}
  \begin{split}
   &\quad \left||\mathcal{T}f|(Y)-|\mathcal{T}f|(Y_0)| \right|\\
   &\leq C \iint_{M \times [-1,t]}  \left\{\frac{d(y,y_0)}{\sqrt{t-s}} 
   + \left(\frac{d(y,y_0)}{\sqrt{t-s}} \right)^{\alpha} \right\} \cdot \frac{e^{-\frac{d^2(y_0,z)}{C(t-s)}}}{V_{y_0}(\sqrt{t-s}) \cdot (t-s)} |f|\\
   &=C \frac{d(y,y_0)}{r} \cdot \iint_{B_{1}(y_0) \times [-1, t]} 
   \frac{r}{V_{y_0}(\sqrt{t-s})(t-s)^{\frac{3}{2}}}
            e^{-\frac{d^2(y_0,z)}{C(t-s)}}|f|\\
   &\quad + C \left( \frac{d(y,y_0)}{r}\right)^{\alpha} \cdot \iint_{B_{1}(y_0) \times [-1, t]} 
   \frac{r^{\alpha}}{V_{y_0}(\sqrt{t-s})(t-s)^{1+\frac{\alpha}{2}}}
            e^{-\frac{d^2(y_0,z)}{C(t-s)}}|f|. 
  \end{split}     
  \label{eqn:HB23_4}
  \end{align}

 
 We claim the following inequalities holds:
 \begin{align}
   &\iint_{B_{1}(y_0) \times [-1, t]} 
   \frac{r^{\alpha}}{V_{y_0}(\sqrt{t-s})(t-s)^{1+\frac{\alpha}{2}}}
            e^{-\frac{d^2(y_0,z)}{C(t-s)}}|f|
   \leq C \cdot \mathbf{M}(|f|)(Y_0), \label{eqn:HA19_5} \\
   & \iint_{B_{1}(y_0) \times [-1, t]} 
   \frac{r}{V_{y_0}(\sqrt{t-s})(t-s)^{\frac{3}{2}}}
            e^{-\frac{d^2(y_0,z)}{C(t-s)}}|f|
         \leq C \cdot \mathbf{M}(|f|)(Y_0). 
    \label{eqn:HB23_5}     
 \end{align}

 The proof of (\ref{eqn:HA19_5}) and (\ref{eqn:HB23_5}) are almost the same. 
 We shall focus on the proof of (\ref{eqn:HA19_5}). 
 Note that the support of $f$ is in $\mathcal{Q}' \setminus P_{25r}(Y_0)$. 
 Define 
    \begin{align*}
         A_i := P_{5^{i+1}r}(Y_0)\setminus P_{5^{i}r}(Y_0), 
    \end{align*}
 Then $\mathrm{supp} f \subset \bigcup_{i=1}^N A_i$ for some positive integer $N$ such that $5^{N} r \leq 1 < 5^{N+1}r$.
 Decompose (See Figure~\ref{fig2})
   \begin{align*}
     A_i = \underbrace{\{(z,s) \in A_i| \frac{r}{\sqrt{t-s}} \leq 5^{1-i}\}}_{A_{i,I}} 
     \cup \underbrace{\{(z,s) \in A_i|\frac{r}{\sqrt{t-s}} > 5^{1-i} \}}_{A_{i,II}}.
   \end{align*}
   
\begin{figure}[htbp]
    \centering
    \resizebox{0.3\textwidth}{!}{
        
\tikzset{every picture/.style={line width=0.75pt}} 
\begin{tikzpicture}[x=0.75pt,y=0.75pt,yscale=-1,xscale=1]

\draw  [fill={rgb, 255:red, 128; green, 128; blue, 128 }  ,fill opacity=1 ][line width=0.75]  (89,52) -- (567,52) -- (567,530) -- (89,530) -- cycle ;
\draw  [fill={rgb, 255:red, 255; green, 255; blue, 255 }  ,fill opacity=1 ] (196,52) -- (461,52) -- (461,317) -- (196,317) -- cycle ;
\draw  [fill={rgb, 255:red, 0; green, 0; blue, 0 }  ,fill opacity=1 ] (320,52.5) .. controls (320,50.57) and (321.57,49) .. (323.5,49) .. controls (325.43,49) and (327,50.57) .. (327,52.5) .. controls (327,54.43) and (325.43,56) .. (323.5,56) .. controls (321.57,56) and (320,54.43) .. (320,52.5) -- cycle ;
\draw  [fill={rgb, 255:red, 155; green, 155; blue, 155 }  ,fill opacity=1 ] (89,244) -- (196,244) -- (196,52) -- (89,52) -- cycle ;
\draw  [fill={rgb, 255:red, 155; green, 155; blue, 155 }  ,fill opacity=1 ] (461,52) -- (567,52) -- (567,245) -- (461,245) -- cycle ;
\draw [color={rgb, 255:red, 0; green, 0; blue, 0 }  ,draw opacity=1 ] [dash pattern={on 4.5pt off 4.5pt}]  (325,583) -- (322.01,12) ;
\draw [shift={(322,10)}, rotate = 89.7] [color={rgb, 255:red, 0; green, 0; blue, 0 }  ,draw opacity=1 ][line width=0.75]    (21.86,-6.58) .. controls (13.9,-2.79) and (6.61,-0.6) .. (0,0) .. controls (6.61,0.6) and (13.9,2.79) .. (21.86,6.58)   ;
\draw  [dash pattern={on 4.5pt off 4.5pt}]  (7,53.5) -- (638,51.51) ;
\draw [shift={(640,51.5)}, rotate = 179.82] [color={rgb, 255:red, 0; green, 0; blue, 0 }  ][line width=0.75]    (21.86,-6.58) .. controls (13.9,-2.79) and (6.61,-0.6) .. (0,0) .. controls (6.61,0.6) and (13.9,2.79) .. (21.86,6.58)   ;

\draw (362,386.4) node [anchor=north west][inner sep=0.75pt]  [font=\fontsize{2.65em}{3.18em}\selectfont]  {$P_{5^{i+1} r}^{-}( Y_{0})$};
\draw (321,207.4) node [anchor=north west][inner sep=0.75pt]  [font=\fontsize{2.65em}{3.18em}\selectfont]  {$P_{5^{i} r}^{-}( Y_{0})$};
\draw (121,227.4) node [anchor=north west][inner sep=0.75pt]  [font=\fontsize{2.65em}{3.18em}\selectfont]  {$A_{i}$};
\draw (343,66.4) node [anchor=north west][inner sep=0.75pt]  [font=\fontsize{2.65em}{3.18em}\selectfont]  {$Y_{0}$};
\draw (211,378.4) node [anchor=north west][inner sep=0.75pt]  [font=\fontsize{2.65em}{3.18em}\selectfont]  {$A_{i,I}$};
\draw (469,152.4) node [anchor=north west][inner sep=0.75pt]  [font=\fontsize{2.65em}{3.18em}\selectfont]  {$A_{i}{}_{,}{}_{I}{}_{I}$};
\draw (109,161.4) node [anchor=north west][inner sep=0.75pt]  [font=\fontsize{2.65em}{3.18em}\selectfont]  {$A_{i,II}$};
\draw (346,1.4) node [anchor=north west][inner sep=0.75pt]  [font=\fontsize{2.65em}{3.18em}\selectfont]  {$t$};
\draw (593,78.4) node [anchor=north west][inner sep=0.75pt]  [font=\fontsize{2.65em}{3.18em}\selectfont]  {$M$};

\end{tikzpicture}

    }
    \caption{}
    \label{fig2}
\end{figure}

   Then we can rewrite the left hand side of (\ref{eqn:HA19_5}) as 
   \begin{align}
      \sum_{i=1}^{N} \left\{ \iint_{A_{i,I}} + \iint_{A_{i,II}}  \right\} 
       V_{y_0}^{-1}(\sqrt{t-s})(t-s)^{-1} \cdot \left(\frac{r}{\sqrt{t-s}} \right)^{\alpha} \cdot
       e^{-\frac{d^2(y_0,z)}{C(t-s)}}|f|.
   \label{eqn:SK13_03}
   \end{align}
   On $A_{i,I}$, we note that $\frac{r}{\sqrt{t-s}} \leq 5^{1-i}$ and $e^{-\frac{d^2(y_0,z)}{C(t-s)}} \leq 1$. Then we calculate
   \begin{align*}
         &\quad \iint_{A_{i,I}}  V_{y_0}^{-1}(\sqrt{t-s})(t-s)^{-1}
         \left(\frac{r}{\sqrt{t-s}} \right)^{\alpha}
         e^{-\frac{d^2(y_0,z)}{C(t-s)}}|f| \\
         &\leq C 5^{-\alpha i} \iint_{A_{i,I}}  V_{y_0}^{-1}(\sqrt{t-s})(t-s)^{-1}|f| \\
         &\leq C5^{-\alpha i}\iint_{A_{i,I}} V_{y_0}^{-1}(5^{-i+2}r)(5^{-i+2}r)^{-2} \cdot \frac{V_{y_0}(5^{-i+2}r) \cdot (5^{-i+2}r)^2}{V_{y_0}(\sqrt{t-s}) \cdot (t-s)} \cdot |f|\\
         &=C5^{-\alpha i}\iint_{P_{5^{-i+2}r}(Y_0)} V_{y_0}^{-1}(5^{-i+2}r)(5^{-i+2}r)^{-2}|f|,
    \end{align*}
    where we apply volume comparison in the last step. 
    The volume of $P_{5^{i+2}r}(Y_0)$ is comparable to $C \cdot V_{y_0}(5^{i+2}r) \cdot (5^{i+2}r)^2$.
    It follows from the above inequality that 
    \begin{align}
         \iint_{A_{i,I}}  V_{y_0}^{-1}(\sqrt{t-s})(t-s)^{-1}
         \left(\frac{r}{\sqrt{t-s}} \right)^{\alpha}
         e^{-\frac{d^2(y_0,z)}{C(t-s)}}|f| dzds
          \leq C5^{-\alpha i}\mathbf{M}(|f|)(Y_0).
    \label{eqn:SK10_1}
    \end{align}
    
   On the other hand,  for any 
   $(z,s)\in A_{i,II}$, we have $\sqrt{t-s} < 5^{i-1}r<d(y_0, z)$.
   Thus, 
   \begin{align*}
       \frac{d^2(y_0, z)}{t-s} = \left( \frac{d(y_0, z)}{\sqrt{t-s}} \right)^2>5^{2i-2} \frac{r^2}{t-s}>1. 
   \end{align*}
   Consequently, by volume comparison we have 
    \begin{align*}
         &\quad \iint_{A_{i,II}}  V_{y_0}^{-1}(\sqrt{t-s})(t-s)^{-1} 
           \left(\frac{r}{\sqrt{t-s}} \right)^{\alpha}
          \cdot e^{-\frac{d^2(y_0,z)}{C(t-s)}}|f| dzds\\
         &=\iint_{A_{i,II}}  V_{y_0}^{-1}(5^{i+2} r)(5^{i+2}r)^{-2} \cdot
         \frac{V_{y_0}(5^{i+2}r) \cdot (5^{i+2}r)^2}{V_{y_0}(\sqrt{t-s}) \cdot (t-s)} \cdot
         \left(\frac{r}{\sqrt{t-s}} \right)^{\alpha}
         e^{-\frac{d^2(y_0,z)}{C(t-s)}}|f| dzds. 
    \end{align*}
    By volume comparison we have
    \begin{align*}
       &\quad \frac{V_{y_0}(5^{i+2}r) \cdot (5^{i+2}r)^2}{V_{y_0}(\sqrt{t-s}) \cdot (t-s)} \cdot
         \left(\frac{r}{\sqrt{t-s}} \right)^{\alpha}
         e^{-\frac{d^2(y_0,z)}{C(t-s)}} \\
       &\leq C \cdot \left( \frac{5^{i+2}r}{\sqrt{t-s}} \right)^{n+2} \cdot \left( \frac{r}{\sqrt{t-s}} \right)^{\alpha}
       \cdot e^{-\frac{5^{2i-2}r^2}{t-s}}\\
       &\leq C \left( \frac{5^{i-1}r}{\sqrt{t-s}} \right)^{n+2} \cdot \left( \frac{5^{i-1} r}{\sqrt{t-s}} \right)^{\alpha} \cdot e^{-(\frac{5^{i-1}r}{\sqrt{t-s}})^2} \cdot 5^{-\alpha i}\\
       &\leq C 5^{-\alpha i}. 
    \end{align*}
    Combining the previous two steps yields that
    \begin{equation}
     \begin{split}
         &\quad \iint_{A_{i,II}}  V_{y_0}^{-1}(\sqrt{t-s})(t-s)^{-1}  
         \cdot \left(\frac{r}{\sqrt{t-s}} \right)^{\alpha}
          \cdot e^{-\frac{d^2(y_0,z)}{C(t-s)}}|f| dzds \\
         &\leq C 5^{-\alpha i} \iint_{A_{i,II}}  V_{y_0}^{-1}(5^{i+2} r)(5^{i+2}r)^{-2} |f| dzds \\
         &\leq C 5^{-\alpha i} \iint_{P_{5^{i+2}r}(Y_0)}  V_{y_0}^{-1}(5^{i+2} r)(5^{i+2}r)^{-2} |f| dzds \\
         &\leq C 5^{-\alpha i} \mathbf{M}(|f|)(Y_0).
    \end{split}
    \label{eqn:SK13_1}
    \end{equation}
    Putting (\ref{eqn:SK13_03}), (\ref{eqn:SK10_1}) and (\ref{eqn:SK13_1}) together, we obtain
    \begin{align*} 
         &\quad \iint_{B_{1}(y_0) \times [-1, t]} 
   \frac{r^{\alpha}}{V_{y_0}(\sqrt{t-s})(t-s)^{1+\frac{\alpha}{2}}}
            e^{-\frac{d^2(y_0,z)}{C(t-s)}}|f|\\
         &\leq C \left\{ \sum_{i=1}^{N} 5^{-\alpha i} \right\} \cdot \mathbf{M}(|f|)(Y_0) 
         \leq C \cdot \left( \sum_{i=1}^{\infty} 5^{-\alpha i} \right) \cdot \mathbf{M}(|f|)(Y_0), 
    \end{align*}
    which yields (\ref{eqn:HA19_5}).  
    Similarly, replacing $\alpha$ by $1$ and running through the proof of (\ref{eqn:HA19_5}), we obtain (\ref{eqn:HB23_5}). 

    Notice that $0<\alpha<1$ and $0<\frac{d(y,y_0)}{r}<4$. Thus, it follows from the combination of (\ref{eqn:HA19_5}), (\ref{eqn:HB23_5}) and (\ref{eqn:HB23_4}) that
  \begin{align*}
  \begin{split}
   \left||\mathcal{T}f|(Y)-|\mathcal{T}f|(Y_0)| \right|
   &\leq  C \left\{ \left( \frac{d(y,y_0)}{r}\right)^{\alpha} + \frac{d(y,y_0)}{r} \right\} \cdot \mathbf{M}(|f|)(Y_0)\\
   &\leq  C \cdot  \left( \frac{d(y,y_0)}{r}\right)^{\alpha} \cdot \mathbf{M}(|f|)(Y_0),
  \end{split}     
  \end{align*}
  which is exactly (\ref{eqn:HA19_4}). 
\end{proof}

\begin{lemma}
\label{lma:HA20_2}   
Suppose $Y_0=(y_0,t_0)$, $Y=(y_0,t) \in P_{4r}(Y_0)=P_{4r}(y_0, t_0)$. 
Suppose $\mathrm{supp} f \in \mathcal{Q}' \setminus P_{5r}(Y_0)$. 
Then we have
    \begin{align}
        |\mathcal{T}f(Y_0)-\mathcal{T}f(Y)|< C  \cdot \mathbf{M}(|f|)(Y_0) \cdot \frac{|t_0-t|}{r^2}. 
    \label{eqn:HA20_1}    
    \end{align}
\end{lemma}
\begin{proof}
    Without loss of generality, we assume $t_0>t$. Then we have
    \begin{align*}
 		\begin{split}
 			&\quad ||\mathcal{T}f|(Y)-|\mathcal{T}f|(Y_0)|\\
            &\leq \bigg|\int_{-1}^{t}\int_M |\CK(y,t_0;z,s)-\CK(y,t;z,s)||f(z,s)|dz ds\bigg|
             +\bigg|\int_t^{t_0}\int_M \CK(y,t_0;z,s)f(z,s)dz ds\bigg|\\
           &\leq \underbrace{\bigg|\int_{-1}^{t}\int_M \int_{t}^{t_0} |\partial_\tau \CK(y,\tau;z,s)||f(z,s)|d\tau dz ds\bigg|}_{III}
             +\underbrace{\bigg|\int_t^{t_0}\int_M | \CK(y,t_0;z,s)||f(z,s)|dz ds\bigg|}_{IV}. 
 		\end{split}
  	\end{align*}
    The part $III$ can be estimated as follows (See Figure~\ref{fig3}). 
    \begin{align}
 		\begin{split}
           \quad III=\int_t^{t_0} \Big\{ \underbrace{\sum_{i=0}^N \int_M  \int_{t-5^{2(i+1)}r^2}^{t-5^{2i}r^2} |\partial_{\tau} \CK||f| ds dz}_{III_{A}}
           +  \underbrace{\int_M \int_{t-r^2}^t |\partial_{\tau} \CK||f| ds dz}_{III_{B}} \Big\} d\tau. 
 		\end{split}
    \label{eqn:HA24_3}    
    \end{align}
    \begin{figure}[htbp]
    \centering
   \resizebox{0.3\textwidth}{!}{

\tikzset{every picture/.style={line width=0.75pt}} 

\begin{tikzpicture}[x=0.75pt,y=0.75pt,yscale=-1,xscale=1]

\draw  [color={rgb, 255:red, 0; green, 0; blue, 0 }  ,draw opacity=1 ][fill={rgb, 255:red, 74; green, 74; blue, 74 }  ,fill opacity=1 ][line width=1.5]  (58.5,41.03) -- (567,41.03) -- (567,549.53) -- (58.5,549.53) -- cycle ;
\draw  [dash pattern={on 0.84pt off 2.51pt}]  (58,189) -- (196,190) -- (567,189) ;
\draw  [fill={rgb, 255:red, 155; green, 155; blue, 155 }  ,fill opacity=1 ] (58.5,41.03) -- (156.5,41.03) -- (156.5,188.03) -- (58.5,188.03) -- cycle ;
\draw  [fill={rgb, 255:red, 155; green, 155; blue, 155 }  ,fill opacity=1 ] (466.97,41.03) -- (567,41.03) -- (567,189) -- (466.97,189) -- cycle ;
\draw  [fill={rgb, 255:red, 255; green, 255; blue, 255 }  ,fill opacity=1 ] (154.75,41.03) -- (466.97,41.03) -- (466.97,353.25) -- (154.75,353.25) -- cycle ;
\draw  [color={rgb, 255:red, 0; green, 0; blue, 0 }  ,draw opacity=1 ][fill={rgb, 255:red, 128; green, 128; blue, 128 }  ,fill opacity=1 ] (58.5,188.03) -- (154.5,188.03) -- (154.5,281.03) -- (58.5,281.03) -- cycle ;
\draw  [fill={rgb, 255:red, 128; green, 128; blue, 128 }  ,fill opacity=1 ] (466.97,189) -- (568,189) -- (568,283) -- (466.97,283) -- cycle ;
\draw   (222,40) -- (388,40) -- (388,206) -- (222,206) -- cycle ;
\draw  [fill={rgb, 255:red, 0; green, 0; blue, 0 }  ,fill opacity=1 ] (307,41) .. controls (307,38.79) and (308.79,37) .. (311,37) .. controls (313.21,37) and (315,38.79) .. (315,41) .. controls (315,43.21) and (313.21,45) .. (311,45) .. controls (308.79,45) and (307,43.21) .. (307,41) -- cycle ;
\draw  [fill={rgb, 255:red, 0; green, 0; blue, 0 }  ,fill opacity=1 ] (304.5,189) .. controls (304.5,186.79) and (306.29,185) .. (308.5,185) .. controls (310.71,185) and (312.5,186.79) .. (312.5,189) .. controls (312.5,191.21) and (310.71,193) .. (308.5,193) .. controls (306.29,193) and (304.5,191.21) .. (304.5,189) -- cycle ;
\draw  [dash pattern={on 0.84pt off 2.51pt}]  (59,189) -- (568,189) ;
\draw  [dash pattern={on 0.84pt off 2.51pt}]  (58,281) -- (568,283) ;
\draw  [dash pattern={on 0.84pt off 2.51pt}]  (314,583) -- (310.01,11) ;
\draw [shift={(310,9)}, rotate = 89.6] [color={rgb, 255:red, 0; green, 0; blue, 0 }  ][line width=0.75]    (10.93,-3.29) .. controls (6.95,-1.4) and (3.31,-0.3) .. (0,0) .. controls (3.31,0.3) and (6.95,1.4) .. (10.93,3.29)   ;
\draw  [dash pattern={on 0.84pt off 2.51pt}]  (2,41) -- (641,42) ;
\draw [shift={(643,42)}, rotate = 180.09] [color={rgb, 255:red, 0; green, 0; blue, 0 }  ][line width=0.75]    (10.93,-3.29) .. controls (6.95,-1.4) and (3.31,-0.3) .. (0,0) .. controls (3.31,0.3) and (6.95,1.4) .. (10.93,3.29)   ;

\draw (230,48.4) node [anchor=north west][inner sep=0.75pt]  [font=\fontsize{2.65em}{3.18em}\selectfont]  {$P_{r}^{-}$};
\draw (183,286.4) node [anchor=north west][inner sep=0.75pt]  [font=\fontsize{2.65em}{3.18em}\selectfont]  {$P_{5r}^{-}$};
\draw (313,48.4) node [anchor=north west][inner sep=0.75pt]  [font=\fontsize{2.65em}{3.18em}\selectfont]  {$Y_{0}$};
\draw (291,421.4) node [anchor=north west][inner sep=0.75pt]  [font=\fontsize{2.65em}{3.18em}\selectfont]  {$III_{A}$};
\draw (330,-7.6) node [anchor=north west][inner sep=0.75pt]  [font=\fontsize{2.65em}{3.18em}\selectfont]  {$t$};
\draw (591,68.4) node [anchor=north west][inner sep=0.75pt]  [font=\fontsize{2.65em}{3.18em}\selectfont]  {$M$};
\draw (331,154.4) node [anchor=north west][inner sep=0.75pt]  [font=\fontsize{2.65em}{3.18em}\selectfont]  {$Y$};
\draw (77,86.4) node [anchor=north west][inner sep=0.75pt]  [font=\fontsize{2.65em}{3.18em}\selectfont]  {$IV$};
\draw (488,89.4) node [anchor=north west][inner sep=0.75pt]  [font=\fontsize{2.65em}{3.18em}\selectfont]  {$IV$};
\draw (493,200.4) node [anchor=north west][inner sep=0.75pt]  [font=\fontsize{2.65em}{3.18em}\selectfont]  {$III_{B}$};
\draw (74,210.4) node [anchor=north west][inner sep=0.75pt]  [font=\fontsize{2.65em}{3.18em}\selectfont]  {$III_{B}$};

\end{tikzpicture}

}
    \caption{}
    \label{fig3}
\end{figure}
    Decompose $M$ into $\{d^2(y,z) \geq \tau-s\} \cup \{d^2(y,z)<\tau-s\}$. 
    If $s \in [t-5^{2(i+1)}r^2, t-5^{2i}r^2]$, then 
    $\tau-s \geq t-s \geq 5^{2i}r^2$.  Thus, by containment relationship, we have
     \begin{align}
        \begin{split}
             \quad \int_{d^2(y,z)\geq (\tau-s)} \int_{t-5^{2(i+1)}r^2}^{t-5^{2i}r^2} |\partial_{\tau} \CK||f|
             \leq \int_{t-5^{2(i+1)}r^2}^{t-5^{2i}r^2} \sum_{j=i}^N 
             \int_{5^jr\leq d(y,z)\leq 5^{j+1}r} |\partial_{\tau} \CK||f|.
        \end{split}
    \label{eqn:HA24_1}    
    \end{align}
    For $j \in \{i, i+1, \cdots, N\}$,  the volume comparison implies that
    \begin{align}
    \begin{split}
       &\quad V_y^{-1}(\sqrt{\tau-s})(\tau-s)^{-2}\\
       &\leq C \cdot V_y^{-1} (5^{j+1}r) \cdot \left(\frac{5^{j+1}r}{\sqrt{\tau-s}} \right)^{n}
       \cdot (\tau-s)^{-2}\\
       &=C \cdot V_y^{-1} (5^{j+1}r) \cdot (5^{j+1}r)^{-2} \cdot (5^{j+1}r)^{n+2} \cdot (\sqrt{\tau-s})^{-n-4}\\
       &\leq C \cdot V_y^{-1} (5^{j+1}r) \cdot (5^{j+1}r)^{-2} \cdot (5^{j+1}r)^{n+2} \cdot (5^i r)^{-n-4}\\
       &\leq C \cdot \frac{1}{|P_{5^{j+2}r}|} \cdot (5^i r)^{-2} \cdot (5^{j-i+1})^{n+2}. 
    \end{split}   
    \label{eqn:HA23_3}
    \end{align}
    Thus, we have
     \begin{equation*}
        \begin{split}
             &\quad \int_{t-5^{2(i+1)}r^2}^{t-5^{2i}r^2}
             \int_{5^jr\leq d(y,z)\leq 5^{j+1}r} |\partial_{\tau} \CK||f|dsdz\\
             &\leq \int_{t-5^{2(i+1)}r^2}^{t-5^{2i}r^2} \int_{5^jr\leq d(y,z)\leq 5^{j+1}r}V^{-1}_y(\sqrt{\tau-s})(\tau-s)^{-2}e^{-\frac{d^2(y,z)}{C_1(\tau-s)}}|f(z,s)| dz ds\\
             &\leq  C \cdot (5^i r)^{-2} \cdot (5^{j-i+1})^{n+2} \cdot \fint_{P_{5^{j+2}r}}   
                e^{-\frac{d^2(y,z)}{C_1(\tau-s)}}|f(z,s)| dz ds\\
             &\leq C(5^{i}r)^{-2}(5^{j-i})^{n+4}e^{-\frac{5^{2(j-i)}}{C}}\fint_{P_{5^{j+2}r}(Y_0)}|f(z,s)| dz ds.
        \end{split}
    \end{equation*}
    Summing them up by $j$ and using (\ref{eqn:HA24_1}), we obtain
      \begin{align}
      \label{eq315}
        \begin{split}
         &\quad      \int_{d^2(y,z)\geq (\tau-s)} \int_{t-5^{2(i+1)}r^2}^{t-5^{2i}r^2} 
         |\partial_{\tau} \CK||f|dsdz\\
         &\leq C(5^{i}r)^{-2} \left\{\sum_{j=i}^N  (5^{j-i})^{n+4}e^{-\frac{5^{2(j-i)}}{C}}   \right\} \cdot \mathbf{M}(|f|)(Y_0)
          \leq C5^{-2i} r^{-2}\mathbf{M}(|f|)(Y_0).
        \end{split}
    \end{align}
    On the other hand, applying (\ref{eqn:HA23_3}) for $j=i$, we have 
    \begin{equation}
    \label{eqn:SK13_3}
        \begin{split}
            &\quad \int_{d^2(x,y)\leq (\tau-s)} \int_{t-5^{2(i+1)}r^2}^{t-5^{2i}r^2} |\partial_{\tau} \CK||f|dsdz\\
            &\leq C (5^i r)^{-2} \fint_{P_{5^{i+2}r}}|f|(z,s) dzds
             \leq C 5^{-2i} r^{-2}\mathbf{M}(|f|)(Y_0).  
        \end{split}
    \end{equation}
    Combining (\ref{eq315}) and (\ref{eqn:SK13_3}), we obtain 
    \begin{align*}
        \int_{M} \int_{t-5^{2(i+1)}r^2}^{t-5^{2i}r^2} |\partial_{\tau} \CK||f|dsdz
        \leq C 5^{-2i}r^{-2}\mathbf{M}(|f|)(Y_0),    
    \end{align*}
    whose summation over $i$ yields that
    \begin{align}
       III_{A}
       =\int_{M} \sum_{i=1}^N \int_{t-5^{2(i+1)}r^2}^{t-5^{2i}r^2} |\partial_{\tau} \CK||f|dsdz
        \leq C r^{-2} \mathbf{M}(|f|)(Y_0).  
    \label{eqn:SK13_4}    
    \end{align}

    If $s \in [t-r^2, t]$ and $(z,s) \in \mathrm{supp}(f)$, we know $1 \geq d(y,z)\geq 5r$ and $\tau-s \geq r^2$ by the assumption of $f$.
    Using the deduction in (\ref{eqn:HA23_3}) again, we have
    \begin{align*}
        \begin{split}
             &\quad \int_{t-r^2}^{t}\int_M |\partial_{\tau} \CK||f| dzds =\int_{t-r^2}^{t}\sum_{j=0}^{N}\int_{B_{5^{j+1}r}(y)\setminus B_{5^jr}(y)} |\partial_{\tau} \CK||f| dzds\\
             &\leq C \int_{t-r^2}^{t}\sum_{j=0}^{N}\int_{B_{5^{j+1}r}(y)\setminus B_{5^jr}(y)}  V^{-1}_y(\sqrt{\tau-s})(\tau-s)^{-2}e^{-\frac{d^2(x,y)}{C_1(\tau-s)}}|f| dzds\\
             &\leq C r^{-2} \cdot \sum_{j=0}^N \int_{t-r^2}^{t} \int_{B_{5^{j+1}r}(y)\setminus B_{5^jr}(y)} \frac{1}{|P_{5^{j+2}r}(Y_0)|} \cdot (5^{j+1})^{n+2} e^{-\frac{5^{2j}}{C}} |f| dzds.
        \end{split}
    \end{align*}
    As $\{B_{5^{j+1}r}(y)\setminus B_{5^jr}(y)\} \times [t-r^2, t] \subset P_{5^{j+2}r}(Y_0)$, we have
    \begin{align}
        \begin{split}
             III_{B}
             &=\int_{t-r^2}^{t}\int_M |\partial_{\tau} \CK||f| dzds 
             \leq C r^{-2} \cdot \sum_{j=0}^N  \fint_{P_{5^{j+2}r}(Y_0)} (5^{j+1})^{n+2} e^{-\frac{5^{2j}}{C}} |f| dzds\\
             &\leq C r^{-2} \cdot \left\{ \sum_{j=0}^N (5^{j+1})^{n+2} e^{-\frac{5^{2j}}{C}} \right\} \cdot \mathbf{M}(|f|)(Y_0)
             \leq C r^{-2}\mathbf{M}(|f|)(Y_0). 
        \end{split}
    \label{eqn:HA24_2}    
    \end{align}
    Plugging (\ref{eqn:SK13_4}) and (\ref{eqn:HA24_2}) into (\ref{eqn:HA24_3}), we obtain 
    \begin{align}
        \begin{split}
             III= \int_{t}^{t_0} \int_{t-r^2}^{t}\int_M |\partial_{\tau} \CK||f| dzdsd\tau 
              \leq C  \cdot \mathbf{M}(|f|)(Y_0) \cdot \frac{|t_0-t|}{r^2}.
        \end{split}
    \label{eqn:HA23_1}    
    \end{align}

Then we estimate $IV$.  By the support set assumption, we have 
\begin{align}
\label{eqn:SK18_1}
\begin{split}
\bigg|\int_t^{t_0}\int_M |\CK||f|(z,s)dzds\bigg|
   =\sum_{i=1}^N  \int_t^{t_0}\int_{B_y(5^{i+1}r)\setminus B_y(5^ir)} |\CK||f|dzds, 
\end{split}
\end{align}
where $\mathcal{K}=\CK(y,t_0;z,s)$.   
In the annulus $\{B_y(5^{i+1}r)\setminus B_y(5^ir)\} \times [t, t_0]$, by heat kernel estimate, volume comparison and the fact $s \in [t, t_0]$, we have
\begin{align}
\begin{split}
  |\CK|(y,t_0;z,s) &\leq C V^{-1}_y(\sqrt{t_0-s})(t_0-s)^{-1}e^{-\frac{d^2(x,y)}{C_1(t_0-s)}}\\
 &\leq 
 C \cdot V^{-1}_y(5^{i+2}r)(5^{i+2}r)^{-2}\bigg(\frac{5^{i+2}r}{\sqrt{t_0-s}}\bigg)^{n+2}
 e^{-\frac{5^{2i}r^2}{C_1(t_0-s)}}\\
 &\leq C \cdot \frac{E_i}{|P_{5^{i+2}r}(Y_0)|}, 
\end{split} 
\label{eqn:SK18_2}
\end{align}
where  
\begin{equation*}
\begin{split}
    E_i:&=\bigg(\frac{5^{i+2}r}{\sqrt{t_0-s}}\bigg)^{n+2}e^{-\frac{5^{2i}r^2}{C_1(t_0-s)}}
    =5^{2n+4} \left(\frac{5^i r}{\sqrt{t_0-s}} \right)^{n+2} \cdot e^{-\frac{1}{C_1} \cdot (\frac{5^i r}{\sqrt{t_0-s}})^2}\\
    &\leq C \cdot \left( \frac{|t_0-s|}{5^{2i} r^2} \right)
    \cdot \left(\frac{5^i r}{\sqrt{t_0-s}} \right)^{n+4} \cdot e^{-\frac{1}{C_1} \cdot (\frac{5^i r}{\sqrt{t_0-s}})^2}. 
\end{split}
\end{equation*}
Note that $\rho^{n+4} \cdot e^{-\frac{\rho^2}{C_1}}$ is uniformly bounded.
We have
\begin{align}
    E_i \leq  C \cdot 5^{-2i} 
    \cdot  \frac{|t_0-s|}{r^2} \leq  C \cdot 5^{-2i} \cdot  \frac{|t_0-t|}{r^2}.  
\label{eqn:HA24_4}    
\end{align}
Plugging (\ref{eqn:SK18_2}) and (\ref{eqn:HA24_4}) into (\ref{eqn:SK18_1}), we have
\begin{align}
\begin{split}
 IV &\leq \bigg|\int_t^{t_0}\int_M |\CK(y,t_0;z,s)||f(z,s)|dz ds\bigg|
     \leq  C \sum_{i=1}^N E_i \cdot \fint_{P_{5^{i+2}r}(Y_0)}|f| dz ds\\
    &\leq  C \cdot \left( \sum_{i=1}^N E_i \right) \cdot  \mathbf{M}(|f|)(Y_0)
     \leq  C  \cdot \mathbf{M}(|f|)(Y_0) \cdot \frac{|t_0-t|}{r^2}.
\end{split} 
\label{eqn:HA23_2}
\end{align}
Combining (\ref{eqn:HA23_1}) and (\ref{eqn:HA23_2}), we obtain (\ref{eqn:HA20_1}). 
\end{proof}

\begin{lemma}
\label{lma:HA20_5}
There exists a constant $C=C(n,\Lambda_0, C_1)$ with the following property. 

Suppose $Y_0=(y_0,t_0)$, $Y=(y, t)\in P_{4r}(Y_0)$.
Suppose $\mathrm{supp}f\subset \mathcal{Q}' \setminus P_{5r}(Y_0)$ and $f\in L^2(\mathcal{Q}')$.   
Then we have
 \begin{align}
      \mathbf{M}(|f|)(Y) \leq C \mathbf{M}(|f|)(Y_0). 
 \label{eqn:HA20_6}    
 \end{align}
 \end{lemma}

 \begin{proof}
By the triangle inequality, we know $P_{r}(y_0,t)\subset P_{5r}(Y_0)$,  where $f \equiv 0$ by the choice of $f$. Thus, we have 
\begin{align}
\begin{split}
        \mathbf{M}(|f|)(y_0,t)&=\sup_{0<s<\frac12}\frac{1}{|P_s(y_0,t)|}\int_{P_s(y_0,t)}|f| dzds\\
        &=\sup_{r<s<\frac12}\frac{1}{|P_s(y_0,t)|}\int_{P_s(y_0,t)}|f| dzds.
\end{split}        
\label{eqn:HA20_5}        
\end{align}
However, if $s \in (r, \frac12)$, we have
\begin{align*}
  P_s(y_0,t) \subset P_{5s}(y_0, t_0).  
\end{align*}
Volume comparison then implies that
\begin{align*}
    |P_s(y_0, t)|  \sim |B_s(y_0, t)| s^2 \geq \frac{1}{C} |B_{5s}(y_0,t)| \cdot (5s)^2
    \geq \frac{1}{C} |P_{5s}(y_0, t_0)|. 
\end{align*}
If $s \in (r, \frac{1}{10})$, then we have 
\begin{align}
\begin{split}
        &\quad \frac{1}{|P_s(y_0,t)|}\int_{P_s(y_0,t)}|f| dzds\\
        &\leq C \frac{1}{|P_{5s}(y_0,t_0)|}\int_{P_{5s}(y_0,t_0)}|f| dzds
         \leq C  \sup_{5 r<\rho <\frac12}\frac{1}{|P_{\rho}(y_0,t_0)|}\int_{P_{\rho}(y_0,t_0)}|f| dzds\\
        &\leq C \mathbf{M}(|f|)(y_0,t_0).
\end{split}  
\label{eqn:HA20_3}
\end{align}
If $s \in [\frac{1}{10}, \frac12]$, then $P_{\frac12}(y_0,t_0) \supset \mathcal{Q}' \supset \mathrm{supp} f$. Thus, we have
\begin{align}
\begin{split}
       &\quad \frac{1}{|P_s(y_0,t)|}\int_{P_s(y_0,t)}|f| dzds\\
       &\leq \frac{1}{|P_{1/10}(y_0,t)|}\int_{\mathcal{Q}'}|f| dzds\\
       &\leq \frac{C}{|P_{\frac12}(y_0,t_0)|}\int_{P_{\frac12}(y_0,t_0)}|f| dzds\\
       & \leq C \mathbf{M}(|f|)(y_0,t_0).
\end{split}    
\label{eqn:HA20_4}
\end{align}
Combining (\ref{eqn:HA20_3}) and (\ref{eqn:HA20_4}), we obtain 
\begin{align*}
   \sup_{r<s<\frac12}\frac{1}{|P_s(y_0,t)|}\int_{P_s(y_0,t)}|f| dzds 
   \leq C \mathbf{M}(|f|)(y_0,t_0).
\end{align*}
Plugging this into (\ref{eqn:HA20_5}), we obtain (\ref{eqn:HA20_6}). 
\end{proof}

\begin{proposition}
\label{prn:SK27_1}
There exists a constant $C=C(n,\Lambda_0, C_1)$ with the following property. 

Suppose $Y_0=(y_0,t_0)$. 
Suppose $\mathrm{supp}f\subset \mathcal{Q}' \setminus P_{5r}(Y_0)$ and $f\in L^2(\mathcal{Q}')$.   
Then we have
 \begin{align}
     ||\mathcal{T}f|(Y)-|\mathcal{T}f|(Y_0)| \leq C \mathbf{M}(|f|)(Y_0)
 \label{eqn:SJ26_2}    
 \end{align}
 for every $Y=(y, t)\in P_{4r}(Y_0)$.
 \end{proposition}


\begin{proof}

 For general $Y=(y,t)$, we have 
    \begin{equation}
    \label{eqn:SK18_5}
    \begin{split}
         ||\mathcal{T}f|(Y)-|\mathcal{T}f|(Y_0)|
        &\leq ||\mathcal{T}f|(y,t)-|\mathcal{T}f|(y_0,t)|+ ||\mathcal{T}f|(y_0,t)-|\mathcal{T}f|(y_0,t_0)|\\
        &\leq C \left\{\mathbf{M}(|f|)(y_0,t)+\mathbf{M}(|f|)(Y_0) \right\}.
    \end{split}
    \end{equation}
   Plugging (\ref{eqn:HA20_6}) into the above inequality, we arrive at (\ref{eqn:SJ26_2}). 
\end{proof}

\begin{proposition}
\label{prn:SK27_2}
 The same conditions as in Proposition~\ref{prn:SK27_1}. 
 Then we have 
 	\begin{align}
    \label{eqn:SK18_6}
 		\mathbf{M}(|\mathcal{T}f|^2)(Y)
        \leq C \left\{ \mathbf{M}(|\mathcal{T}f|^2)(Y_0) + \mathbf{M}(|f|^2)(Y_0) \right\} 
 	\end{align}
    for every $Y=(y,t) \in P_{3r}(Y_0)$. 
\end{proposition}

 \begin{proof}
 Denote 
 \begin{align}
    a^2:= \mathbf{M}(|\mathcal{T}f|^2)(Y_0), \quad  b^2 :=\mathbf{M}(|f|^2)(Y_0). 
 \label{eqn:SJ24_6}    
 \end{align} 
 It suffices to show
 \begin{align}
    \label{eq2.4}
 		\mathbf{M}(|\mathcal{T}f|^2)(Y)
        \leq C \left\{ a^2 + b^2 \right\}.  
 	\end{align}

    In light of Proposition~\ref{prn:SK27_1}, we can apply the triangle inequality to obtain 
 	\begin{equation*}
 		\begin{split}
 			|\mathcal{T}f|(Y)&\leq |\mathcal{T}f|(Y_0)+||\mathcal{T}f|(Y)-|\mathcal{T}f|(Y_0)|\\
 			&\leq \mathbf{M}(|\mathcal{T}f|)(Y_0)+||\mathcal{T}f|(Y)-|\mathcal{T}f|(Y_0)|\\
            &\leq C \left\{ \mathbf{M}(|\mathcal{T}f|)(Y_0) + \mathbf{M}(|f|)(Y_0)\right\}. 
 		\end{split}
 	\end{equation*}
 	The H\"older inequality implies that
 	\begin{equation*}
 		\begin{split}
 			&\mathbf{M}(|\mathcal{T}f|)(Y_0)\leq \sqrt{\mathbf{M}(|\mathcal{T}f|^2)(Y_0)}= a, \\
            &\mathbf{M}(|f|)(Y_0)\leq \sqrt{\mathbf{M}(|f|^2)(Y_0)}=b.
 		\end{split}
 	\end{equation*}
    Combining the previous two steps, we have
 	\begin{equation}
 		\begin{split}
 			|\mathcal{T}f|(Y)\leq C(a+b), \quad \forall \; Y=(y,t)\in P_{5r}(Y_0).
 		\end{split}
    \label{eqn:SJ27_2}    
 	\end{equation}

Now we prove (\ref{eq2.4}).  Fix $Y=(y,t) \in P_{3r}(Y_0)$. 

If $s \in (0, r)$, then $P_s(Y) \subset P_{4r}(Y_0)$ by the triangle inequality.
Thus, it follows from (\ref{eqn:SJ27_2}) that
\begin{align}
    \fint_{P_s(Y)}|\mathcal{T}f|^2 
    \leq  
    \sup_{P_{s}(Y)} |\mathcal{T}f|^2  
    \leq 
    \sup_{P_{5r}(Y_0)} |\mathcal{T}f|^2  
    \leq C(a+b)^2 \leq C(a^2+b^2).
\label{eq2.24}
\end{align}    

If $s\in (r, \frac18)$, then $r \leq  s$.  Note that $P_s(Y) \subset P_{4s}(Y_0)$. 
Consequently, we have
    \begin{equation}
    \label{eq2.25}
    \fint_{P_s(Y)}|\mathcal{T}f|^2\leq C\fint_{P_{4s}(Y_0)}|\mathcal{T}f|^2\leq C \mathbf{M}(|\mathcal{T}f|^2)(Y_0) \leq Ca^2 
    \leq C(a^2+b^2). 
\end{equation}

If $s \in [\frac18, \frac12]$, then $\mathcal{Q}' \subset  P_{\frac12}(Y_0)$ and $P_{\frac{1}{16}}(Y_0) \subset P_s(Y)$ 
Thus, by assumption (\ref{eqn:HA19_1}) and volume comparison,  we have 
\begin{align}
\begin{split}
    &\quad \fint_{P_s(Y)} |\mathcal{T}f|^2\\
    &\leq 
    \frac{1}{|P_s(Y)|} \int_{M \times [-1, 0]} |\mathcal{T}f|^2  \leq \frac{C_2}{|P_s(Y)|} \int_{\mathcal{Q}'} |f|^2
    \leq \frac{C_2}{|P_s(Y)|} \int_{P_{\frac12}(Y_0)} |f|^2\\
    &=C_2 \cdot \frac{|P_{\frac12}(Y_0)|}{|P_{\frac{1}{16}}(Y)|} \cdot \frac{1}{|P_{\frac12}(Y_0)|} \int_{P_{\frac12}(Y_0)} |f|^2 
    \leq C \mathbf{M}(|f|^2)(Y_0)=C b^2
    \leq C (a^2+b^2). 
\end{split}    
\label{eqn:HA19_2}
\end{align}

Therefore,~\eqref{eq2.4} follows from the combination of~\eqref{eq2.24},~\eqref{eq2.25} and (\ref{eqn:HA19_2}). 
 \end{proof}

\section{The Calder\'on-Zygmund inequalities}
\label{sec6}

Using the stability of maximal functions, we are able to show an iteration relationship, which is a key step for the Calder\'on-Zygmund inequality.  

\begin{proposition}
\label{prn:SK27_3}
There exist a constant $N=N(n, \Lambda_0, C_1, C_2)$ and a small constant $\varepsilon_0(n, \Lambda_0)$ with the following property. 

Suppose $\mathrm{supp}f\subset \mathcal{Q}', f\in L^2(\mathcal{Q}')$ and $\varepsilon \in (0,\varepsilon_0)$. 
For each $t > 1$, we have
	\begin{equation}
	\begin{split}
	&\quad |\{X\in \mathcal{Q}:\mathbf{M}(|\mathcal{T}f|^2)(X)> N^2 t\}|\\
    &\leq C \varepsilon 
     \bigg( \left|\{X\in \mathcal{Q}:\mathbf{M}(|\mathcal{T}f|^2)(X)> t \} \right|
	   + \left|\{X\in \mathcal{Q}:\mathbf{M}(|f|^2)(X)>t \varepsilon\} \right|	\bigg). 
		\end{split}
    \label{eqn:SK22_1}    
	\end{equation}
\end{proposition}

 The proof of Proposition~\ref{prn:SK27_3} follows from a local iteration argument and a standard covering argument. We shall separate these two ingredients into two lemmas. 

\begin{lemma}
\label{lm2.5}
Same conditions as in Proposition~\ref{prn:SK27_3}. 
If 
\begin{align}
\label{eq3.29}
	 \left|\left\{X\in P_{5r}(X_0):\mathbf{M}(|\mathcal{T}f|^2)(X)\leq 1, 
    \mathbf{M}(|f|^2)(X)\leq \varepsilon \right\} \right| 
     \geq \frac{1}{2}|P_{5r}(X_0)|
\end{align}
for some $X_0\subset \mathcal{Q}$ and $r \in (0,\frac12)$,  then 
\begin{align}
  \left| \left\{X\in P_{5r}(X_0):\mathbf{M}(|\mathcal{T}f|^2)(X)> N^2 \right\} \right| \leq \varepsilon |P_{5r}(X_0)|.
\label{eqn:SH23_4}           
\end{align}
\end{lemma}

\begin{proof}

  We analyze the behavior of $f$ nearby $X_0$. For this purpose, we define (See Figure~\ref{fig4})
    \begin{align*}
        f_1(X) &:=
        \begin{cases}
            f(X), & \textrm{if} \; X \in P_{25r}(X_0); \\
            0, &\textrm{if} \; X \notin P_{25r}(X_0). 
        \end{cases}\\
        f_2(X) &:=f(X)-f_1(X). 
    \end{align*}
    \begin{figure}[htbp]
    \centering
   \resizebox{0.3\textwidth}{!}{

\tikzset{every picture/.style={line width=0.75pt}} 
\begin{tikzpicture}[x=0.75pt,y=0.75pt,yscale=-1,xscale=1]

\draw  [fill={rgb, 255:red, 155; green, 155; blue, 155 }  ,fill opacity=1 ][line width=0.75]  (69.5,52) -- (578,52) -- (578,560.5) -- (69.5,560.5) -- cycle ;
\draw  [color={rgb, 255:red, 155; green, 155; blue, 155 }  ,draw opacity=1 ][fill={rgb, 255:red, 255; green, 255; blue, 255 }  ,fill opacity=1 ] (179.69,155) -- (475.69,155) -- (475.69,451) -- (179.69,451) -- cycle ;
\draw  [fill={rgb, 255:red, 0; green, 0; blue, 0 }  ,fill opacity=1 ] (320,52.5) .. controls (320,50.57) and (321.57,49) .. (323.5,49) .. controls (325.43,49) and (327,50.57) .. (327,52.5) .. controls (327,54.43) and (325.43,56) .. (323.5,56) .. controls (321.57,56) and (320,54.43) .. (320,52.5) -- cycle ;
\draw [color={rgb, 255:red, 0; green, 0; blue, 0 }  ,draw opacity=1 ] [dash pattern={on 4.5pt off 4.5pt}]  (327,584) -- (324.01,13) ;
\draw [shift={(324,11)}, rotate = 89.7] [color={rgb, 255:red, 0; green, 0; blue, 0 }  ,draw opacity=1 ][line width=0.75]    (21.86,-6.58) .. controls (13.9,-2.79) and (6.61,-0.6) .. (0,0) .. controls (6.61,0.6) and (13.9,2.79) .. (21.86,6.58)   ;
\draw  [dash pattern={on 4.5pt off 4.5pt}]  (18,296) -- (631,298.99) ;
\draw [shift={(633,299)}, rotate = 180.28] [color={rgb, 255:red, 0; green, 0; blue, 0 }  ][line width=0.75]    (10.93,-3.29) .. controls (6.95,-1.4) and (3.31,-0.3) .. (0,0) .. controls (3.31,0.3) and (6.95,1.4) .. (10.93,3.29)   ;

\draw (329,208.4) node [anchor=north west][inner sep=0.75pt]  [font=\fontsize{2.65em}{3.18em}\selectfont]  {$P_{25r}( X_{0})$};
\draw (346,1.4) node [anchor=north west][inner sep=0.75pt]  [font=\fontsize{2.65em}{3.18em}\selectfont]  {$t$};
\draw (592,314.4) node [anchor=north west][inner sep=0.75pt]  [font=\fontsize{2.65em}{3.18em}\selectfont]  {$M$};
\draw (86,72.4) node [anchor=north west][inner sep=0.75pt]  [font=\fontsize{2.65em}{3.18em}\selectfont]  {$\mathrm{supp} \ f_{2}$};
\draw (483,79.4) node [anchor=north west][inner sep=0.75pt]  [font=\fontsize{2.65em}{3.18em}\selectfont]  {$Q'$};
\draw (325.75,309.65) node [anchor=north west][inner sep=0.75pt]  [font=\fontsize{2.65em}{3.18em}\selectfont]  {$\mathrm{supp} \ f_{1}$};

\end{tikzpicture}

}
    \caption{}
    \label{fig4}
\end{figure}

  \textit{Step 1. The following estimates of $f_1$ holds.}
  \begin{align}
     \|f_1\|_{L^2(\mathcal{Q})}^2 \leq C \varepsilon |P_{5r}(X_0)|. 
  \label{eqn:SK19_6}     
  \end{align}\\
    
    In light of the condition~\eqref{eq3.29}, there exists $X_1\in P_{5r}(X_0)$ such that  
	\begin{equation}
		\begin{split}
			\mathbf{M}(|f_1|^2)(X_1)\leq \mathbf{M}(|f|^2)(X_1)\leq \varepsilon .
		\end{split}
	\end{equation}
	It follows from the definition of maximal function that
	\begin{align*}
         \int_{P_{50r}(X_1)} |f_1|^2 \leq \varepsilon |P_{50r}(X_1)|.
	\end{align*}
    As $f_1$ is supported in $P_{25r}(X_0) \subset P_{50r}(X_1)$, it is clear that 
    \begin{align}
      \|f_1\|_{L^2(\mathcal{Q})}^2=\int_{\mathcal{Q}} |f_1|^2=\int_{P_{50r}(X_1) \cap \mathcal{Q}} |f_1|^2 \leq  \varepsilon |P_{50r}(X_1)|.
    \label{eq3.33}
    \end{align}
    Applying the volume comparison in the last step, we obtain (\ref{eqn:SK19_6}).\\

\textit{Step 2. There exists a point $X_2\in P_{5r}(X_0)\cap \mathcal{Q}$ such that
	\begin{align}
			\mathbf{M}(|f|^2)(X_2)\leq \varepsilon,\quad 	\mathbf{M}(|\mathcal{T}f|^2)(X_2)\leq 1, 
            \quad \mathbf{M}(|\mathcal{T}f_1|^2)(X_2)\leq 1.
    \label{eqn:SK19_2}        
	\end{align}
    Consequently, we have}
    \begin{align}
        &\max\{ \mathbf{M}(|f_1|^2)(X_2), \mathbf{M}(|f_2|^2)(X_2), \mathbf{M}(|f|^2)(X_2)\} \leq \varepsilon, \label{eqn:SK23_1}\\
        &\max\{\mathbf{M}(|\mathcal{T}f_1|^2)(X_2), \mathbf{M}(|\mathcal{T}f_2|^2)(X_2), \mathbf{M}(|\mathcal{T}f|^2)(X_2)\} \leq 4. 
         \label{eqn:SK23_2}
    \end{align}\\
    
	Since $\mathbf{M}$ is of weak type $(1,1)$ and $T$ is of strong type $(2,2)$, we have
	\begin{align*}
			&\quad \left|\{X\in P_{5r}(X_0):\mathbf{M}(|\mathcal{T}f_1|^2)(X)>1\} \right|\\
            &\leq \int_{\mathcal{Q}} \mathbf{M}(|\mathcal{T}f_1|^2) \leq C \int_{\mathcal{Q}} |\mathcal{T}f_1|^2
            = C \|\mathcal{T}f_1\|^2_{L^2(\mathcal{Q})} \leq C \|f_1\|^2_{L^2(\mathcal{Q})}.
	\end{align*}
    Plugging \eqref{eq3.33} into the above inequality yields that
    \begin{equation}
		\begin{split}
			|\{X\in P_{5r}(X_0):\mathbf{M}(|\mathcal{T}f_1|^2)(X)>1\}|
            \leq C\varepsilon |P_{50r}(X_1)|
            \leq \frac13|P_{5r}(X_0)|, 
		\end{split}
	\end{equation}
    where we apply volume comparison and the fact $\varepsilon \in (0, \varepsilon_0)$ is very small in the last step. 
    Consequently, we have
    \begin{align}
        |\{X\in P_{5r}(X_0):\mathbf{M}(|\mathcal{T}f_1|^2)(X)\leq 1\}| 
        \geq \frac23|P_{5r}(X_0)|. 
    \label{eq3.36}    
    \end{align}

	From \eqref{eq3.29} and \eqref{eq3.36}, by considering the measure of the intersection set, we can find a point $X_2\in P_{5r}(X_0)$ satisfying (\ref{eqn:SK19_2}). 

    It follows from definition that
		\begin{align*}
			\max\{\mathbf{M}(|f_1|^2)(X_2), \mathbf{M}(|f_2|^2)(X_2)\} \leq \mathbf{M}(|f|^2)(X_2)\leq \varepsilon, 
		\end{align*}
    which proves (\ref{eqn:SK23_1}). 
	Since $|\mathcal{T}f_2|^2=|\mathcal{T}f-\mathcal{T}f_1|^2\leq 2(|\mathcal{T}f|^2+|\mathcal{T}f_1|^2)$,  the sub-additive property of $\mathbf{M}$ implies that
	\begin{align*}
			\mathbf{M}(|\mathcal{T}f_2|^2)(X_2)\leq 2(\mathbf{M}(|\mathcal{T}f|^2)(X_2)+\mathbf{M}(|\mathcal{T}f_1|^2)(X_2))\leq 4, 
	\end{align*}
    which yields (\ref{eqn:SK23_2}). \\

   \textit{Step 3. For every $X \in P_{5r}(X_0)$, we have}
   \begin{align}
		\mathbf{M}(|\mathcal{T} f_2|^2)(X) < C.  
    \label{eqn:SH23_3}    
	\end{align}

    Note that $\textrm{supp} f_2 \in \mathcal{Q} \backslash P_{25r}(X_0)$ by definition, we have $\textrm{supp} f_2 \in \mathcal{Q} \backslash P_{20r}(X_2)$.
    Therefore, we can apply Proposition~\ref{prn:SK27_2} with $X_2$ and radius $4r$. It follows from  (\ref{eqn:SK23_1}) and (\ref{eqn:SK23_2}) that 
    \begin{align*}
		\mathbf{M}(|\mathcal{T} f_2|^2)(X)\leq C\left\{ \mathbf{M}(|\mathcal{T}f_2|^2)(X_2) + \mathbf{M}(|f_2|^2)(X_2) \right\}
        \leq (4+\varepsilon) C,  
	\end{align*}
   for any $X\in P_{12r}(X_2)$, which yields (\ref{eqn:SH23_3}), by adjusting $C$ slightly.\\

    From now on, until the end of this proof, we fix $C>1$ to be the largest constant that appears in the previous steps.  Define
    \begin{align}
        N^2:= \mathrm{max}\{5^n, 4C^2\}. 
    \label{eqn:SK19_1}    
    \end{align}

    \textit{Step 4.   The following relationship holds.} 
    \begin{equation}
    \label{eqn:SK19_4}
    \begin{split}
       \{X\in P_{5r}(X_0):\mathbf{M}(|\mathcal{T}f|^2)(X)>N^2\}
      \subset 
      \{X\in P_{5r}(X_0):\mathbf{M}(|\mathcal{T}f_1|^2)(X)>N^2/4\}. 
    \end{split}  
    \end{equation}\\

    By sub-additivity of $\mathbf{M}$, we have
	\begin{align*}
		\mathbf{M}(|\mathcal{T}f|^2)(X)\leq 2\left\{ \mathbf{M}(|\mathcal{T}f_1|^2)(X)+\mathbf{M}(|\mathcal{T}f_2|^2)(X) \right\}.
	\end{align*}
    Thus, for any $X\in P_{5r}(X_0)\cap \mathcal{Q}$, we have
	\begin{align*}
		\mathbf{M}(|\mathcal{T}f|^2)(X)\leq 2\left\{ \mathbf{M}(|\mathcal{T}f_1|^2)(X)+ C \right\}
        \leq 2 \mathbf{M}(|\mathcal{T}f_1|^2)(X) + \frac{N^2}{2}. 
	\end{align*}
    Then it is clear that (\ref{eqn:SK19_4}) follows from the above inequality. \\

    \textit{Step 5.  The following estimate holds.}
     \begin{align}
       |\{X\in P_{5r}(X_0):\mathbf{M}(|T f_1|^2)(X)>N^2/4\}|
       \leq \varepsilon |P_{5r}(X_0)|. 
     \label{eqn:SK19_5}
     \end{align}\\

    Since $\mathbf{M}$ is weak type $(1,1)$, we can apply (\ref{eqn:SK19_6}), (\ref{eqn:HA19_1}) to obtain 
  \begin{align*}
       &\quad |\{X\in P_{5r}(X_0):\mathbf{M}(|T f_1|^2)(X)>N^2/4\}|\\
       &\leq  \frac{4 C}{N^2}  \|T f_1\|^2_{L^2} \leq  \frac{4 C^2}{N^2}  \|f_1\|^2_{L^2}
       \leq \frac{4C^2 \varepsilon}{N^2} |P_{5r}(X_0)|.
  \end{align*}
  Plugging the definition (\ref{eqn:SK19_1}) into the last step of the above inequality, we arrive at (\ref{eqn:SK19_5}). \\

  \textit{Step 6. The proof of (\ref{eqn:SH23_4}).}\\
   
  It is clear that (\ref{eqn:SH23_4}) follows directly from the combination of 
  (\ref{eqn:SK19_4}) and (\ref{eqn:SK19_5}). 
 \end{proof}

\begin{lemma}
\label{lma:SJ24_4}
	The same conditions as in Proposition~\ref{prn:SK27_3}. 
    Define
    \begin{align*}
     &U:=\{X\in \mathcal{Q}: \mathbf{M}(|\mathcal{T}f|^2)(X)>N^2\}, \\
     &V :=\{X\in \mathcal{Q}: \mathbf{M}(|\mathcal{T}f|^2)(X)>1\}
     \cup \{X\in \mathcal{Q}: \mathbf{M}(|f|^2)(X)> \varepsilon \}.
   \end{align*} 
   If 
   \begin{align}
       |U|<\varepsilon |\mathcal{Q}|,   \label{eqn:SJ24_1}
   \end{align}
   then
   \begin{align}
    |U|< C(n, \Lambda_0) \varepsilon |V|.     \label{eqn:SJ24_2}
   \end{align} 
   In particular, (\ref{eqn:SJ24_2}) holds if 
   \begin{align}
       \|f\|_{L^2(\mathcal{Q}')} \leq \varepsilon |\mathcal{Q}|.  
   \label{eqn:SK22_4}    
   \end{align} 
\end{lemma}

\begin{proof}
Suppose $\|f\|_{L^2(\mathcal{Q}')}<\varepsilon |\mathcal{Q}|$, then we have 
 \begin{align*}
		 |U|
         =|\{X\in \mathcal{Q}:\mathbf{M}(|\mathcal{T}f|^2)(X) >  N^2\}|\leq \frac{\|\mathbf{M}(\mathcal{T}f)\|^2_{L^2}}{N^2}\leq C \frac{\|f\|^2_{L^2}}{ N^2} < \varepsilon |\mathcal{Q}|.
 \end{align*}
 Thus, (\ref{eqn:SJ24_1}) holds.  Therefore, we only need to focus on the proof that 
 (\ref{eqn:SJ24_1}) implies (\ref{eqn:SJ24_2}). \\

Note that for each $X \in \mathcal{Q}$, the triangle inequality implies that
$ \mathcal{Q} \subset P_{\frac12}(X)$. 
Thus, the condition (\ref{eqn:SJ24_1}) is exactly
\begin{align}
    |U \cap P_{\frac12}(X)|=|U|<\varepsilon |\mathcal{Q}| = \varepsilon |\mathcal{Q} \cap P_{\frac12}(X)|. 
\label{eqn:SK21_1}    
\end{align}

By parabolic version of Lebesgue differentiation theorem, for almost every $X \in U$, we have
\begin{align*}
 &\lim_{r \to 0^+} \frac{|P_r(X) \cap U|}{|P_r(X)|}  
 =\lim_{r \to 0^+} \frac{1}{|P_r(X)|} \int_{P_r(X)} \chi_U =1, \\
 &\lim_{r \to 0^+} \frac{|P_r(X) \cap \mathcal{Q}|}{|P_r(X)|}  
 =\lim_{r \to 0^+} \frac{1}{|P_r(X)|} \int_{P_r(X)} \chi_{\mathcal{Q}} =1. 
\end{align*}
Thus, for almost every $X \in U$, we know
\begin{align}
    \lim_{r \to 0^+} \frac{|P_r(X) \cap U|}{|P_r(X) \cap \mathcal{Q}|}=1>\varepsilon.   
\label{eqn:SK21_2}
\end{align}

Combining (\ref{eqn:SK21_1}) and (\ref{eqn:SK21_2}), we see that for almost every $X\in U$, there exists an $r_X<\frac12$ such that 
\begin{align*}
  |U\cap P_{r_X}(X)|=\varepsilon |P_{r_X}(X)\cap \mathcal{Q}|
\end{align*}
and 
\begin{equation}
 |U\cap P_r(X)|<\varepsilon |P_r(X) \cap \mathcal{Q}|, \quad \forall \; r \in \left(r_X, \frac12 \right). 
\label{eqn:SK19_8} 
\end{equation}
By Vitali covering lemma (parabolic version), there exist countable many points $X_k \in U$ such that
\begin{itemize}
    \item  $P_{r_{X_j}}(X_j) \cap P_{r_{X_k}}(X_k) = \emptyset$, if $j \neq k$.
    \item  $\bigcup_{k}P_{5 r_{X_k}}(X_k)\cap \mathcal{Q}\supset U$. 
\end{itemize}
By the choice of $P_{r_X}(X)$,  we know from Lemma~\ref{lm2.5} that
\begin{align*}
    |V\cap P_{r_{X_k}}(X_k)|\geq \frac12 |P_{r_{X_k}}(X_k)\cap \mathcal{Q}|.  
\end{align*}
In other words, we have
\begin{align}
    |P_{r_{X_k}}(X_k)\cap \mathcal{Q}| \leq 2 |V\cap P_{r_{X_k}}(X_k)|. 
\label{eqn:SK19_7} 
\end{align}
By the choice of $\{P_{5r_{X_j}}(X_j)\}$, 
the volume comparison,  (\ref{eqn:SK19_8}) and (\ref{eqn:SK19_7}), we have
\begin{align*}
        |U|&\leq \sum_k|U\cap P_{5 r_{X_k}}(X_k)|
         \leq \sum_k\varepsilon | P_{5 r_{X_k}}(X_k)\cap \mathcal{Q}|
         \leq  C \varepsilon \sum_k | P_{ r_{X_k}}(X_k)\cap \mathcal{Q}|\\
         &\leq 2 C \epsilon \sum_k |V\cap P_{r_{X_k}}(X_k)| \leq 2C\epsilon |V|, 
\end{align*}
which is exactly (\ref{eqn:SJ24_2}). 
\end{proof}

The proof of Proposition~\ref{prn:SK27_3} is ready now.

\begin{proof}[Proof of Proposition~\ref{prn:SK27_3}:]
Replacing $f$ by $\frac{f}{\sqrt{t}}$, it suffices to prove (\ref{eqn:SK22_1}) in the case $t=1$, 
which is nothing but (\ref{eqn:SJ24_2}). 
\end{proof}

Proposition~\ref{prn:SK27_3} indicates that the measure of the level set $\{\mathbf{M}(|\mathcal{T}f|^2)>t\}$ decays at a very fast speed when $t \to \infty$. This is exactly the reason why the $L^p$-norm of 
$\mathbf{M}(|\mathcal{T}f|^2)$ can be estimated.

\begin{theorem}[Local version of CZ estimate]
\label{thm:HA14_1}
  Let $\mathcal{K}(x,t;y,s)$ be a Calder\'on-Zygmund kernel,  
  $\mathcal{T}$ be the Calder\'on-Zygmund integral operator associated with $\mathcal{K}$, as defined in Definition \ref{def2.1}. 

 For each $p \in (2, \infty)$, there is a positive constant $C=C(n,p,\Lambda_0,C_1,C_2)$ such that
\begin{equation}
    \|\mathcal{T}f\|_{L^p(\mathcal{Q})}\leq C\|f\|_{L^p(\mathcal{Q}')}
\label{eqn:SJ24_3}    
\end{equation}
for every $f\in C_{c}^{\infty}(\mathcal{Q}')$. 
\end{theorem}

\begin{proof}
By the homogeneous property of (\ref{eqn:SJ24_3}), we can always assume that
 \begin{align}
		\|f\|_{L^2(\mathcal{Q}')}^2=\varepsilon |\mathcal{Q}|
 \label{eqn:SK22_5}       
 \end{align}
without loss of generality.  The exact value of $\varepsilon$ will be determined later (cf. (\ref{eqn:SK27_4})).  

 Note that 
 \begin{align}
     \int_{\mathcal{Q}}(\mathbf{M}(|\mathcal{T}f|^2)(X))^{\frac{p}{2}}
	=\frac{p}{2}\int_0^\infty t^{\frac{p}{2}-1}|\{X\in \mathcal{Q}:\mathbf{M}(|\mathcal{T}f|^2)(X)>t\}|dt.
 \label{eqn:SK22_2}   
 \end{align}
 The direct calculation shows that
		\begin{align*}
			&\quad \frac{p}{2}\int_0^\infty t^{\frac{p}{2}-1}|\{X\in \mathcal{Q}:\mathbf{M}(|\mathcal{T}f|^2)(X)>t\}|dt\\
			&=\frac{p}{2}N^p \left( \int_1^\infty +  \int_{0}^1 \right) t^{\frac{p}{2}-1}|\{X\in \mathcal{Q}:\mathbf{M}(|\mathcal{T}f|^2)(X)>N^2t\}|dt\\
            &\leq \frac{p}{2}N^p \int_1^\infty t^{\frac{p}{2}-1}|\{X\in \mathcal{Q}:\mathbf{M}(|\mathcal{T}f|^2)(X)>N^2t\}|dt + \frac{p}{2}N^p|\mathcal{Q}|.
	\end{align*}
 Plugging  (\ref{eqn:SK22_1}) into the above inequality, we have
\begin{align*}
			&\quad \frac{p}{2}\int_0^\infty t^{\frac{p}{2}-1}|\{X\in \mathcal{Q}:\mathbf{M}(|\mathcal{T}f|^2)(X)>t\}|dt\\
			&\leq C \varepsilon \frac{p}{2}N^p\int_1^\infty t^{\frac{p}{2}-1} \bigg(|\{X\in \mathcal{Q}:\mathbf{M}(|\mathcal{T}f|^2)(X)>t \}|
			\bigg)dt\\
            &\quad +C \varepsilon \frac{p}{2}N^p\int_1^\infty |\{X\in \mathcal{Q}:\mathbf{M}(|f|^2)(X)>t \varepsilon \}| dt  +\frac{p}{2}N^p|\mathcal{Q}|\\
            &\leq C \varepsilon \frac{p}{2}N^p\int_1^\infty t^{\frac{p}{2}-1} \bigg(|\{X\in \mathcal{Q}:\mathbf{M}(|\mathcal{T}f|^2)(X)>t \}|
			\bigg)dt\\
            &\quad +C\frac{p}{2}N^p\int_{\varepsilon}^\infty 
            |\{X\in \mathcal{Q}:\mathbf{M}(|f|^2)(X)>t\}| dt  +\frac{p}{2}N^p|\mathcal{Q}|.
\end{align*}
 We fix
 \begin{align}
     \varepsilon = \frac{1}{5 C N^p}.  
 \label{eqn:SK27_4}    
 \end{align}
 It is clear that
    $\varepsilon<\varepsilon_0$, which justifies the application of Lemma~\ref{lma:SJ24_4} in the second step of the above inequality.  
    It follows that
\begin{equation*}
		\begin{split}
		&\quad \frac{p}{2}\int_0^\infty t^{\frac{p}{2}-1}|\{X\in \mathcal{Q}:\mathbf{M}(|\mathcal{T}f|^2)(X)>t\}|dt\\
		&\leq \frac12 \cdot \frac{p}{2} \int_1^\infty t^{\frac{p}{2}-1} 
        |\{X\in \mathcal{Q}:\mathbf{M}(|\mathcal{T}f|^2)(X)>t \}| dt\\
        &\quad +\frac12 \cdot \frac{p}{2\varepsilon}  \int_{0}^\infty 
            |\{X\in \mathcal{Q}:\mathbf{M}(|f|^2)(X)>t\}| dt +\frac{p}{2}N^p|\mathcal{Q}|.
		\end{split}
\end{equation*}
 Consequently, we have
 \begin{equation*}
		\begin{split}
		&\quad \frac{p}{2}\int_0^\infty t^{\frac{p}{2}-1}|\{X\in \mathcal{Q}:\mathbf{M}(|\mathcal{T}f|^2)(X)>t\}|dt\\
		&\leq \frac{p}{2\varepsilon}\int_0^\infty t^{\frac{p}{2}-1}	|\{X\in \mathcal{Q}:\mathbf{M}(|f|^2)(X)>t \}| dt + pN^p|\mathcal{Q}|\\
        &=\varepsilon^{-1} \int_{\mathcal{Q}} \mathbf{M}(|f|^2)^{\frac{p}{2}} +pN^p \cdot \varepsilon^{-1} \|f\|_{L^2}^2\\
        &\leq \varepsilon^{-1} \left\{ \int_{\mathcal{Q}} \mathbf{M}(|f|^2)^{\frac{p}{2}} 
         + p N^p \|f\|_{L^2}^2 \right\}. 
		\end{split}     
 \end{equation*}
 Replacing the first term by (\ref{eqn:SK22_2}) and 
 applying the strong type $(\frac{p}{2},\frac{p}{2})$ inequality of $\mathbf{M}$ in the last step, we obtain 
 \begin{align}
    \int_{\mathcal{Q}}\mathbf{M}(|\mathcal{T}f|^2)^{\frac{p}{2}}
    \leq \varepsilon^{-1} \left\{ \int_{\mathcal{Q}} |f|^p  + p N^p \|f\|_{L^2}^2 \right\}. 
 \label{eqn:SK22_3}   
 \end{align}
 Since $p>2$, the H\"older inequality implies that
 \begin{align*}
     \|f\|_{L^2}^2&=\int_{\mathcal{Q}'} |f|^2 \leq \left(\int_{\mathcal{Q}'} |f|^p \right)^{\frac{2}{p}} 
     \cdot \left(\int_{\mathcal{Q}'} 1 \right)^{1-\frac{2}{p}} 
     \leq \|f\|_{L^p}^2 \cdot |\mathcal{Q}'|^{1-\frac{2}{p}}
     \leq \|f\|_{L^p}^2 \cdot |\mathcal{Q}|^{1-\frac{2}{p}}\\
     &=\|f\|_{L^p}^2 \cdot \left( \varepsilon^{-1} \|f\|_{L^2}^2\right)^{1-\frac{2}{p}}.
 \end{align*}
 Thus, 
 \begin{align}
     \|f\|_{L^2}^2 \leq \varepsilon^{1-\frac{p}{2}} \|f\|_{L^p}^p. 
     \label{eqn:SK22_6}
 \end{align}
 Plugging (\ref{eqn:SK22_6}) into (\ref{eqn:SK22_3}), and noting that $|\mathcal{T}f|^2 \leq \mathbf{M}(|\mathcal{T}f|^2)$ automatically, we obtain
 \begin{align*}
   \int_{\mathcal{Q}} |\mathcal{T}f|^p &\leq  \int_{\mathcal{Q}}\mathbf{M}(|\mathcal{T}f|^2)^{\frac{p}{2}}
   \leq \varepsilon^{-1} \cdot (1+p N^p \varepsilon^{1-\frac{p}{2}}) \cdot \int_{\mathcal{Q}} |f|^p\\
   &=(\varepsilon^{-1} + pN^p \varepsilon^{-\frac{p}{2}}) \int_{\mathcal{Q}} |f|^p
   \leq \left(5CN^p +p(5CN^{p+2})^{\frac{p}{2}} \right) \int_{\mathcal{Q}} |f|^p, 
 \end{align*}
 where we apply the choice of $\varepsilon=\frac{1}{5CN^p}$ in the last step.
 By redefining $C$ to absorb the powers of $N$, we arrive at (\ref{eqn:SJ24_3}). 
\end{proof}

We are now ready to prove the main theorem, which we copy here for the convenience of the readers.

\begin{theorem}[Global version of CZ estimate]
\label{thm:HA12_1}
  Suppose $\{(M,g(t)), -1 \leq t \leq 0\}$ is an evolving manifold satisfying (\ref{eqn:HA02_1}).
  Suppose $\mathcal{T}$ is a $(C_1,C_2)$-Calder\'on-Zygmund integral operator, as defined in Definition~\ref{def2.1}. 

  For each $p \in [2, \infty)$, there is a positive constant $C=C(n,p,\Lambda_0,C_1,C_2)$ such that
      \begin{align}
      \|\mathcal{T}f\|_{L^p(\mathcal{M}')}\leq C\|f\|_{L^p(\mathcal{M})} 
      \label{eqn:HA12_2}    
       \end{align}
   for every $f \in C^{\infty}(\mathcal{M}, \mathcal{F})$.     In particular, if $f$ is supported in $\mathcal{Q}'$, then 
    \begin{align}
      \|\mathcal{T}f\|_{L^p(\mathcal{M})}\leq C\|f\|_{L^p(\mathcal{Q}')}. 
      \label{eqn:HB14_8}    
    \end{align}
\end{theorem}

\begin{proof}

By the definition of $\mathcal{Q}'$, it is clear that the support of $\mathcal{T}f$ is contained in $\mathcal{M}'=M\times [-\frac12, 0]$. Thus, if $f$ is supported in $\mathcal{Q}'$, it follows directly from (\ref{eqn:HA12_2}) that
\begin{align*}
    \|\mathcal{T}f\|_{L^p(\mathcal{M})}=\|\mathcal{T}f\|_{L^p(\mathcal{M}')}
    \leq C\|f\|_{L^p(\mathcal{M})}=C\|f\|_{L^p(\mathcal{Q}')}.   
\end{align*}
Therefore, (\ref{eqn:HB14_8}) is a direct consequence of (\ref{eqn:HA12_2}).

We focus on the proof of (\ref{eqn:HA12_2}). 
In light of Theorem~\ref{thm:HA14_1}, by adjusting some parameters if necessary, we have already obtained
\begin{align*}
    \|\mathcal{T}f\|_{L^p(B_4(x_0) \times [-1, 0])}\leq C\|f\|_{L^p(B_1(x_0) \times [-\frac34, 0])}
\end{align*}
when $\mathrm{supp} f \subset B_1(x_0) \times [-\frac34, 0]$.  
This is a local estimate. We next want to globalize it and try to show
\begin{align}
    \|\mathcal{T}f\|_{L^p( M  \times [-\frac14, 0])}\leq C\|f\|_{L^p(B_1(x_0) \times [-1, 0])}
\label{eqn:HA12_4}    
\end{align}
if $\mathrm{supp} f \subset B_1(x_0) \times [-1, 0]$.  

Fix $x_0$. 
We cover $M \backslash B_2(x_0)$ by unit balls $\{B_1(x_i)\}_{i=1}^{\infty}$ such that $B_{\frac15}(x_i)$ are disjoint to each other.  
Thus, $B_2(x_0) \cup \{\cup_{i=1}^{\infty} B_1(x_i)\}$ is a covering of $M$ with uniformly finite intersection property. 


Recall the expression
 \begin{align*}
     \mathcal{T}f(x,t) = \int_{-1}^t \int_M \mathcal{K}(x,t;y,s) f(y,s)dyds. 
 \end{align*}
For each $i \geq 1$, we denote $d_i :=d(x_i, x_0)$ under the metric $g(0)$.
It is clear that $d_i \geq 1$ by the triangle inequality. 
If $x \in B_1(x_i)$,  then we have
 \begin{align*}
     |\mathcal{T}f|(x,t) &\leq \int_{-1}^t \int_M |\mathcal{K}|(x,t;y,s) |f|(y,s)dyds\\
     &=\int_{-1}^t \int_{B_1(x_0)} |\mathcal{K}|(x,t;y,s) |f|(y,s)dyds\\
     &\leq C\int_{-1}^t\int_{B_1(x_0)} (t-s)^{-\frac{n+2}{2}}V^{-1}_{x_i}(1)e^{-\frac{(d_i-1)^2}{C(t-s)}}|f|(y,s) dyds. 
 \end{align*}
 Note that $\theta^{-\frac{n+2}{2}}e^{-\frac{d^2}{C\theta}}\leq Ce^{-\frac{d^2}{C}}$ if $\theta \in (0,1)$ and $d \geq 1$. 
 Applying the H\"older inequality,  the above inequality can be simplified as 
  \begin{align*}
          \sup_{B_1(x_i) \times [-1, 0]}|\mathcal{T}f|(x,t)
         \leq Ce^{-\frac{d_i^2}{C}} V^{-1}_{x_i}(1)\int_{-1}^t\int_{B_1(x_0)} |f|(y,s)dyds.  
  \end{align*}
  Thus,
  \begin{align*}
      &\quad \|\mathcal{T}f\|_{L^p(B_1(x_i) \times [-1,0])}^p\\
      &\leq Ce^{-\frac{d_i^2}{C}}V_{x_i}^{1-p}(1) \|f\|_{L^1(B_1(x_0)\times [-1,0])}^p
       \leq Ce^{-\frac{d_i^2}{C}} \left\{ \frac{V_{x_0}(1)}{V_{x_0}(1)} \right\}^{p-1} \|f\|_{L^p(B_1(x_0)\times [-1,0])}^p\\
      &=Ce^{-\frac{d_i^2}{C}} \|f\|_{L^p(M \times [-1,0])}^p.
  \end{align*}
Summing the above inequalities over $i$, we have
\begin{align}
  \sum_{i=1}^{\infty} \|\mathcal{T}f\|_{L^p(B_1(x_i) \times [-1,0])}^p
  \leq C \left\{\sum_{i=1}^{\infty}  e^{-\frac{d_i^2}{C}} \right\} \|f\|_{L^p(M
  \times [-1,0])}^p.
\label{eqn:HA12_5}
\end{align}
By the covering property in Lemma~\ref{lm7.1}, we have
\begin{align*}
    \sum_{i=1}^{\infty}  e^{-\frac{d_i^2}{C}}=\sum_{m=1}^{\infty} \sum_{m-1 \leq d_i \leq m} e^{-\frac{d_i^2}{C}}
    \leq \sum_{m=1}^{\infty} e^{-\frac{(m-1)^2}{C}+Cm} < C. 
\end{align*}
Plugging it into (\ref{eqn:HA12_5}), we arrive at 
\begin{align}
  \sum_{i=1}^{\infty} \|\mathcal{T}f\|_{L^p(B_1(x_i) \times [-1,0])}^p \leq C \|f\|_{L^p(M\times [-1,0])}^p. 
\label{eqn:HA14_5}  
\end{align}

Now we estimate $\|\mathcal{T}f\|_{L^p(B_2(x_0) \times [-1,0])}$.  We write
\begin{align*}
    f=\eta f + (1-\eta) f = f_1+f_2. 
\end{align*}

Since $\mathrm{supp} f_1 \subset B_1(x_0) \times [-\frac34, 0]$, Theorem~\ref{thm:HA14_1} implies that
\begin{align}
 \|\mathcal{T}f_1\|_{L^p(B_2(x_0) \times [-1,0])}^p \leq C \|f_1\|_{L^p(B_1(x_0) \times [-\frac34,0])}^p
 \leq C \|f\|_{L^p(M\times [-1,0])}^p. 
\label{eqn:HA14_2} 
\end{align}
As $\mathrm{supp} f_2 \subset B_1(x_0) \times [-1, -\frac12]$, it is clear $t-s \geq \frac14$ if 
$(x,t) \in M \times [-\frac14, 0]$. Thus, it follows from the heat kernel upper bound that
\begin{align*}
     &\quad |\mathcal{T}f_2|(x,t)\\
     &\leq \int_{-1}^t \int_M |\mathcal{K}|(x,t;y,s) |f_2|(y,s)dyds
      =\int_{-1}^{-\frac12} \int_{B_1(x_0)} |\mathcal{K}|(x,t;y,s) |f_2|(y,s)dyds\\
     &\leq C  V^{-1}_{x_0}(1) \int_{-1}^{-\frac12}\int_{B_1(x_0)}|f_2|(y,s) dyds
      \leq C \|f_2\|_{L^p(M \times [-1,0])}. 
 \end{align*}
 Integrating the $p$-power of the above inequality yields
 \begin{align}
      \|\mathcal{T}f_2\|_{L^p(B_2(x_0) \times [-\frac14,0])}^p 
      \leq C \|f_2\|_{L^p(M\times [-1,0])}^p
      \leq C \|f\|_{L^p(M\times [-1,0])}^p. 
 \label{eqn:HA14_3}     
 \end{align}
 Thus, combining (\ref{eqn:HA14_2}) and (\ref{eqn:HA14_3}) yields that 
 \begin{align}
 \begin{split}
     \|\mathcal{T}f\|_{L^p(B_2(x_0) \times [-\frac14,0])}
     &\leq C \left\{ \|\mathcal{T}f_1\|_{L^p(B_2(x_0) \times [-1,0])} + \|\mathcal{T}f_2\|_{L^p(B_2(x_0) \times [-1,0])} \right\}\\
     &\leq C \|f\|_{L^p(M\times [-1,0])}. 
 \end{split}  
 \label{eqn:HA14_4}
 \end{align}
As $B_2(x_0) \cup \{\cup_{i=1}^{\infty} B_1(x_i)\}$ is a covering of $M$ with a uniformly finite intersection property, it follows from (\ref{eqn:HA14_5}) and (\ref{eqn:HA14_4}) that
\begin{align*}
     \|\mathcal{T}f\|_{L^p(M \times [-\frac14,0])}^p &\leq C \left\{ \|\mathcal{T}f\|_{L^p(B_2(x_0) \times [-\frac14,0])}^p
     +\sum_{i=1}^{\infty} \|\mathcal{T}f\|_{L^p(B_1(x_i) \times [-\frac14,0])}^p \right\}\\
     &\leq C \|f\|_{L^p(M\times [-1,0])}^p,
\end{align*}
which is exactly (\ref{eqn:HA12_4}).  By partition of unity and the linear property of $T$, we can replace $f$ by any smooth function which is compactly supported in $M$. 

\end{proof}

\section{Applications}
\label{sec7}

In this section, as applications of our main theorem,  
we prove Theorem~\ref{thm:HA02_1}, Theorem~\ref{thm:HA02_3}, and Theorem~\ref{thm:HA02_4}.   
      	
    \begin{theorem}
    \label{thm:HA02_1}
		Let $\{(M,g(t)), -1 \leq t \leq 0\}$ be an evolving manifold that satisfies (\ref{eqn:HA02_1}).
        Let $u$ be a smooth function such that
        \begin{align}
            (\partial_t -\Delta)u=f. 
        \end{align}
        Then for any $p\in [2,\infty)$, there exists a positive constant $C=C(n,p,\Lambda_0)$ such that 
        \begin{align}
           \label{eqn:SJ27_1}
            \begin{split}
                \|\dot{u}\|_{L^p(\mathcal{M}')}+\|\mathrm{Hess}(u)\|_{L^p(\mathcal{M}')}
                \leq C \left\{ \|u\|_{L^p(\mathcal{M})}+\|f\|_{L^p(\mathcal{M})} \right\}
            \end{split}
	\end{align}
	for any  $u\in C^\infty_c(\mathcal{M})$.
	\end{theorem}

    Theorem \ref{thm:HA02_1} is valid only if $u$ is a function.
    If $u$ is a tensor field, we either need a stronger regularity assumption of the underlying space-time or we need to weaken the estimate.

    \begin{theorem}
    \label{thm:HA02_3}
		Let $\{(M,g(t)), -1 \leq t \leq 0\}$ be a Ricci flow solution satisfying (\ref{eqn:HA02_1}).
        Let $u,f$ be smooth sections of $\mathrm{Sym}^2(T^*M)$ satisfying
		\begin{align}
			(\partial_t -\Delta) u=f. 
		\end{align}
		Then for each $p \in [2, \infty)$,  there exists a positive constant 
        $C=C(n,p,\Lambda_0)$ such that  
         \begin{align}
           \label{eqn:SJ27_1A}
            \begin{split}
                \|\dot{u}\|_{L^p(\mathcal{M}')}+\|\mathrm{Hess}(u)\|_{L^p(\mathcal{M}')}
                \leq C \left\{ \|u\|_{L^p(\mathcal{M})}+\|f\|_{L^p(\mathcal{M})} \right\}.
            \end{split}
	\end{align}
    \end{theorem}

    \begin{theorem}
    \label{thm:HA02_4}
    Let $\{(M,g(t)), -1 \leq t \leq 0\}$ be an evolving manifold that satisfies (\ref{eqn:HA02_1}). 
    Let $u \in C^{\infty}(\mathrm{Sym}^2(T^*M))$ be a solution of 
		\begin{align}
			(\partial_t -\Delta_L) u=\nabla^*f
		\end{align}
    where $f\in C^{\infty} (\mathrm{Sym}^2(T^*M) \otimes T^*M)$,  
    $\Delta_L$ is the Lichnerowicz Laplacian, and $\nabla^*$ is the formal adjoint operator of $\nabla$. Then for each $p \in [2, \infty)$,  there exists a positive constant 
        $C=C(n,p,\Lambda_0)$ such that  
		\begin{align}
            \label{eqn:HA11_1}
			\|\nabla u\|_{L^p(\mathcal{M}')}
            \leq C \left\{ \|u\|_{L^p(\mathcal{M})}+\|f\|_{L^p(\mathcal{M})} \right\}. 
		\end{align}
    \end{theorem}

Basically, we need to show that for each theorem we can find a corresponding Calder\'on-Zygmund kernel and operator.

\begin{proposition}
\label{prn:HA25_3}
Let $\{(M,g(t)), -1 \leq t \leq 0\}$ be an evolving manifold that satisfies (\ref{eqn:HA02_1}).
Let $H$ be the heat kernel of $\Delta$. 
Let $\mathcal{K}(x,t;y,s) := \nabla_x \nabla_x H(x,t;y,s)$ and
$\mathcal{T}$ be the convolution operator with $\mathcal{K}$:
\begin{align*}
    \mathcal{T} f(x,t) :=\int_{-1}^t \int_M \mathcal{K}(x,t;y,s)f(y,s) dyds.  
\end{align*}
Then $\mathcal{T}$ is a Calder\'{o}n-Zygmund kernel in the sense of Definition~\ref{def2.1}.
\end{proposition}

\begin{proof}
By Theorem~\ref{thm3.2}, we have the proper heat kernel estimate. 
Therefore, in order to show that $\mathcal{T}$ is a Calder\'{o}n-Zygmund kernel, 
it suffices to show the following estimate:
 \begin{align}
 \label{eqn:HA25_2}
	\begin{split}
	 \int_{-1}^0\int_{M} |\nabla \nabla u|^2 d\mu dt \leq C(n, \Lambda_0) \cdot \int_{-1}^0\int_{M}|f|^2  d\mu dt,
	\end{split}
 \end{align}
 where $f$ is compactly supported and 
 \begin{align}
 \label{eqn:HA28_4}
     u(x,t) :=\int_{-1}^t \int_M H(x,t;y,s)f(y,s) dyds.  
 \end{align}
 We shall prove (\ref{eqn:HA25_2}) in the case of static metric, i.e. $g(t) \equiv g$. The Ricci flow case is almost the same, and its proof is easier than the corresponding estimate for the $(0,2)$ -tensor field in Proposition~\ref{prn:HA25_6}. 
 So we leave it to interested readers.

   It follows from (\ref{eqn:HA25_2}) that
   \begin{align}
            (\partial_t -\Delta)u=f, \quad u(\cdot, -1) \equiv 0, \quad \int_{-1}^0\int_M |u|^2<\infty. 
   \end{align} 

 The proof is basically divided into 3 steps.  We estimate the following integral in turn:
 \begin{itemize}
     \item $\int_{-1}^0\int_M (|\nabla u|^2+|u|^2)$;
     \item $\int_{-1}^0\int_M  (\Delta u)^2$;
     \item $\int_{-1}^0\int_M |\nabla \nabla u|^2$. 
 \end{itemize}
 The key ingredients for these estimates are the integration by parts and the Bochner formula. \\

    \textit{Step 1.  The $W_2^1$ estimate holds:}
    \begin{align}
        \int_{-1}^0 \int_M \{ |u|^2 + |\nabla u|^2\}
        \leq 8 \int_{-1}^0 \int_M |f|^2. 
    \label{eqn:HA26_5}    
    \end{align}
 
    We choose a cut-off function $\varphi$ such that $\varphi=1$ in $B_{r}(q)$ and $\varphi=0$ out of $B_{r+1}(q)$.  Furthermore, $|\nabla \varphi| \leq 2$. 
    The direct calculation shows that
   \begin{align}
   \label{eqn:HA25_11}
       \begin{split}
        &\quad \int_{-1}^t\int_M \varphi^2\langle f,u\rangle
         =\int_{-1}^t\int_M \varphi^2\langle \dot{u},u \rangle -\varphi^2\langle \Delta u, u\rangle\\
        &= \frac12\int_M \varphi^2|u(\cdot,t)|^2 +\int_{-1}^t\int_M \varphi^2|\nabla u|^2+2\varphi u  \langle \nabla \varphi, \nabla u \rangle\\
        &\geq \frac{1}{2}\int_M \varphi^2|u(\cdot,t)|^2+\frac12 \int_{-1}^t\int_M \varphi^2|\nabla u|^2-\int_{-1}^t\int_M 2|\nabla \varphi|^2|u|^2.
\end{split}
   \end{align}
 By the Cauchy-Schwarz inequality, we have
    \begin{align}
        \int_{-1}^t\int_M \varphi^2\langle f,u\rangle\leq \int_{-1}^t\int_M \varphi^2|f|^2+\frac{1}{4}\int_{-1}^t\int_M \varphi^2|u|^2.
    \label{eqn:HA25_12}    
    \end{align}
    Combining (\ref{eqn:HA25_11}) and (\ref{eqn:HA25_12}) yields that
    \begin{equation*}
       \begin{split}
         \int_{-1}^t\int_M \bigg(\varphi^2|f|^2+\frac{1}{4}\varphi^2|u|^2+  2|\nabla \varphi|^2|u|^2\bigg)\geq \frac{1}{2}\int_M \varphi^2|u(\cdot,t)|^2+\frac12
         \int_{-1}^t\int_M \varphi^2|\nabla u|^2.
       \end{split}
   \end{equation*}
 Letting $r\to \infty$ and noting that $t \leq 0$,  we obtain
\begin{align}
\label{eqn:HA25_13}
       \begin{split}
         \int_{-1}^t\int_M \bigg(|f|^2+\frac{1}{4}|u|^2\bigg)
         \geq 
         \frac12\int_{-1}^t\int_M |\nabla u|^2, 
\end{split}
\end{align}
and
 \begin{equation}\label{eq2-31}
       \begin{split}
         \int_{-1}^0\int_M \bigg(|f|^2+\frac{1}{4}|u|^2\bigg)\geq \frac{1}{2}\int_M |u(\cdot,t)|^2.
\end{split}
\end{equation}
Recall that $\int_{-1}^0\int_M |u|^2<\infty$. 
Integrating~\eqref{eq2-31} from $-1$ to $0$, we have
   \begin{equation}
   \label{eq1.22}
       \begin{split}
         \int_{-1}^0\int_M |f|^2  
           \geq \frac{1}{4}\int_{-1}^0\int_M |u|^2.
       \end{split}
   \end{equation}
Letting $t=0$ in~\eqref{eqn:HA25_13}, we obtain
   \begin{equation}
   \label{eq1.21}
       \begin{split}
         \int_{-1}^0\int_M \bigg(|f|^2+\frac{1}{4}|u|^2\bigg)\geq \frac12
     \int_{-1}^0\int_M |\nabla u|^2.
    \end{split}
   \end{equation}
Thus, (\ref{eqn:HA26_5}) follows from the combination of (\ref{eq1.22}) and (\ref{eq1.21}). \\

\textit{Step 2. The integral of Laplacian square is bounded:}
   \begin{equation}
   \label{eqn:HA26_6}
     \int_{-1}^0\int_M |\Delta u|^2\leq  \int_{-1}^0\int_M |f|^2.
   \end{equation}

   Note that
    \begin{equation}
    \label{eq1.18}
    \begin{split}
    &\quad \int_{-1}^0\int_M \varphi^2|f|^2
      =\int_{-1}^0\int_M \varphi^2( |\dot{u}|^2+|\Delta u|^2-2\dot{u}\Delta u) \\
    &=\int_{-1}^0\int_M  \big\{\varphi^2(|\dot{u}|^2+|\Delta u|^2)
     +4\varphi\dot{u} \langle \nabla u,  \nabla \varphi \rangle \big\}+2\int_{-1}^0\int_M \varphi^2 \langle \nabla \dot{u}, \nabla u \rangle\\
    &\geq \int_{-1}^0\int_M 
    \left\{\varphi^2 |\Delta u|^2-4|\nabla \varphi|^2|\nabla u|^2 \right\}+2\int_{-1}^0\int_M \varphi^2 \langle \nabla\dot{u}, \nabla u \rangle.
   \end{split}
   \end{equation}
   The last term can be simplified as
    \begin{equation*}
    \begin{split}
     2\int_{-1}^0\int_M \varphi^2 \langle \nabla\dot{u}, \nabla u \rangle 
     =\int_{-1}^0 \left\{\partial_t \int_M \varphi^2 |\nabla u|^2 \right\}
     =\int_M \varphi^2|\nabla u(\cdot,0)|^2\geq 0.
   \end{split}
    \end{equation*}
    Putting it into (\ref{eq1.18}) implies
    \begin{align}
    \int_{-1}^0\int_M \varphi^2 |\Delta u|^2
    \leq \int_{-1}^0\int_M \left\{ \varphi^2|f|^2 + 4|\nabla \varphi|^2|\nabla u|^2 \right\}.
    \label{eqn:HA26_7}
    \end{align}
    In light of the finiteness of $\int_{-1}^0\int_M |\nabla u|^2$, the integral $\int_{-1}^0\int_M |\nabla \varphi|^2|\nabla u|^2 \to 0$ as $r \to \infty$. 
    Letting $r\to \infty$ in (\ref{eqn:HA26_7}), 
    we arrive at (\ref{eqn:HA26_6}). \\

   \textit{Step 3. The integral of Hessian square is bounded:}
   \begin{align}
        \int_{-1}^0 \int_M |\nabla^2 u|^2
        \leq \int_{-1}^0 \int_M  \left\{ |\Delta u|^2 + (n-1)\Lambda_0 |\nabla u|^2 \right\}.   
   \label{eqn:HA26_9}     
   \end{align}

 The standard Bochner formula implies
   \begin{align*}
     |\nabla^2 u|^2= -\langle \nabla u, \nabla \Delta u\rangle -Rc(\nabla u, \nabla u) + \frac12 \Delta |\nabla u|^2.   
   \end{align*}
   Multiplying both sides by $\varphi^2$ and integrating on $M \times [-1, 0]$ yields
   \begin{align}
   \label{eqn:HA26_8}
   \begin{split}
     &\quad \int_{-1}^0 \int_M \varphi^2|\nabla^2 u|^2\\
     &=\int_{-1}^0 \int_M \left\{ \varphi^2 |\Delta u|^2 + 2\varphi  \Delta u \langle \nabla \varphi, \nabla u\rangle -\varphi^2 Rc(\nabla u, \nabla u) - \varphi \langle \nabla \varphi, \nabla |\nabla u|^2\rangle \right\}\\
     &\leq \int_{-1}^0 \int_M  \varphi^2\left\{ |\Delta u|^2 +(n-1)\Lambda_0|\nabla u|^2 \right\}
      +4 \int_{-1}^0 \int_M \varphi  |\nabla u|\left\{ |\Delta u| + |\nabla^2 u| \right\}\\
     &\leq \int_{-1}^0 \int_M  \varphi^2\left\{ |\Delta u|^2 + (n-1)\Lambda_0|\nabla u|^2 \right\}
      + C(n) \int_{-1}^0 \int_M \varphi  |\nabla u| \cdot |\nabla^2 u|. 
   \end{split}   
   \end{align}
   
 By Theorem~\ref{thm3.2}, we have
 \begin{equation}
     \begin{split}
         &\quad |\nabla^2 u|(x,t)=\bigg|\int_{-1}^t\int_M\nabla^2_x H(x,t;y,s)f(y,s)dyds\bigg|\\
         &\leq C\int_{-1}^t\int_{B_1(q)}(t-s)^{-1}V_y(\sqrt{t-s})e^{-\frac{d^2(x,y)}{C(t-s)}}|f|(y,s)dyds.
     \end{split}
 \end{equation}
For any $y\in B_1(q)$, the volume comparison implies that
 \begin{equation}
     V_y(\sqrt{t-s})\leq CV^{-1}_y(2)(t-s)^{-n/2}\leq CV^{-1}_q(1)(t-s)^{-n/2}.
 \end{equation}
 Thus, by adjusting $C$ and absorbing extra terms into $e^{-\frac{d^2}{C}}$, we have
 \begin{equation}
 \label{eq2-43}
     \begin{split}
         &\quad |\nabla^2 u|(x,t)=\bigg|\int_{-1}^t\int_M\nabla^2_x H(x,t,y,s)f(y,s)dyds\bigg|\\
         &\leq C\int_{-1}^t\int_{B_1(q)} V_q^{-1}(1)(t-s)^{-(n+2)/2}e^{-\frac{(d(x,q)-1)^2}{C(t-s)}}|f|(y,s)dyds\\
         &\leq Ce^{\frac{-d^2(x,p)}{C}}\int_{0}^t\int_{B_1(q)}CV^{-1}_q(1)|f|(y,s)dyds.\\
         &\leq Ce^{-\frac{d^2(x,p)}{C}}.
     \end{split}
 \end{equation}
 Using the same trick, we have
 \begin{equation}\label{eq2-44}
     \begin{split}
         |\nabla u|(x,t)\leq Ce^{-\frac{d^2(x,p)}{C}}. 
     \end{split}
 \end{equation}
 Inserting \eqref{eq2-43} and \eqref{eq2-44} into (\ref{eqn:HA26_8}) and letting $r \to \infty$, we arrive at (\ref{eqn:HA26_9}). \\
 
It follows from the combination of
 (\ref{eqn:HA26_5}), (\ref{eqn:HA26_6}) and (\ref{eqn:HA26_9}) that 
 \begin{align}
 \label{eqn:HA29_5}
	\begin{split}
      \int_{-1}^0\int_{M} |\nabla \nabla u|^2 d\mu dt 
     &\leq \left\{ 4(n-1) \Lambda_0 +1\right\} \int_{-1}^0\int_{M}|f|^2  d\mu dt\\
     &\leq 4n \Lambda_0 \int_{-1}^0\int_{M}|f|^2  d\mu dt, 
	\end{split}
 \end{align}
 which yields (\ref{eqn:HA25_2}) by setting $C(n, \Lambda_0)=4n\Lambda_0$.  
\end{proof}

From the proof of Proposition~\ref{prn:HA25_3}, it is clear that the choice of the cutoff function $\varphi$ is technical.
In each step, we use $\varphi$ to guaranty the application of integration by parts and then let $r \to \infty$ push the support of $\nabla \varphi$ to infinity. Then we obtain the inequality without $\varphi$, which is the same as the one on the closed manifold.  For simplicity of argument, we shall ignore the application of $\varphi$ in the following proof.

\begin{proposition}
\label{prn:HA25_6}
Let $\{(M,g(t)), -1 \leq t \leq 0\}$ be a Ricci flow solution satisfying (\ref{eqn:HA02_1}).
Let $\Psi$ be the heat kernel of $\Delta$ and $\mathcal{K}(x,t;y,s) := \nabla_x \nabla_x \Psi(x,t;y,s)$ and
$\mathcal{T}$ be the convolution operator with $\mathcal{K}$. Then $\mathcal{T}$ is a Calder\'{o}n-Zygmund kernel in the sense of Definition~\ref{def2.1}.
\end{proposition}

\begin{proof}
    Using Theorem~\ref{thm:HA02_5}, we have the desired heat kernel estimate. 
    Setting
    \begin{align}
    \label{eqn:HA28_5}
        u(x,t):=\int_{-1}^t \int_M \Psi(x,t;y,s)f(y,s)dyds.
    \end{align}
    In order for $\mathcal{T}$ to be a Calder\'{o}n-Zygmund kernel, it suffices to show the following $L^2$-estimate:
    \begin{align}
        \label{eqn:HA30_6}
			\begin{split}
				\int_{-1}^0\int_{M} |\mathrm{Hess}\; u|^2 d\mu dt \leq C(n, \Lambda_0) \cdot \int_{-1}^0\int_{M}|f|^2  d\mu dt.
			\end{split}
	\end{align}
    From (\ref{eqn:HA28_5}), we know
    \begin{align*}
        (\partial_t -\Delta)u=f, \quad u(\cdot, -1) \equiv 0, \quad \int_{-1}^0 \int_M |u|^2 <\infty. 
    \end{align*} 
    Note that under assumption $(\partial_t-\Delta)u=f$, the quotient is
    \begin{align*}
        \frac{\int_{-1}^0\int_{M} |\mathrm{Hess}\; u|^2 d\mu dt}{\int_{-1}^0\int_{M}|f|^2  d\mu dt}
    \end{align*}
    is scaling invariant. Therefore, after proper parabolic rescaling, we assume 
    \begin{itemize}
        \item The Ricci flow exists on $[-\Lambda_0 \xi^{-2}, 0]$. 
        \item $\displaystyle \sup_{t\in [-\Lambda_0 \xi^{-2},0]}|Rm| \leq \xi^{2}$.
        \item $(\partial_t-\Delta u)=f$ where $f$ is compactly supported.
        \item $u(\cdot, -\Lambda_0 \xi^{-2}) \equiv 0$ and $\int_{-\Lambda_0 \xi^{-2}}^0\int_M |u|^2<\infty$. 
    \end{itemize}
    Therefore, it suffices to show the following inequality under the above assumptions: 
    \begin{align}
    \label{eqn:HA25_5}
    \begin{split}
     \int_{-\Lambda_0 \xi^{-2}}^0\int_{M} |\mathrm{Hess}\; u|^2 d\mu dt 
     \leq 
     C(n, \Lambda_0) 
     \cdot \int_{-\Lambda_0 \xi^{-2}}^0\int_{M}|f|^2  d\mu dt.
    \end{split}
    \end{align}
    We divide the proof into four steps.\\

\textit{Step 1. We have}
 \begin{align}
   \label{eqn:HA27_5}    
    \int_{-1}^0 \int_M \left\{ |u|^2+|\nabla u|^2 \right\}
    \leq 8 \left\{\int_{-1}^0\int_M |f|^2 + \int_M |u(\cdot,-1)|^2 \right\}
   \end{align}
   \textit{and}
  \begin{align}
  \label{eqn:HA29_11}
      \int_M |u(\cdot,t)|^2
    \leq 8 \left\{\int_{-1}^0\int_M |f|^2 +\int_M |u(\cdot,-1)|^2\right\}, \quad \forall \; t \in [-1, 0]. 
  \end{align} 

Note that $u=u_{ij}dx^idx^j$ is a $(0,2)$-tensor. Thus,
\begin{align*}
    \partial_t \left\{ |u|^2 d\mu \right\}&=\partial_t \left\{u_{ij}u_{kl}g^{ik}g^{jl} \sqrt{\det g} dx^1 \wedge dx^2 \cdots \wedge dx^n \right\}\\
     &=\{2\dot{u}_{ij}u_{ij}+4u_{ij}u_{jk}R_{ik} -R|u|^2\} d\mu,
\end{align*}
whose integration yields
\begin{align*}
    \int_M |u(\cdot, t)|^2 d\mu-\int_M |u(\cdot,-1)|^2 
    =\int_{-1}^0 \int_M \{2 \langle \dot{u}, u\rangle+4u_{ij}u_{jk}R_{ik} -R|u|^2\} d\mu dt.
\end{align*}
Therefore, we have
\begin{align}
\label{eqn:HB02_1}
\begin{split}
   &\quad \int_{-1}^t \int_M  \langle \dot{u}, u\rangle d\mu dt\\ 
   &=\frac12\int_M |u(\cdot,t)|^2-\frac12\int_M |u(\cdot,-1)|^2+\int_{-1}^t \int_M \{-2 u_{ij}u_{jk}R_{ik} 
    +\frac12 R|u|^2\} d\mu dt\\
   &\geq \frac12\int_M |u(\cdot,t)|^2-\frac12\int_M |u(\cdot,-1)|^2 -\frac{1}{16} \int_{-1}^t\int_M |u|^2. 
\end{split}   
\end{align}
Consequently, we obtain 
 \begin{equation}
 \label{eq2-57}
       \begin{split}
        &\quad \int_{-1}^t\int_M \langle f,u\rangle=\int_{-1}^t\int_M \langle \dot{u},u \rangle -\langle \Delta u, u\rangle\\
        &\geq \frac12\int_M |u(\cdot,t)|^2 -\frac12\int_M |u(\cdot,-1)|^2 -\frac{1}{16}\int_{-1}^t\int_M |u|^2+\int_{-1}^t\int_M |\nabla u|^2.   
       \end{split}
   \end{equation}
 On the other hand, the Cauchy-Schwarz inequality implies
    \begin{equation}
        \int_{-1}^t\int_M \langle f,u\rangle\leq \int_{-1}^t\int_M |f|^2+\frac{1}{4}\int_{-1}^t\int_M |u|^2.
    \end{equation}
 Thus, we have
    \begin{equation*}
       \begin{split}
       \frac12\int_M |u(\cdot,-1)|^2+
        \int_{-1}^t\int_M \bigg(|f|^2+\frac{5}{16}|u|^2\bigg)
        \geq 
        \frac{1}{2}\int_M |u(\cdot,t)|^2
        +\int_{-1}^t\int_M |\nabla u|^2.
       \end{split}
   \end{equation*}
  Since $t \leq 0$, the $t$ on the left hand side can be replaced by $0$. Thus, we have
  \begin{equation}
  \label{eqn:HA27_1}
       \begin{split}
        \frac12\int_M |u(\cdot,-1)|^2+
        \int_{-1}^0\int_M \bigg(|f|^2+\frac{5}{16}|u|^2\bigg)
        \geq 
        \frac{1}{2}\int_M |u(\cdot,t)|^2, 
       \end{split}
   \end{equation}
   and
  \begin{equation}
  \label{eqn:HA27_2}
       \begin{split}
        \frac12\int_M |u(\cdot,-1)|^2+
        \int_{-1}^0\int_M \bigg(|f|^2+\frac{5}{16}|u|^2\bigg)
        \geq 
        \int_{-1}^0 \int_M |\nabla u|^2.
       \end{split}
   \end{equation}

   Integrating (\ref{eqn:HA27_1}) on $[-1, 0]$, we obtain
   \begin{align*}
       \begin{split}
         \frac12\int_M |u(\cdot,-1)|^2+
         \int_{-1}^0\int_M |f|^2  
           \geq \frac{3}{16} \int_{-1}^0\int_M |u|^2, 
       \end{split}
   \end{align*}
   which we rewrite as
   \begin{align}
   \label{eqn:HA27_3}
       \int_{-1}^0\int_M |u|^2 \leq \frac{16}{3} \int_{-1}^0\int_M |f|^2 + \frac{8}{3} \int_M |u(\cdot,-1)|^2.
   \end{align}
   Plugging (\ref{eqn:HA27_3}) into (\ref{eqn:HA27_2}) and (\ref{eqn:HA27_1})  yields
   \begin{align}
   \label{eqn:HA27_4}
       \int_{-1}^0 \int_M |\nabla u|^2
       \leq \frac{8}{3} \int_{-1}^0\int_M |f|^2 + \frac43 \int_M |u(\cdot,-1)|^2,  
   \end{align}
   and 
   \begin{align}
    \frac12 \int_M |u(\cdot,t)|^2
    \leq  \frac{8}{3}\int_{-1}^0\int_M |f|^2 +\frac{4}{3} \int_M |u(\cdot,-1)|^2.
   \label{eqn:HA31_1} 
   \end{align}
   Combining (\ref{eqn:HA27_3}) and (\ref{eqn:HA27_4}), we arrive at
   \begin{align}
       \int_{-1}^0 \int_M \left\{|u|^2+|\nabla u|^2 \right\}
       \leq 8 \int_{-1}^0\int_M |f|^2 + 4 \int_M |u(\cdot,-1)|^2. 
   \label{eqn:HA31_2}    
   \end{align}
   Therefore, (\ref{eqn:HA27_5}) follows from (\ref{eqn:HA31_2}), and (\ref{eqn:HA29_11}) follows from (\ref{eqn:HA31_1}). \\

 \textit{Step 2. We have}
 \begin{align}
 \label{eqn:HA28_1}
      \int_{-1}^0\int_M |\Delta u|^2 \leq 5 \left\{ \int_{-1}^0 \int_M |f|^2 +\int_M \left\{|u(\cdot,-1)|^2 +|\nabla u(\cdot, -1)|^2 \right\} \right\}, 
 \end{align}
 \textit{and for each $t \in [-1,0]$ that} 
 \begin{align}
 \label{eqn:HA28_1A}
 \begin{split}
   &\quad \int_M \left\{|u(\cdot,t)|^2 +|\nabla u(\cdot, t)|^2 \right\}\\
   &\leq 10 \left\{ \int_{-1}^t\int_M |f|^2 + \int_M \left\{|u(\cdot,-1)|^2 +|\nabla u(\cdot, -1)|^2 \right\} \right\}. 
 \end{split}  
 \end{align}

    It follows directly from the equation $(\partial_t-\Delta)u=f$ that
     \begin{equation}
     \label{eqn:HA26_0}
     \begin{split}
   \int_{-1}^0\int_M |f|^2=\int_{-1}^0\int_M |\dot{u}-\Delta u|^2
   =\int_{-1}^0\int_M  \left\{|\dot{u}|^2+|\Delta u|^2-2 \langle \dot{u}, \Delta u \rangle \right\}.
   \end{split}
    \end{equation}
  Note that 
  \begin{align}
  \label{eqn:HA27_6}
  \begin{split}
      &\quad \frac{d}{dt} \left\{ \langle u, \Delta u\rangle d\mu \right\}\\
      &=\left\{-R \langle u, \Delta u \rangle + 4 u_{ik} (\Delta u)_{jl} R_{kl} 
      +\langle \dot{u}, \Delta u\rangle + \langle u, (\dot{\Delta} u+ \Delta \dot{u})\rangle \right\} d\mu\\
      &=\left\{-R \langle u, \Delta u \rangle + 4 u_{ik} (\Delta u)_{jl} R_{kl} 
      +2\langle \dot{u}, \Delta u\rangle + \langle u, \dot{\Delta} u\rangle \right\} d\mu.
  \end{split}    
  \end{align}
  Along the Ricci flow, the direct calculation shows that 
  \begin{align*}
    \dot{\Delta} u_{ij}
    &=2u_{ja,l}(R_{ak,i}+R_{ai,k}-R_{ki,a})g^{kl}+2u_{ia,l}(R_{ak,j}+R_{aj,k}-R_{kj,a})g^{kl}\\
    &\quad +u_{ja}(R_{ak,il}+R_{ai,kl}-R_{ki,al})g^{kl}
     +u_{ia}(R_{ak,jl}+R_{aj,kl}-R_{kj,al})g^{kl}.
 \end{align*}
 As $|Rm|+|\nabla Rm|+|\nabla \nabla Rm| \leq \xi_n$ is very small,  the above inequality implies 
 \begin{align}
     |\dot{\Delta} u| \leq \frac18 (|u| + |\nabla u|). 
 \label{eqn:HA27_8}    
 \end{align}
 It follows from (\ref{eqn:HA27_6}) and (\ref{eqn:HA27_8}) that 
 \begin{align*}
 \begin{split}
     &\quad -2\int_{-1}^t\int_M \langle \dot{u}, \Delta u \rangle\\
     &=-\left.\int_M \langle u, \Delta u \rangle \right|_{-1}^t 
      +\int_{-1}^t \int_M \left\{ -R \langle u, \Delta u\rangle + 4 u_{ik}(\Delta u)_{jl}R_{kl} 
      +\langle u, \dot{\Delta} u\rangle \right\}\\
     &\geq \left.\int_M |\nabla u|^2 \right|_{-1}^t -\int_{-1}^t \int_M \frac18 |u| (|u|+|\nabla u| + |\Delta u|). 
 \end{split}    
 \end{align*}
 Using $\frac18 xy \geq -\frac18 x^2-\frac{1}{32}y^2$ in the last step, we obtain
 \begin{align}
 \label{eqn:HA27_9}
 \begin{split}
    &\quad -2\int_{-1}^t\int_M \langle \dot{u}, \Delta u \rangle\\
    &\geq -\frac38 \int_{-1}^t \int_M |u|^2 -\frac{1}{32} \int_{-1}^t \int_M (|\nabla u|^2 + |\Delta u|^2)
     +\left.\int_M |\nabla u|^2 \right|_{-1}^t.
 \end{split}   
 \end{align}
 Plugging (\ref{eqn:HA27_9}) into (\ref{eqn:HA26_0}) and using (\ref{eqn:HA27_5}), we obtain
 \begin{align*}
   &\quad \int_{-1}^t\int_M \{ |\dot{u}|^2 +|\Delta u|^2\} +\int_M |\nabla u(\cdot, t)|^2\\
   &\leq \int_{-1}^t\int_M |f|^2 +\frac38 \int_{-1}^t \int_M |u|^2 +\frac{1}{32} \int_{-1}^t \int_M (|\nabla u|^2 + |\Delta u|^2) +\int_M |\nabla u(\cdot, -1)|^2\\
   &\leq 4 \int_{-1}^t\int_M |f|^2 +3 \int_M \left\{|u(\cdot,-1)|^2 +|\nabla u(\cdot, -1)|^2 \right\} + \frac{1}{32} \int_{-1}^t \int_M  |\Delta u|^2.  
 \end{align*}
 Consequently, we have
 \begin{align}
 \label{eqn:HA31_5}
 \begin{split}
   &\quad \frac{31}{32}\int_{-1}^t\int_M \{ |\dot{u}|^2 +|\Delta u|^2\} +\int_M |\nabla u(\cdot, t)|^2\\
   &\leq 4 \left\{ \int_{-1}^t\int_M |f|^2 + \int_M \left\{|u(\cdot,-1)|^2 +|\nabla u(\cdot, -1)|^2 \right\} \right\}.
 \end{split}  
 \end{align}
 It follows that
 \begin{align}
     \int_{-1}^t\int_M \{ |\dot{u}|^2 +|\Delta u|^2\} 
     \leq \frac{128}{31} \left\{ \int_{-1}^t\int_M |f|^2 + \int_M \left\{|u(\cdot,-1)|^2 +|\nabla u(\cdot, -1)|^2 \right\} \right\},
 \label{eqn:HA31_4}    
 \end{align}
 which implies (\ref{eqn:HA28_1}). 
 Inequality (\ref{eqn:HA31_5}) also implies
 \begin{align}
   \int_M |\nabla u(\cdot, t)|^2
   \leq 4 \left\{ \int_{-1}^t\int_M |f|^2 + \int_M \left\{|u(\cdot,-1)|^2 +|\nabla u(\cdot, -1)|^2 \right\} \right\}.
 \label{eqn:HA31_3}  
 \end{align}
 Recall that we have proved in (\ref{eqn:HA31_1}) that
 \begin{align*}
      \int_M |u(\cdot,t)|^2
    \leq \frac{16}{3} \int_{-1}^0\int_M |f|^2 + \frac{8}{3}\int_M |u(\cdot,-1)|^2. 
 \end{align*} 
 Adding it to (\ref{eqn:HA31_3}), we arrive at (\ref{eqn:HA28_1A}). \\
 
\textit{Step 3. We have}
\begin{align}
 \label{eqn:HA29_2}   
 \begin{split}
   \int_{-1}^0 \int_M |\nabla \nabla u|^2  
   \leq 8 \left\{ \int_{-1}^0 \int_M |f|^2 + \int_M \left\{|u(\cdot,-1)|^2 +|\nabla u(\cdot, -1)|^2 \right\} \right\}.   
 \end{split}  
 \end{align}

 Since $u=u_{ij}dx^idx^j$, the direct calculation shows the following Bochner-type identity:
 \begin{align*}
     \frac12\Delta |\nabla u|^2&=|\nabla \nabla u|^2 + \langle \nabla \Delta u, \nabla u\rangle\\ 
    &\quad +(2R_{klip}u_{pj,l} +2R_{kljp}u_{ip,l}+R_{kp}u_{ij,p})u_{ij,k}
    + 2(R_{kp,i}-R_{ki,p})u_{pj}u_{ij,k}.  
 \end{align*}
 Integrating this on $M$ implies
 \begin{align*}
        \int_M |\nabla \nabla u|^2
         &\leq \int_M \left\{|\Delta u|^2 + \frac{1}{16}|\nabla u|^2 + \frac{1}{16}|u||\nabla u| \right\}
          \leq \int_M \left\{|\Delta u|^2 + \frac{3}{32}|\nabla u|^2 + \frac{1}{32}|u|^2 \right\}\\
         &\leq \int_M |\Delta u|^2 + \frac{1}{10} \int_M \left\{ |u|^2+|\nabla u|^2\right\}. 
 \end{align*}
 Plugging (\ref{eqn:HA27_5}) and (\ref{eqn:HA28_1}) into the above inequality, we obtain (\ref{eqn:HA29_2}). \\

\textit{Step 4.  Set up the induction relationship.}

 Choose a positive integer $N$ such that $\Lambda_0 \xi^{-2} \in [N-1, N]$. Let $t_k=(-1+\frac{k}{N})\Lambda_0 \xi^{-2}$.  Then $t_0=-\Lambda_0 \xi^{-2}$ and $t_N=0$.

From (\ref{eqn:HA28_1A}), we obtain an induction estimate. 
\begin{align*}
\begin{split}
    &\quad \int_M \left\{ |u(\cdot,t_{k+1})|^2 + |\nabla u(\cdot,t_{k+1})|^2\right\}\\
    &\leq 16 \left\{\int_{t_k}^{t_{k+1}} \int_M |f|^2 +\int_M \left\{|u(\cdot,t_k)|^2+ |\nabla u(\cdot,t_k)|^2 \right\}\right\}\\
    &\leq 16 \int_{t_k}^{t_{k+1}} \int_M |f|^2 + 16 \left\{16 \left\{\int_{t_{k-1}}^{t_{k}} \int_M |f|^2 
    +\int_M \left\{|u(\cdot,t_{k-1})|^2+ |\nabla u(\cdot,t_{k-1})|^2 \right\}\right\} \right\}\\
    &\leq 16^2 \left\{ \int_{t_{k-1}}^{t_{k+1}} \int_M |f|^2 
      +\int_M \left\{|u(\cdot,t_{k-1})|^2+ |\nabla u(\cdot,t_{k-1})|^2 \right\}\right\}\\
    &\leq 16^{k+1} \left\{ \int_{t_0}^{t_{k+1}} \int_M |f|^2 +\int_M \left\{|u(\cdot,t_0)|^2
    + |\nabla u(\cdot,t_0)|^2 \right\}\right\}. 
\end{split}    
\end{align*} 
Note that $u(\cdot, t_0) \equiv 0$.  It follows that
\begin{align}
\label{eqn:HA30_7}
\begin{split}
    \int_M \left\{ |u(\cdot,t_{k+1})|^2 + |\nabla u(\cdot,t_{k+1})|^2\right\}
    \leq 16^{k+1}\int_{t_0}^{t_{k+1}} \int_M |f|^2.  
\end{split}    
\end{align} 
As (\ref{eqn:HA28_1A}) holds for each $t \in [-1,0]$, the above deduction actually implies 
\begin{align*}
    \int_M \left\{|u(\cdot,t)|^2+ |\nabla u(\cdot,t)|^2 \right\}
    \leq 16^{k+1}\int_{t_0}^{t_{k+1}} \int_M |f|^2, \quad \forall\; t \in [t_k, t_{k+1}].  
\end{align*}
As $|t_{k+1}-t_k| \leq 1$, the integration of the above inequality yields
\begin{align}
\label{eqn:HA30_1}
   \int_{t_k}^{t_{k+1}} \int_M \left\{|u(\cdot,t)|^2+ |\nabla u(\cdot,t)|^2 \right\}
   \leq 16^{k+1}\int_{t_0}^{t_{k+1}} \int_M |f|^2
   \leq 16^{k+1}\int_{t_0}^{t_{N}} \int_M |f|^2. 
\end{align}
Summing (\ref{eqn:HA30_1}) over $k \in \{0, 1, \cdots, N-1\}$, we obtain
\begin{align*}
     \int_{t_0}^{t_{N}} \int_M \left\{|u(\cdot,t)|^2+ |\nabla u(\cdot,t)|^2 \right\}
     \leq \left\{\sum_{k=0}^{N-1} 16^{k+1}\right\} \cdot \int_{t_0}^{t_{N}} \int_M |f|^2
     \leq 16^{N+1} \int_{t_0}^{t_{N}} \int_M |f|^2. 
\end{align*}

It follows from (\ref{eqn:HA29_2}) that
\begin{align*} 
\begin{split}
 \int_{t_k}^{t_{k+1}} \int_M  |\nabla \nabla u|^2 
    &\leq 16 \left\{\int_{t_k}^{t_{k+1}}\int_M |f|^2 + \int_M \left\{|u(\cdot,t_k)|^2 +|\nabla u(\cdot, t_k)|^2 \right\} \right\}.
\end{split}    
\end{align*}
By the same as the deduction in (\ref{eqn:HA30_7}), we obtain 
\begin{align*}
    \int_{t_k}^{t_{k+1}} \int_M  |\nabla \nabla u|^2 
    \leq 
    16^{k+1}\int_{t_0}^{t_{k+1}} \int_M |f|^2, 
\end{align*}
whose sum yields
\begin{align} 
\label{eqn:HA30_3}
\begin{split}
     \int_{t_0}^{t_{N}} \int_M |\nabla \nabla u|^2 
     \leq 16^{N+1} \cdot \int_{t_0}^{t_{N}}\int_M |f|^2. 
\end{split}    
\end{align} 
Recall that $N \leq 1+\Lambda_0\xi^{-2}$, we have
\begin{align*}
    16^{N+1} \leq 16^{\Lambda_0 \xi^{-2}+2} 
    =:C(n, \Lambda_0). 
\end{align*}
Plugging it into (\ref{eqn:HA30_3}) and noting that $t_0=-\Lambda_0 \xi^{-2}$ and $t_N=0$, we arrive at (\ref{eqn:HA25_5}). 
The proof of Proposition~\ref{prn:HA25_6} is complete. 
\end{proof}

\begin{proposition}
\label{prn:HA25_9}
Let $\{(M,g(t)), -1 \leq t \leq 0\}$ be an evolving manifold that satisfies (\ref{eqn:HA02_1}). 
Let $\Psi$ be the heat kernel of $\Delta_L$ and $\mathcal{K}(x,t;y,s) := \nabla_x \nabla_y \Psi(x,t;y,s)$ and
$\mathcal{T}$ be the convolution operator with $\mathcal{K}$. Then $\mathcal{T}$ is a Calder\'{o}n-Zygmund kernel in the sense of Definition~\ref{def2.1}.
\end{proposition}

\begin{proof}

   We shall only deal with the Ricci flow case. The static metric case is easier and left to the interested readers. 

    The heat kernel estimate is guaranteed by Theorem~\ref{thm3.3}.
    Define
    \begin{align}
        u(x,t) :=\int_{-1}^t \int_M \Psi(x,t;y,s) \nabla_y^* f(y,s)dyds
         =\int_{-1}^t \int_M \nabla_y \Psi(x,t;y,s) f(y,s)dyds.  
    \end{align}
    It suffices to show the $L^2$-estimate:
    	\begin{align}
        \label{eqn:HA25_8}
			\begin{split}
				\int_{-1}^0\int_{M}|\nabla u|^2 d\mu dt \leq C(n, \Lambda_0) \cdot \int_{-1}^0\int_{M}|f|^2  d\mu dt.
			\end{split}
		\end{align}
     Notice that
     \begin{align*}
	(\partial_t -\Delta_L) u=\nabla^*f, \quad u(\cdot, -1) \equiv 0, \quad \int_{-1}^0 \int_M |u|^2 d\mu dt<\infty. 
     \end{align*}
 Using the same method as in the proof of Proposition~\ref{prn:HA25_6}, we may use parabolic rescaling and assume $|Rm| \leq \xi^2$. Then we need to show
 \begin{align}
 \label{eqn:HB02_2}
 \begin{split}
	\int_{-\Lambda_0 \xi^{-2}}^0\int_{M}|\nabla u|^2 d\mu dt \leq C(n, \Lambda_0) \cdot \int_{-\Lambda_0 \xi^{-2}}^0\int_{M}|f|^2  d\mu dt.
 \end{split}    
 \end{align}
 under the assumption 
 \begin{align}
 \label{eqn:HB02_3}
	(\partial_t -\Delta_L) u=\nabla^*f, \quad u(\cdot, -\Lambda_0 \xi^{-2}) \equiv 0, \quad \int_{-\Lambda_0 \xi^{-2}}^0 \int_M |u|^2 d\mu dt<\infty. 
 \end{align}
 We divide the proof into two steps. \\

 \textit{Step 1. We have}
 \begin{align}
   \label{eqn:HB02_9}
       \int_M |u|^2(\cdot, t) 
        \leq 2 \left\{ \int_{-1}^0\int_M  |f|^2 + \int_M |u|^2(\cdot, -1)\right\}, \quad \forall \; t \in [-1,0],  
 \end{align}
 \textit{and} 
 \begin{align}
    \label{eqn:HB02_4}
       \begin{split}
     \int_{-1}^0\int_M  |\nabla u|^2
     \leq 2 \left\{ \int_{-1}^0\int_M  |f|^2 + \int_M |u|^2(\cdot, -1)\right\}. 
       \end{split}
 \end{align}

 On the one hand, using the definition of the Lichnerowicz Laplacian and (\ref{eqn:HB02_1}), 
we have
   \begin{equation}
   \label{eqn:HA29_6}
       \begin{split}
        &\quad \int_{-1}^t\int_M \langle \nabla^* f,u\rangle\\ 
        &=\int_{-1}^t\int_M \langle (\partial_t -\Delta)u,  u \rangle +\int_{-1}^t\int_M (-2R_{iklj}u_{kl}+R_{ik}u_{kj}+R_{jk}u_{ki})u_{ij}\\
        &=\int_M |u|^2(\cdot, t)-\int_M |u|^2(\cdot, -1)\\
        &\quad + \int_{-1}^t\int_M \left\{ R|u|^2 + |\nabla u|^2
        -(2R_{iklj}u_{kl}+R_{ik}u_{kj}+R_{jk}u_{ki})u_{ij}\right\}\\
        &\geq \int_M |u|^2(\cdot, t)-\int_M |u|^2(\cdot, -1)+\int_{-1}^t \int_M \left\{ |\nabla u|^2 -\frac14 |u|^2\right\}.
       \end{split}
   \end{equation}
   On the other hand, we have
   \begin{equation}
   \label{eqn:HA28_3}
       \begin{split}
        \int_{-1}^t\int_M \langle \nabla^* f,u\rangle  =\int_{-1}^t\int_M \langle f,\nabla u\rangle 
        \leq \frac14\int_{-1}^t \int_M  |\nabla u|^2+\int_{-1}^t\int_M |f|^2.
       \end{split}
   \end{equation}
  Combining (\ref{eqn:HA29_6}) and (\ref{eqn:HA28_3}), we have 
  \begin{equation}
  \label{eqn:HA29_7}
       \begin{split}
      \int_M |u|^2(\cdot, t)+ 
      \frac34\int_{-1}^t\int_M |\nabla u|^2
      \leq \int_M |u|^2(\cdot, -1)
       +\int_{-1}^t\int_M  |f|^2 +\frac14 \int_{-1}^t\int_M |u|^2. 
       \end{split}
   \end{equation}
   Forgetting the first term on the left hand side and setting $t=0$, we have
    \begin{equation}
    \label{eqn:HA29_10}
       \begin{split}
     \int_{-1}^0\int_M  |\nabla u|^2
     \leq 
     \frac43\int_M |u|^2(\cdot, -1)+
     \frac43 \int_{-1}^0\int_M  |f|^2 +\frac13 \int_{-1}^0\int_M |u|^2.
       \end{split}
   \end{equation}
   By ignoring the second term on the left hand side of (\ref{eqn:HA29_7}) and noting that $t \leq 0$, we have
   \begin{align}
   \label{eqn:HA29_8}
       \int_M |u|^2(\cdot, t) 
       \leq \int_M |u|^2(\cdot, -1)
       +\int_{-1}^0\int_M  |f|^2 +\frac14 \int_{-1}^0\int_M |u|^2. 
   \end{align}
   Integrating (\ref{eqn:HA29_8}) over $[-1,0]$, we obtain 
   \begin{align}
   \label{eqn:HA29_9}
        \int_{-1}^0 \int_M |u|^2 
        \leq \frac43 \left\{ \int_{-1}^0\int_M  |f|^2 + \int_M |u|^2(\cdot, -1)\right\}. 
   \end{align}
   Plugging (\ref{eqn:HA29_9}) into (\ref{eqn:HA29_8}), we have
   \begin{align*}
       \int_M |u|^2(\cdot, t)
       \leq \frac43 \left\{ \int_{-1}^0\int_M  |f|^2 + \int_M |u|^2(\cdot, -1)\right\}, \quad \forall \; t\in [-1, 0].  
   \end{align*}
   Putting (\ref{eqn:HA29_9}) into (\ref{eqn:HA29_10}), we arrive at (\ref{eqn:HB02_4}).\\

   \textit{Step 2. Induction.}

 Similarly as in the proof of Proposition~\ref{prn:HA25_6}, we choose a positive integer $N$ such that $\Lambda_0 \xi^{-2} \in [N-1, N]$. 
   Let $t_k=(-1+\frac{k}{N})\Lambda_0 \xi^{-2}$.  Then $t_0=-\Lambda_0 \xi^{-2}$ and $t_N=0$.

    By time shifting,  (\ref{eqn:HB02_9}) implies that
 \begin{align*}
      \int_M |u|^2(\cdot, t)
      \leq  2 \left\{ \int_{t_k}^{t_{k+1}}\int_M  |f|^2 + \int_M |u|^2(\cdot, t_k)\right\}
 \end{align*}
 for every $t\in [t_k, t_{k+1}]$.  In particular, we have
 \begin{align*}
      \int_M |u|^2(\cdot, t_{k+1})
      \leq  2 \left\{ \int_{t_k}^{t_{k+1}}\int_M  |f|^2 + \int_M |u|^2(\cdot, t_k)\right\}.
 \end{align*}
 Similarly to the deduction from (\ref{eqn:HA28_1A}) to (\ref{eqn:HA30_7}), we have
 \begin{align}
\label{eqn:HB02_6}
\begin{split}
    \int_M  |u(\cdot,t_{k+1})|^2 
    \leq 2^{k+1}\int_{t_0}^{t_{k+1}} \int_M |f|^2,  
\end{split}    
\end{align} 
and 
\begin{align*}
\begin{split}
    \int_M  |u(\cdot,t)|^2 
    \leq 2^{k+1}\int_{t_0}^{t_{k+1}} \int_M |f|^2,  \quad \forall \; t \in [t_k, t_{k+1}]. 
\end{split}    
\end{align*}    
The integration of the above inequality implies
\begin{align*}
\begin{split}
    \int_{t_k}^{t_{k+1}} \int_M  |u(\cdot,t)|^2 
    \leq 2^{k+1}\int_{t_0}^{t_{k+1}} \int_M |f|^2,  \quad \forall \; t \in [t_k, t_{k+1}], 
\end{split}    
\end{align*}    
whose summation over $k$ yields
\begin{align}
     \int_{t_0}^{t_{N}} \int_M  |u(\cdot,t)|^2 
    \leq 2^{N+1}\int_{t_0}^{t_{N}} \int_M |f|^2. 
\end{align}

 By time shifting again, it follows from (\ref{eqn:HB02_4}) that 
 \begin{align*}
       \begin{split}
     \int_{t_k}^{t_{k+1}}\int_M  |\nabla u|^2
     \leq 2 \left\{  \int_{t_k}^{t_{k+1}}\int_M  |f|^2 + \int_M |u|^2(\cdot, t_k)\right\}. 
       \end{split}
   \end{align*}
 Plugging (\ref{eqn:HB02_6}) into the above inequality implies 
 \begin{align}
       \begin{split}
     \int_{t_k}^{t_{k+1}}\int_M  |\nabla u|^2
     \leq 2^{k+1}\int_{t_0}^{t_{k+1}} \int_M |f|^2.
       \end{split}
   \end{align}
 Taking the sum of the above inequality yields
   \begin{align}
   \label{eqn:HB02_8}
       \int_{t_0}^{t_{N}}\int_M  |\nabla u|^2
       \leq 2^{N+1} \int_{t_0}^{t_N} \int_M |f|^2. 
   \end{align}
   Recall that $t_0=-\Lambda_0 \xi^{-2}$ and $t_N=0$. It is clear that (\ref{eqn:HB02_2}) follows from (\ref{eqn:HB02_8}) by setting
   \begin{align*}
       C(n, \Lambda_0) := 2^{\Lambda_0 \xi^{-2}+2}. 
   \end{align*}
\end{proof}

Now we are ready to prove the main theorems.

\begin{proof}[Proof of Theorem~\ref{thm:HA02_1}:]
We choose $\eta$ as before. That is, $\eta$ is a non-decreasing function defined on $[-1, 0]$ such that $\eta \equiv 1$ on $[-\frac12, 0]$ and $\eta \equiv 0$ on $[-1, -\frac34]$. Furthermore, $0 \leq \dot{\eta} \leq 10$. 
Setting $u_1=\eta u$ and $u_2=(1-\eta)u$.  Then we have
\begin{align*}
    u=u_1+u_2, 
\end{align*}
and the function $u_1$ satisfies 
\begin{align*}
    (\partial_t-\Delta)u_1= (\partial_t-\Delta) (\eta u)
    =\dot{\eta} u +\eta (\partial_t-\Delta)u=\dot{\eta} u + \eta f. 
\end{align*}
 Note that $u_1(\cdot,-1)=0$. Thus, for each $(x,t) \in \mathcal{M}'=M \times [-\frac14, 0]$, we have
  \begin{align*}
        \mathrm{Hess}(u)(x,t)&=\mathrm{Hess}(u_1)(x,t)=\int_{-1}^t\int_M \nabla_x \nabla_x H(x,t;y,s) (\dot{\eta} u +\eta f) dyds
  \end{align*}
  Let $\mathcal{K}(x,t;y,s) :=\nabla_x \nabla_x H(x,t;y,s)$ and $T$ be the convolution operator with $\mathcal{K}$.  
  By Proposition~\ref{prn:HA25_3}, we know $\mathcal{T}$ is a Calder\'{o}n-Zygmund operator in the sense of Definition~\ref{def2.1}. 
  Then we have
  \begin{align}
  \begin{split}
      \|\mathrm{Hess}(u)\|_{L^p(\mathcal{M}')}
      &=\|\mathrm{Hess}(u_1)\|_{L^p(\mathcal{M}')}
      =\|T(\dot{\eta} u +\eta f)\|_{L^p(\mathcal{M}')}\\
      &\leq C \|\dot{\eta} u +\eta f\|_{L^p(\mathcal{M})}
      \leq C \{\|u\|_{L^p(\mathcal{M})} + \|f\|_{L^p(\mathcal{M})}\}. 
  \end{split}    
  \label{eqn:HA14_6}
  \end{align}
  By the evolution equation of $u$, we have
  \begin{align}
      |\dot{u}| =|\Delta u +f| \leq |\Delta u| +|f|
      \leq C(n) |\mathrm{Hess} u| + |f|.
  \label{eqn:HA14_7}    
  \end{align}
  Consequently,  the combination of (\ref{eqn:HA14_6}) and (\ref{eqn:HA14_7}) yields 
  \begin{align}
     \|\dot{u}\|_{L^p(\mathcal{M}')}
     \leq C(n) \|\mathrm{Hess}(u)\|_{L^p(\mathcal{M}')} + \|f\|_{L^p(\mathcal{M}')}
     \leq C \{\|u\|_{L^p(\mathcal{M})} + \|f\|_{L^p(\mathcal{M})}\}. 
  \label{eqn:HA14_8}   
  \end{align}
  Therefore, the inequality (\ref{eqn:SJ27_1}) follows directly from (\ref{eqn:HA14_6}) and (\ref{eqn:HA14_8}). 
  The proof of Theorem~\ref{thm:HA02_1} is complete. 
  \end{proof}

 \begin{proof}[Proof of Corollary~\ref{cly:HA02_2}:]
 The Corollary follows by taking  $u(x,t)=\eta(t)u(x)$ and applying Theorem~\ref{thm:HA02_1}.
 \end{proof}

\begin{proof}[Proof of Theorem~\ref{thm:HA02_3}:] 
 The proof is almost the same as that of Theorem \ref{thm:HA02_1}.
 By standard cut-off technique, it suffices to prove (\ref{eqn:SJ27_1A}) for (0,2)-tensor valued smooth section $u$ satisfying $u(\cdot, -1) \equiv 0$. 
 Note that
 \begin{align*}
     \nabla_x \nabla_x u
     =\int_{-1}^t\int_M \nabla_x \nabla_x \Psi(x,t;y,s) f(y,s) dyds. 
 \end{align*}
 By Proposition~\ref{prn:HA25_6}, we know $\mathcal{K}(x,t;y,s) :=\nabla_x \nabla_x \Psi(x,t;y,s)$ is a Calder\'{o}n-Zygmund kernel in the sense of Definition~\ref{def2.1}. It follows directly from Theorem~\ref{thm:HA12_1} that 
  \begin{align*}
      \|\mathrm{Hess}(u)\|_{L^p(\mathcal{M}')}
      \leq C \{\|u\|_{L^p(\mathcal{M})} + \|f\|_{L^p(\mathcal{M})}\}.   
  \end{align*}
 Since 
 \begin{align*}
    \dot{u}=\Delta_L u + f=\Delta u + Rm*u +f,  
 \end{align*}
 it is clear that
 \begin{align*}
     |\dot{u}| \leq |\Delta u| +C_n |Rm| |u| +|f| 
     \leq C (|\mathrm{Hess} u| + |f|),
 \end{align*}
 which yields that
 \begin{align*}
     \|\dot{u}\|_{L^p(\mathcal{M}')}
     \leq C \left\{ \|\mathrm{Hess}(u)\|_{L^p(\mathcal{M}')}  +\|f\|_{L^p(\mathcal{M}')} \right\}
     \leq C \left\{ \|u\|_{L^p(\mathcal{M})}  +\|f\|_{L^p(\mathcal{M})} \right\}. 
 \end{align*}
 Combing the previous estimate of $\|\dot{u}\|_{L^p(\mathcal{M}')}$ and
 $\|\mathrm{Hess}(u)\|_{L^p(\mathcal{M}')}$, we obtain (\ref{eqn:SJ27_1A}). 
\end{proof}

\begin{proof}[Proof of Theorem~\ref{thm:HA02_4}:]
If $u(\cdot, -1)=0$, we have
\begin{equation*}
    \begin{split}
          \nabla u(x,t)
         =\int_{-1}^t\int_M \nabla_x\Psi(x,y,t-s)\nabla_y^*f(y,s)
         =\int_{-1}^t\int_M \nabla_x\nabla_y\Psi(x,y,t-s)  f(y,s). 
    \end{split}
\end{equation*}
Let $\mathcal{K}(x,t;y,s) :=\nabla_x\nabla_y\Psi(x,t;y,s)$. 
By Proposition~\ref{prn:HA25_9},  we know the convolution with $\mathcal{K}$ provides a Calder\'{o}n-Zygmund operator in the sense of Definition~\ref{def2.1}.  Then the remaining argument follows verbatim as that in the proof of Theorem~\ref{thm:HA02_3}. 
\end{proof}

\end{document}